\documentclass[11pt,leqno,a4paper,oneside]{amsart}

\usepackage{amssymb,amsthm,enumerate}
\usepackage{wasysym}
\usepackage{array}
\usepackage{hyperref}                         
\usepackage[all]{xy}

\usepackage{tikz}                                 
\usetikzlibrary{angles}
\usetikzlibrary{decorations.markings,calc}
\usetikzlibrary{cd}

\usepackage{longtable}
\usepackage{mathrsfs}
\usepackage{hhline}
\usepackage{multicol}
\usepackage{multirow}
\usepackage{blkarray} 
\usepackage{lscape}
\usepackage{rotating}

\input xy
\xyoption{all}

\entrymodifiers={!!<0pt,0.7ex>+}  

\hypersetup{pdftex,                           
bookmarks=true,
pdffitwindow=true,
colorlinks=true,
citecolor=black,
filecolor=black,
linkcolor=black,
urlcolor=black,
hypertexnames=true}

\newcommand{\IC}{{\mathbb{C}}}
\newcommand{\IF}{{\mathbb{F}}}

\newcommand{\IZ}{{\mathbb{Z}}}

\newcommand{\fA}{{\mathfrak{A}}}   
\newcommand{\fp}{{\mathfrak{p}}}     

\newcommand{\cO}{{\mathcal{O}}}

\newcommand{\bB}{{\mathbf{B}}}
\newcommand{\bb}{{\mathbf{b}}}
\newcommand{\bc}{{\mathbf{c}}}

\DeclareMathOperator{\Hom}{Hom}             
\DeclareMathOperator{\Syl}{Syl}                  
\DeclareMathOperator{\Res}{Res}               
\DeclareMathOperator{\Ind}{Ind}                  
\DeclareMathOperator{\Inf}{Inf}                    
\DeclareMathOperator{\tr}{tr}			

\DeclareMathOperator{\Irr}{Irr}

\DeclareMathOperator{\Lin}{Lin}
\DeclareMathOperator{\soc}{soc}
\DeclareMathOperator{\head}{head}
\DeclareMathOperator{\rad}{rad}

\DeclareMathOperator{\TS}{TS}

\DeclareMathOperator{\St}{St}
\DeclareMathOperator{\Sc}{Sc}
\DeclareMathOperator{\Br}{Br}
\DeclareMathOperator{\Triv}{Triv}

\DeclareMathOperator{\GL}{\operatorname{GL}}
\DeclareMathOperator{\SL}{\operatorname{SL}}
\DeclareMathOperator{\PSL}{\operatorname{PSL}}

\newcommand{\indhg}[2]{\!\uparrow_{#1}^{#2}}   
\newcommand{\resgh}[2]{\!\downarrow^{#1}_{#2}}   

\let\lra=\longrightarrow

\let\wh=\widehat

\newtheorem{thm}{Theorem}[section]
\newtheorem{lem}[thm]{Lemma}

\newtheorem{cor}[thm]{Corollary}
\newtheorem{prop}[thm]{Proposition}

\newtheorem{nota}[thm]{Notation}

\newtheoremstyle{defnew}{4ex}{}{}{}{\bf}{}{.5em}{}
\theoremstyle{defnew}

\newtheorem{notadefn}[thm]{Notation-Definition}
\newtheorem{conv}[thm]{Convention}

\theoremstyle{remark}
\newtheorem{rem}[thm]{Remark}

\theoremstyle{theorem}

\begin{document}


	\title[Trivial source character tables of $\SL_2(q)$]{Trivial source character tables of $\SL_2(q)$}
	
	\date{\today}

	\author{{Bernhard B\"ohmler, Niamh Farrell and Caroline Lassueur}}
	\address{{\sc Bernhard B\"ohmler},  FB Mathematik, TU Kaiserslautern, Postfach 3049, 67653 Kaiserslautern, Germany}
	\email{boehmler@mathematik.uni-kl.de}
	\address{{\sc Niamh Farrell}, Institut f\"ur Algebra, Zahlentheorie und Diskrete Mathematik, Leibniz Universit\"at Hannover, Welfengarten~1, 30176 Hannover, Germany.}
	\email{farrell@math.uni-hannover.de}
	\address{{\sc Caroline Lassueur},  FB Mathematik, TU Kaiserslautern, Postfach 3049, 67653 Kaiserslautern, Germany}
	\email{lassueur@mathematik.uni-kl.de}
	
	
	\makeatletter
	\@namedef{subjclassname@2020}{%
		\textup{2020} Mathematics Subject Classification}
	\makeatother
	

	\begin{abstract}
		We compute the trivial source character tables (also called species tables of the trivial source ring) of the infinite family of  finite groups $\SL_{2}(q)$ over a large enough field $k$ of positive characteristic~${\ell}$ via character-theoretical methods in the cases in which $q$ is odd,  $\ell \mid (q\pm1)$ when~$\ell$ is odd, and $q\equiv \pm 3\pmod{8}$ when $\ell=2$. 
	\end{abstract}

	\thanks{
		The authors gratefully acknowledge financial support by DFG SFB -TRR 195 `Symbolic Tools in Mathematics and their Application'. The present article started as  part of the Project A18 thereof. 
	}
	
	\maketitle

	
	\pagestyle{myheadings}
	\markboth{B. B\"ohmler, N. Farrell, C. Lassueur}{Trivial source character tables of $\SL_2(q)$}

	\vspace{6mm}
	\section{Introduction}

	In the representation theory of finite groups in characteristic zero it is customary to work with the ordinary character table, from which information about the representations of the group can be recovered  efficiently.  In positive characteristic, representation rings are a convenient way of organising information about direct sums and tensor products of modules, and they give us a unified way to view ordinary and Brauer character tables.   
	Given a finite group $G$ and a field $F$, the Green ring $a(F G)$ of $F G$ is defined to be the free abelian group on the set of  isomorphism classes $[M]$ of indecomposable $F G$-modules, with addition given by taking direct sums and multiplication induced by the tensor product over $F$.  Then $A(F G):=\IC\otimes_{\IZ}a(FG)$  is  a commutative and associative $\IC$-algebra. When $F=\IC$, $a(\IC G)$ is the Grothendieck ring of~$\IC G$ and every ring homomorphism $a(\IC G)\lra \IC$ is given by a trace map
	$t_{g}:[M]\mapsto \tr(g,M)$ with~$g\in G$. After tensoring with $\IC$, the sum of these maps over a set of representatives $\text{cc}(G)$ of the conjugacy classes of $G$ is an isomorphism $\sum_{g \in \text{cc}(G)} t_{g}: A(\IC G)\overset{\cong}{\lra} \bigoplus_{\text{cc}(G)}\IC\,,$ and hence  $A(\IC G)$ is semisimple.  
	In positive characteristic, on the other hand, if $F$ is a large enough field of characteristic $\ell>0$ then the Grothendieck ring is $R(FG):=a(FG)/a_{0}(FG,1)$, where $a_{0}(FG,1)$ is the ideal of $a(FG)$ spanned by the difference elements $[M_{2}]-[M_{1}]-[M_{3}]$ where~$0\rightarrow M_{1}\rightarrow M_{2}\rightarrow M_{3}\rightarrow 0$ is a short exact sequence of $FG$-modules.  In this case every ring homomorphism $R(FG)\lra \IC$ is given by a map $t_{g}: R(FG)\lra \IC$ where $g\in G$ is an $\ell'$-element  and for a given $F G$-module $M$, $t_{g}(M)$ is the sum of the lifts to $\IC$ of the eigenvalues of $g$ on the restricted module $\Res^{G}_{\langle g\rangle}(M)$. Once again, after tensoring with $\IC$, the sum of these maps over a set of representatives $\text{cc}(G)_{\ell'}$ of the ${\ell'}$-conjugacy classes of $G$ yields an isomorphism  \smallskip $\sum_{g \in \text{cc}(G)_{\ell'}} t_{g}: \IC\otimes_{\IZ}R(FG)\overset{\cong}{\lra} \bigoplus_{\text{cc}(G)_{\ell'}}\IC.$ 
		\par
	In \cite{BP} Benson and Parker generalised such constructions and defined a \emph{species} of any subalgebra, ideal or quotient $A$ of $A(FG)$, to be an algebra homomorphism $A\lra\IC$. 
	Evaluating the species of $A(\IC G)$ at the simple $\IC G$-modules yields the species table of $A(\IC G)$, which is in fact just the ordinary character table of $G$. Similarly, if $F$ is large enough and of characteristic $\ell>0$ then 
	the species table of $\IC\otimes_{\IZ}R(FG)$, calculated by evaluating the species of $\IC\otimes_{\IZ}R(FG)$ at the simple $FG$-modules,  is just the Brauer character table of $G$. 
	Benson and Parker \cite{BP} went on to prove that in the latter, positive characteristic case, many of the properties of representation rings, species and $\ell$-blocks are governed by the \emph{trivial source ring} $a(FG,\Triv)$, which is defined to be the subring of $a(FG)$ generated by the (finite) set of all isomorphism classes of indecomposable trivial source $FG$-modules.  
	We will see in Section~\ref{sec:prelim} that $A(FG,\Triv):=\IC\otimes_{\IZ} a(FG,\Triv)$ is semisimple.  
	Evaluating the species of $A(FG,\Triv)$ at the  indecomposable trivial source $FG$-modules yields  a square matrix  called the \emph{trivial source character table} (or \emph{species table of the trivial source ring}), denoted by $\Triv_{\ell}(G)$. This table provides us with  information about the character values of all trivial source $FG$-modules and their Brauer quotients at $\ell'$-conjugacy classes. For $\ell$-groups,  $\Triv_{\ell}(G)$ is just the table of marks. For all other groups, however, unlike ordinary and Brauer character tables, only a very small number of trivial source character tables have been published in the literature: \cite[Appendix]{BensonBookOld} gives $\Triv_{2}(\fA_{5})$ and \cite[Example~4.10.12]{LP} gives $\Triv_{3}(M_{11})$.
	\par
	The aim of this article is to calculate the \emph{generic} trivial source character tables for the infinite family of finite groups of Lie type $\SL_{2}(q)$.
	Our main results are obtained in Theorems~\ref{thm:l|q-1}, ~\ref{thm:l|q+1} and ~\ref{thm:sl2ql=2} where  we calculate  $\Triv_{\ell}(\SL_{2}(q))$ using a  block- and character-theoretic approach in the following cases:
	\begin{enumerate}[\,\,(1)]
		\item $q$ and $\ell$ are odd and $\ell\mid (q-1)$, in which case the Sylow $\ell$-subgroups are cyclic;
		\item $q$ and $\ell$ are odd and $\ell\mid (q+1)$, in which case the Sylow $\ell$-subgroups are cyclic;
		\item $q\equiv \pm 3\pmod{8}$ and $\ell=2$,  in which case the Sylow $\ell$-subgroups are quaternion groups of order~$8$. 
	\end{enumerate}
	In all cases, we make extensive use of the known generic character table and block distribution of $\SL_{2}(q)$ as described in \cite{BonBook}. 
	When $\ell$ is odd our method and results strongly rely on the recent classification of the trivial source modules in blocks with cyclic defect groups by Hiss and the 3rd author in \cite{HL20}. 
	When $\ell=2$ and $q\equiv\pm 3\pmod{8}$ we first compute the trivial source character table $\Triv_{2}(\PSL_{2}(q))$, where the situation is  easier to deal with because the Sylow $2$-subgroups of $\PSL_{2}(q)$ ($q\equiv\pm 3\pmod{8}$) are Klein-four groups,  and $2$-blocks with defect groups~$C_{2}\times C_{2}$ have been classified up to source-algebra equivalence in \cite{CEKL}. This allows us to reduce our calculations to the cases of $\PSL_{2}(3)\cong\fA_{4}$ and $\PSL_{2}(5)\cong\fA_{5}$. Then $\Triv_{2}(\SL_{2}(q))$  is  obtained to a large extent via inflation from $\PSL_{2}(q)$ because the centre of $\SL_{2}(q)$ is contained in all the non-trivial $2$-subgroups. Indeed, inflation preserves trivial source modules, however, it must be taken into account that vertices change as $\ell=2=|Z(\SL_{2}(q))|$.

	The paper is organised as follows.
	In Section~\ref{sec:prelim} we summarise the background material we will use throughout the paper. In particular, we  introduce the trivial source character tables of a finite group formally following the approach of Lux and Pahlings in~\cite[Section 4.10]{LP}. In addition, we collect useful results about trivial source modules, the species of the trivial source ring, and the group $\SL_{2}(q)$ which we will use throughout our calculations of the tables. 
	In Section~\ref{sec:lmidq-1} and Section~\ref{sec:lmidq+1} we calculate $\Triv_{\ell}(\SL_{2}(q))$ in the cases in which $\ell$ is odd and divides either $q-1$ or $q+1$. 
	Next, we consider the characteristic $\ell=2$ and assume that  $q\equiv \pm 3\pmod{8}$. In Section~\ref{sec:PSLl=2} we compute $\Triv_{2}(\PSL_{2}(q))$. From this we can deduce a large part of $\Triv_{2}(\SL_{2}(q))$ and so in Section~\ref{SL2l=2} it only remains to describe the $2$-projective characters of $\SL_{2}(q)$.
	\par
	Further combinations for $\ell$ and $q$ will be studied in forthcoming papers.

	\vspace{8mm}
\section{Notation and background results}\label{sec:prelim}

\vspace{2mm}

\subsection{General notation}
Throughout, unless otherwise stated, we adopt the notation and conventions given below.
We let~$\ell$ be a prime number and  $G$ denote a finite group of order divisible by~$\ell$. We let $(K,\cO,k)$ be an $\ell$-modular system, where  $\cO$ denotes a complete discrete valuation ring of characteristic zero with unique maximal ideal $\frak{p}:=J(\cO)$, algebraically closed residue field~$k=\cO/\fp$ of characteristic~$\ell$, and field of fractions $K=\text{Frac}(\cO)$, which we assume to be large enough for $G$ and its subgroups in the sense that $K$ contains a root of unity of order $\exp(G)$, the exponent of~$G$. We let $\IF_{q}$ denote the finite field with $q$ elements and assume that  
$q$ is odd. If $r$ is a non-zero natural number, we denote by $\mu_{r}$ the group of the $r$-th roots of unity in an algebraic closure $\IF$ of $\IF_{q}$, i.e. $\mu_{r}=\{\xi\in\IF\mid \xi^r=1\}$.\par
For a subgroup $H\leq G$, we let $[H]$ denote a set of representatives of the conjugacy classes of~$H$,  $[H/\equiv]$ a set of representatives of the conjugacy classes of $H$ up to inverse, and $[H]_{\ell'}$ a set of representatives of the conjugacy classes of $H$ of order prime to $\ell$. Given a positive integer~$n$, we denote by $C_{n}$ the cyclic group of order~$n$ and by $\mathcal{D}_{n}$ the dihedral group of order $n$. The quaternion group of order~$8$ is denoted by $\mathcal{Q}_{8}$. 
\par
For $R\in\{\cO,k\}$, $RG$-modules are assumed to be finitely generated left $RG$-lattices, that is,  free as $R$-modules.  We let  $R$  denote the trivial $RG$-lattice. 
Given a subgroup $H\leq G$, we write $\Res^G_H(M)=M\resgh{G}{H}$ for the restriction of the $kG$-module $M$ to $H$ and  $\Ind_H^G(N)=N\indhg{H}{G}$ for the induction of the $kH$-module $N$ to $G$. Given a normal subgroup $U$ of $G$, we write $\Inf_{G/U}^{G}(M)$ for the inflation of the $k[G/U]$-module $M$ to $G$.  We write $M^{\ast}:=\Hom_k(M,k)$ for the $k$-dual of a $kG$-module~$M$,  $\soc(M)$ for its socle, 
and if $M$ is  uniserial then we denote by $\ell(M)$ its  composition length. 
If $M$ is a $kG$-module and $Q\leq G$, then the Brauer quotient (or Brauer construction) of $M$ at $Q$ is the $k$-vector space 
$$M[Q]:=M^{Q}\big/ \sum_{R<Q}\tr_{R}^{Q}(M^{R}),$$
where $M^{Q}$ denotes the fixed points of $M$ under $Q$ and $\tr_{R}^{Q}$  for $R<Q$ denotes  the relative trace map. 
This vector space has a natural structure of a $kN_{G}(Q)$-module, but also of a $kN_{G}(Q)/Q$-module, and is equal to zero if $Q$ is not an $\ell$-group.
If $S$ is a simple $kG$-module, then we let $P_{S}$ denote its projective cover, and we let $\Sc(G,Q)$ denote the Scott module of the group $G$ with respect to the subgroup $Q\leq G$ (see~\cite{Bro85}, for a definition).
\par 
Unless otherwise stated, by an $\ell$-\emph{block} we mean a block of $kG$. We write $\bB_0(G)$ for the principal $\ell$-block of $G$ and  $\bB_{1}\sim_{SA}\bB_{2}$ to indicate that there exists a source-algebra equivalence between  the $\ell$-blocks $\bB_{1}$ and $\bB_{2}$.
We denote by $\Irr(G)$ (resp. $\Irr(\bB)$)  the set of irreducible $K$-characters of~$G$ (resp. of the $\ell$-block $\bB$). Since the blocks of $\cO G$ are in bijection with the blocks of $kG$ via reduction modulo $J(\cO)$, by abuse of notation and terminology, given an $\ell$-block $\bB$ of $kG$, we write $\Irr(\bB)$ for the set of irreducible $K$-characters of the corresponding block of $\cO G$ and talk about the ordinary irreducible characters of $\bB$. For a cyclic group $H$, we also write $H^\wedge$~instead of $\Irr(H)$. We let $\langle-,-\rangle_{G}$ denote the standard inner product on the space of class functions of $G$.
\par
Finally, we refer the reader to \cite{LinckBook, TheBook, BonBook} for further standard notation and background results in the modular representation theory of finite groups and on the special linear group $\SL_{2}(\IF_{q})$.\\


\vspace{2mm}

\subsection{Trivial source modules and trivial source character tables}
We begin with a quick review of the major properties of trivial source modules. We refer the reader to \cite[\S 2]{BT10}, \cite{Bro85},  \cite[\S 3.11 and \S5.5]{BensonBookI}, \cite[\S 27]{TheBook} for proofs and details of the results mentioned in this section.
\medskip

Given $R\in\{\mathcal{O},k\}$, an $RG$-lattice $M$ is a \emph{permutation} $RG$-lattice if it admits a $G$-invariant $R$-basis, whereas it is an $\ell$-\emph{permutation} $RG$-lattice if it is a direct summand of a permutation $RG$-lattice. 
The indecomposable $\ell$-permutation $RG$-lattices are also called \emph{trivial source lattices} (\emph{trivial source modules} over $k$), because they coincide with the indecomposable $RG$-lattices having a trivial source. \par
It is well known that any trivial source $kG$-module $M$ is liftable to an $\cO G$-lattice (see e.g. \cite[Corollary 3.11.4]{BensonBookI}). More accurately, in general, such modules afford several lifts, but there is a unique one amongst these which is a trivial source $\cO G$-lattice. We denote this trivial source lift by $\wh{M}$ and we let $\chi_{\wh{M}}$ be the ordinary character afforded by  $K\otimes_{\cO}\wh{M}$.
Character values of trivial source lattices have the following properties.

\begin{lem}[{}{\cite[Lemma II.12.6]{LandrockBook}}]\label{lem:tscharacters}
	Let $M$ be a trivial source $kG$-module with character $\chi_{\wh{M}}$  and let $x$ be an $\ell$-element of $G$. Then:
	\begin{enumerate}[{\,\,\rm(a)}]\setlength{\itemsep}{2pt}
		\item $\chi_{\wh{M}}(x)$  is a non-negative integer,  equal to the number of indecomposable direct summands of $\Res^{G}_{\langle x \rangle}(M)$ isomorphic to the trivial $k{\langle x \rangle}$-module; and 
		\item $\chi_{\wh{M}}(x)\neq 0$ if and only if $x$ belongs to some vertex of $M$.
	\end{enumerate}
\end{lem}

\noindent Up to isomorphism, there are only finitely many trivial source $kG$-modules and we will study them vertex by vertex. Thus, we denote by $\TS(G;Q)$ the set of isomorphism classes of indecomposable trivial source $kG$-modules with vertex $Q$. Since projective indecomposable $kG$-modules are trivial source modules with vertex $\{1\}$, with this notation, $\TS(G;1)$ is the set of isomorphism classes of PIMs of $kG$. Moreover, we set $\overline{N}_{G}(Q):=N_{G}(Q)/Q$.
\par

\begin{prop}[{Omnibus properties of $\ell$-permutation and trivial source $kG$-modules}]{\ }
	\begin{enumerate}[{\,\,\rm(a)}]\setlength{\itemsep}{2pt}
		\item The $\ell$-permutation modules are preserved under direct sums, tensor products, inflation, restriction and induction.
		\item If $Q\leq G$ is an $\ell$-subgroup and $M$ is an $\ell$-permutation $kG$-module, then $M[Q]$ is an $\ell$-permutation $k\overline{N}_{G}(Q)$-module. 
		\item The vertices of a trivial source $kG$-module $M$ are the maximal $\ell$-subgroups $Q$ of $G$ such that $M[Q]\neq\{0\}$.
		\item A trivial source $kG$-module $M$ has vertex $Q$ if and only if $M[Q]$ is a non-zero projective $k\overline{N}_{G}(Q)$-module.
		Moreover, if this is the case, then the $kN_{G}(Q)$-Green correspondent $f(M)$ of $M$ is $M[Q]$ (viewed as a $kN_{G}(Q)$-module). Thus, there are \smallskip bijections
		\begin{center}
			\begin{tabular}{ccccc}
				$\TS(G;Q)$      & $\longleftrightarrow$ &    $\TS(N_{G}(Q);Q)$   &  $\longleftrightarrow$ & $\TS(\overline{N}_{G}(Q);1)$  \\
				$M$      & $\mapsto$     &   $f(M)$      &  $\mapsto$     & $M[Q]$
			\end{tabular}.
		\end{center}
		These sets are also in bijection with the set of $\ell'\text{-conjugacy classes of }\overline{N}_{G}(Q)$. 
		\item Let $H\leq G$. Then the Scott module $\Sc(G,H)$ is a trivial source $kG$-module lying in~$\bB_{0}(G)$.  If $Q\in\Syl_{\ell}(H)$, then $Q$ is a vertex of $\Sc(G,H)$ and $\Sc(G,H)\cong\Sc(G,Q)$. 
		\item The trivial source $kG$-modules, together with their vertices, are preserved under source-algebra equivalences. 
	\end{enumerate}
\end{prop}\label{prop:omnts}

\noindent Statements {\rm(a)} to {\rm(e)} are proved in~\cite{Bro85}. Statement (f) is given by \cite[(38.3) Proposition]{TheBook} since source-algebra equivalences preserve vertices and sources. 


\begin{notadefn}[{}{\cite[\S 2]{BT10}}]\label{def:TSCT}{\ }
	\begin{enumerate}[\,\,1.]\setlength{\itemsep}{2pt}
		\item  Set $\mathcal{P}_{G,\ell}:=\{(Q,E)\mid Q\mbox{ is an $\ell$-subgroup of }G,\, E\in \TS(\overline{N}_{G}(Q);1) \}$. The group $G$ acts by conjugation on  $\mathcal{P}_{G,\ell}$ and we  let $[\mathcal{P}_{G,\ell}]$ denote a set of representatives of the $G$-orbits on $\mathcal{P}_{G,\ell}$.
		Then, given $(Q,E)\in [\mathcal{P}_{G,\ell}]$, we let $M_{(Q,E)}\in \TS(G;Q)$ denote the unique trivial source $kG$-module (up to isomorphism) such \smallskip that $M_{(Q,E)}[Q]\cong E$ given by Proposition~\ref{prop:omnts}(d). 
		\item The \emph{trivial source ring}  $a(kG,\mbox{Triv})$ of $kG$ is the Grothendieck group of the category of {$\ell$-permutation} $kG$-modules, with relations corresponding to direct sum decompositions, i.e. $[M]+[N]=[M\oplus N]$, and endowed with the multiplication induced by the tensor product  over~$k$. The identity element is the class of the trivial $kG$-module~$k$. Moreover, $a(kG,\text{Triv})$ is a finitely generated free abelian group with \smallskip basis $\mathcal{B}:=\{[M_{(Q,E)}]\mid (Q,E)\in [\mathcal{P}_{G,\ell}]\}$\,.
		\item Set $\mathcal{Q}_{G,\ell}:=\{(Q,s)\mid Q\mbox{ is an $\ell$-subgroup of }G,\, s\in \overline{N}_{G}(Q)_{{\ell'}} \}$. Again,  $G$ acts on $\mathcal{Q}_{G,\ell}$ by conjugation and we let  $[\mathcal{Q}_{G,\ell}]$ be a set of representatives of the $G$-orbits \smallskip on $\mathcal{Q}_{G,\ell}$. We have 
		$|[\mathcal{Q}_{G,\ell}]|=|[\mathcal{P}_{G,\ell}]|$. 
		\item Given $(Q,s)\in \mathcal{Q}_{G,\ell}$, there is a ring homomorphism  
		\begin{center}
			\begin{tabular}{cccl}
				$\tau_{Q,s}^{G}$\,:            &   $a(kG,\mbox{Triv})$      & $\lra$ &    $K$     \\
				&   $[M]$      & $\mapsto$     &   $\varphi^{}_{M[Q]}(s)$   
			\end{tabular}
		\end{center}
		mapping the class of an $\ell$-permutation $kG$-module $M$ to the value at $s$ of the Brauer character $\varphi^{}_{M[Q]}$ of the Brauer quotient $M[Q]$. It is easy to check that $\tau_{Q,s}^{G}$ only depends on the $G$-orbit of $(Q,s)$, that is,  $\tau_{Q^{x},s^{x}}^{G}= \tau_{Q,s}^{G}$ for every $x\in G$.  Moreoever, this ring homomorphism extends to a $K$-algebra homomorphism
		$$\hat{\tau}_{Q,s}^{G}: K\otimes_{\IZ}a(kG,\mbox{Triv})\lra K\,,$$
		and  the set $\{ \hat{\tau}_{Q,s}^{G}\mid (Q,s)\in  [\mathcal{Q}_{G,\ell}] \}$ is the set of all distinct species (= $K$-algebra homomorphisms) from $K\otimes_{\IZ}a(kG,\mbox{Triv})$ to $K$. These species induce a $K$-algebra isomorphism
		$$\prod_{(Q,s)\in [\mathcal{Q}_{G,\ell}]} \hat{\tau}_{Q,s}^{G}\,:  K\otimes_{\IZ}a(kG,\mbox{Triv})\,\lra \,\prod_{(Q,s)\in [\mathcal{Q}_{G,\ell}]}K\,.$$
		The matrix of this isomorphism with respect to the basis $\mathcal{B}$, denoted by $\text{Triv}_{\ell}(G)$, is called the  \textbf{trivial source character table} (or \textbf{species table})  of the group $G$ at the prime $\ell$ and  is a square matrix of size $|[\mathcal{Q}_{G,\ell}]|\times|[\mathcal{Q}_{G,\ell}]|$.
	\end{enumerate}
\end{notadefn}

\begin{conv}\label{conv:tsctbl}%
	For the purpose of computations we see the trivial source character table as follows, following the approach of \cite[Section 4.10]{LP}.  First, we fix  a set of representatives $Q_1,\ldots, Q_r$ for the conjugacy classes of $\ell$-subgroups of $G$ where $Q_{1}:=\{1\}$ and $Q_{r}\in\Syl_{\ell}(G)$, and for each $1\leq v\leq r$ we set $N_{v}:=N_{G}(Q_{v})$, $\overline{N}_{v}:=N_{G}(Q_{v})/Q_{v}$.   Then, for each $1\leq i,v\leq r$, we define a matrix
	$$T_{i,v}:=\big[ \tau_{Q_{v},s}^{G}([M])\big]_{M\in \TS(G;Q_{i}), s\in [\overline{N}_{v}]_{\ell'} }\,.$$
	With this notation, the trivial source character table of  $G$ at  $\ell$ is the block matrix 
	$$\text{Triv}_{\ell}(G):=[T_{i,v}]_{1\leq i,v\leq r}\,.$$
	Moreover, following  \cite{LP} we label the rows of $\text{Triv}_{\ell}(G)$ with the ordinary characters $\chi^{}_{\widehat{M}}$ instead of the isomomorphism classes of  trivial source modules $M$ themselves.
\end{conv}
\medskip

\noindent Two non-isomorphic trivial source modules $M$ and $N$ with vertex $Q_{i}$  may afford the same ordinary character. Therefore two rows of $\Triv_{\ell}(G)$ may be labelled with the same ordinary character. However, this labelling brings additional information about the trivial source modules and they are distinguished by the values in $T_{i,v}$ for some $v>1$.\\

\enlargethispage{8mm}

\begin{rem}{\ }\label{rem:tsctbl}
	\begin{enumerate}[{\,\,(a)}] \setlength{\itemsep}{2pt}
		\item The block $T_{i,i}$ consists of the values of projective indecomposable characters of $\overline{N}_{i}$ at the $\ell'$-conjugacy classes of $\overline{N}_{i}$. In particular, $T^{}_{1,1}$ consists of the values of projective indecomposable characters of $G$ at the $\ell'$-conjugacy classes of $G$. 
		\item The trivial $kG$-module $k$ is a trivial source module with vertex $Q_{r}$ and  $k=M_{(Q_{r},k)}$.  For every $1\leq v\leq r$ and every $s\in[\overline{N}_{v}]_{\ell'}$ we have $\tau_{Q_{v},s}^{G}(k)=1$ since  $k[Q_{v}]=k$. 
		\item For $\ell$-subgroups $Q_{i},Q_{v}$ of $G$, it follows immediately from the definition and  Proposition~\ref{prop:omnts}(c) that $\tau^{G}_{Q_{v},s}([M_{(Q_{i},E)}])=0$ unless $Q_{v}\leq_{G} Q_{i}$. In other words $T_{i,v}=\mathbf{0}$ if $Q_{v}\not\leq_{G} Q_{i}$. 
		\item If $v=1$ and $1\leq i\leq r$, then $M[\{1\}]=M$ and so 
		$$\tau_{\{1\},s}^{G}([M])=\chi^{}_{\widehat{M}}(s)$$
		for every $M\in \TS(G;Q_{i})$ and every $s\in [G]_{\ell'}$. In particular, for every $M\in \TS(G;Q_{i})$ we have
		$$\tau_{\{1\},1}^{G}([M])=\dim_{k}M\,.$$
		This explains the terminology \emph{trivial source character table} and our labelling of the rows by the characters. 
		\item More generally, if $(Q,s)\in \mathcal{Q}_{G,\ell}$ with $s=1$ and $M$ is any $\ell$-permutation $kG$-module, then  $\tau^{G}_{Q,1}([M])=\dim_{k} M[Q]$.
		
		\item Species values of $G$ may be computed using species values of smaller groups through the following formula \cite[2.16. Remarks]{BT10}
		$$\tau_{Q,s}^{G} =  \tau_{\{1\},s}^{\langle Qs \rangle /Q}\circ \Br_Q^{\langle Qs \rangle}\circ \Res^G_{\langle Qs \rangle},$$
		where $\langle Qs \rangle$ denotes the inverse image in $N_G(Q)$ of the cyclic group $\langle s \rangle$ of $\overline{N}_G(Q)$ and $\Br_Q^{\langle Qs \rangle}: A(k{\langle Qs \rangle},\mbox{Triv})\lra A(kN_{\langle Qs\rangle}(Q)/Q,\Triv)$ denotes the ring homomorphism induced by the correspondence $M\mapsto M[Q]$ for trivial source $k{\langle Qs \rangle}$-modules $M$. (See  \cite[2.11. Proposition]{BT10}.)
	\end{enumerate}
\end{rem}

\vspace{2mm}

\subsection{Blocks with cyclic defect groups.}\label{subsec:CyclicBlocks}
Many of the trivial source modules in the cases considered in this article lie in  blocks with cyclic defect groups. Thus we review here how to read  these modules from the Brauer tree of the block. 
We refer the reader to \cite{HL20} and the references therein for more details.\\ 

Assume $\bB$ is an $\ell$-block of $kG$ with a non-trivial cyclic defect group $D\cong C_{\ell^n}$. 
If $e$ denotes the inertial index of $\bB$ and $m:=\frac{|D|-1}{e}$ the exceptional multiplicity of~$\bB$, then $e\mid \ell-1$, there are $e$ simple $\bB$-modules and $e+m$ ordinary irreducible characters. We may write
$$\Irr(\bB)=\Irr'(\bB)\sqcup\{\chi_{\lambda} \mid \lambda\in\Lambda\}\,,$$
where $|\Irr'(\bB)|=e$ and $\Lambda$ is an index set with  $|\Lambda|:=m$. If $m>1$, the characters $\{\chi_{\lambda} \mid \lambda\in\Lambda\}$ are the exceptional characters of $\bB$, which  all restrict in the same way to the $\ell$-regular conjugacy classes of~$G$, and $\Irr'(\bB)$ is the set of the non-exceptional characters of $\bB$. 
We set $\chi_{\Lambda}:=\sum_{\lambda\in\Lambda}\chi_{\lambda}$ and $\Irr^{\circ}(\bB):=\Irr'(\bB)\sqcup\{\chi_{\Lambda}\}$. \par
The Brauer tree  $\sigma(\bB)$ of $\bB$ is  the graph whose 
vertices are labelled by the ordinary characters in $\Irr^{\circ}(\bB)$ and whose edges are labelled by the simple $\bB$-modules.  
If $m>1$  the vertex corresponding to $\chi_{\Lambda}$ is called the \emph{exceptional vertex} and is indicated with a black circle in our drawings of $\sigma(\bB)$.  
Furthermore, $\sigma(\bB)$ bears a \emph{type function} associating a sign to each vertex in an alternating way and which is defined as follows: if $D_1\leq D$ denotes the unique subgroup of $D$ of order $\ell$ and $x$ is a generator of $D_{1}$, then a vertex $\chi\in\Irr^{\circ}(\bB)$ of $\sigma(\bB)$ is said to be positive  if $\chi(x)>0$ and we write $\chi>0$, whereas it is said to be negative if $\chi(x)<0$ and in this case we write $\chi<0$. See \cite[\S4.2]{HL20}. Clearly the trivial character is always positive.
\par
In order to understand trivial source modules lying in $\bB$, it is necessary to consider $\bB$ up to source-algebra equivalence. By the work of Linckelmann \cite{Linck96}, the source algebra of $\bB$ is characterised by three parameters: $\sigma(\bB)$, the type function, and an indecomposable capped endo-permutation $kD$-module $W(\bB)$. Letting $\bb$ be the Brauer correspondent of $\bB$ in $N_G(D_1)$ and $\bc$ be a block of $C_G(D_1)$ covered by~$\bb$, the $kD$-module $W(\bB)$ is defined to be a source of the unique simple $\bc$-module.  
Moreover, $\bB$ contains precisely $e$ trivial source $kG$-modules for each possible vertex $Q\leq D$. These trivial source modules are explicitly classified by~\cite[Theorem~5.3]{HL20}, as a function of the three parameters mentioned above. 

\begin{rem}[Classification of trivial source modules in cyclic $\ell$-blocks]{\label{rem:tsmodulecyclicdef}}The description of the trivial source $\bB$-modules can be split in three main cases.
	\begin{enumerate}[1.]
		\item \textbf{Trivial vertex:} The trivial source $\bB$-modules with vertex $\{1\}$ (i.e. the PIMs of~$\bB$) can be read off  from $\sigma(\bB)$ as follows (see \cite[\S 4.18]{BensonBookI}). If $S$ is a simple $\bB$-module corresponding to the edge 
		$$
		\begin{tikzcd}
			\cdots\,\underset{\chi_b}{{\Circle}} \arrow[r, dash,"S"] & \underset{\chi_c}{{\Circle}} \,\cdots
		\end{tikzcd}
		$$
		of $\sigma(\bB)$ with $\chi_{b},\chi_{c}\in\Irr^{\circ}(\bB)$, then  the projective cover $P_{S}$ of $S$ has the form
		$$P_{S}=\boxed{\begin{smallmatrix} S\\ V_b\oplus\, V_c \\ S\end{smallmatrix}}$$
		where $S=\soc(P_{S})=\head(P_{S})$ and the heart of $P_{S}$ is  $\rad(P_{S})/\soc(P_{S})=V_b\oplus V_c$ for two uniserial (possibly zero) $\bB$-modules $V_b$ and $V_c$.
		The projective indecomposable character corresponding to $P_{S}$ is \smallskip $\chi^{}_{\wh{P(S)}}=\chi_{b}+\chi_{c}$.
		\item \textbf{Full vertex:} When $\ell$ is odd and $W(\bB)=k$ (the latter condition will always be satisfied in the cases considered), the trivial source $\bB$-modules with vertex $D$ are precisely the \emph{positive hooks} of $\bB$. The hooks of $\bB$ are by  definition  the liftable uniserial modules 
		$$H_b:=\boxed{\begin{smallmatrix}S\\V_b\end{smallmatrix}}\qquad\text{ and }\qquad H_c:=\boxed{\begin{smallmatrix}S\\V_c\end{smallmatrix}}$$
		where $S$ runs through the set of simple $\bB$-modules. When $W(\bB)=k$  one of $H_{b}$ and $H_{c}$ is a trivial source module and the other has $\Omega(k)$ as a source. 
		If $H_{b}$ is a trivial source module, then $\chi^{}_{\widehat{H_{b}}}=\chi_{b}>0$ so $H_{b}$ is a positive hook. Otherwise, $H_{c}$ is a positive hook. 
		See~\cite[Theorem~5.3(a) and (b)(1)]{HL20}. 
		When $\ell=2$, the only situation of interest for this article is the classification of trivial source modules with full vertex $D \cong C_2$. This case is treated in Remark~\ref{rem:cyclicDefectC2C4}. 
		\item We refer the reader directly to \cite[Theorem~5.3(a) and (b)(2)-(7)]{HL20} 
		for the description of the trivial source $\bB$-modules with non-trivial, non-full vertices,  and their ordinary characters. 
	\end{enumerate}
\end{rem}

\begin{rem}\label{rem:cyclicDefectC2C4}%
	When $\ell=2$ and $D\cong C_{2}$, then clearly $e=m=1$, 
	the Brauer tree is of the form
	$$\sigma(\bB)=  \xymatrix@R=0.0000pt@C=30pt{	
		{\Circle} \ar@{-}[r]^{S}  & {\Circle}    \\
		{^{\chi^{}_{1}}}&{^{\chi^{}_2}}
	}$$
	and, up to isomorphism, $\bB$ contains precisely two indecomposable modules, both of which are trivial source modules. Clearly $P_{S}$ is the unique PIM and  $S$ is the unique trivial source $\bB$-module with vertex $D$. We have $\chi^{}_{\wh{P(S)}}=\chi^{}_{1}+\chi^{}_{2}$ and, by Lemma~\ref{lem:tscharacters}, $\chi_{\widehat{S}}=\chi^{}_{i}$ where $i\in\{1,2\}$ is such that $\chi^{}_{i}>0$. 
\end{rem}

\enlargethispage{3mm}

\begin{lem}
	\label{lem:equalgreencorrs}
	Let $\bB$ be an $\ell$-block of $kG$ with a non-trivial cyclic defect group $D\cong C_{\ell^{n}}$. 
	For each $0\leq i\leq n$, let $D_{i}\cong C_{\ell^{i}}$ be the unique subgroup of $D$ of order $\ell^{i}$. 
	Let $1\leq v\leq n$ and let $M$ be a trivial source $\bB$-module with vertex $D_{v}$. 
	If $N_{G}(D_{1})=N_{G}(D_{i})$ for each $1\leq i\leq n$, then for every $1\leq i\leq v$ we have  
	$$M[D_{i}]=M[D_{v}]=f(M),$$
	that is, the Green correspondent $f(M)$ of $M$ in $kN_{G}(D_{1})$.
\end{lem}
\begin{proof}
	Because $M$ is a trivial source module, by \cite[Exercise (27.4)]{TheBook}, we may write 
	$$\Res^{G}_{N_{G}(D_{i})}=L_{1}\oplus L_{2}$$
	where $L_{1}=M[D_{i}]$ is the direct sum of all direct summands with a vertex containing $D_{i}$ and every indecomposable direct summand of $L_{2}$ has vertex not containg $D_{i}$.
	Now, as $M$ belongs to a block with cyclic defect group, we also know that
	$$\Res^{G}_{N_{G}(D_{i})}=f(M)\oplus X$$
	where $X$ is the direct sum of a projective $kN_{G}(D_{1})$-module and $kN_{G}(D_{1})$-modules lying in blocks with defect groups not conjugate to $D$ under $N_{G}(D_{1})$. (See e.g. \cite[Lemma 6.5.1]{BensonBookI}.)
	Comparing both decompositions yields the claim. 
\end{proof}


\vspace{2mm}

\subsection{The special linear group $\SL_2(q)$ and its $\ell$-subgroups}
For the remainder of this article, we assume that  $G := \SL_2(q)$ is the special linear group of degree $2$ over the finite field $\IF_{q}$ where~$q$ is assumed to be a positive power of an odd prime $p$. We write $Z:=Z(G)=\{\pm I_{2}\}$ and set $\overline G :=G/Z= \PSL_2(q)$.
We let $n \geq 1$ denote the maximal power of~$\ell$ dividing $|G| = q(q-1)(q+1)$, that is, $|G|_{\ell}=\ell^{n}$.
\par
For details on the structure of $\SL_{2}(q)$ and its block theory, we refer the reader to \cite{BonBook}. For convenience we recall here the notation and important preliminaries from  \cite{BonBook}, which we will use throughout.  We let $U$ be the subgroup of $G$ consisting of the unipotent upper-triangular matrices. 
We let $T$ and $T'$ denote a split and a non-split torus of $G$, respectively. These are cyclic groups of order $q-1$ and $q+1$, respectively, and we have isomorphisms 
\[ \mathbf{d}: \mu_{q-1}=\IF_{q}^{\times} \longrightarrow T, a\mapsto \text{diag}(a,a^{-1})\,, \qquad \qquad \mathbf{d'}: \mu_{q+1} \longrightarrow T',\]
where  $\mathbf{d'}:\GL_{\IF_{q}}(\IF_{q^{2}})\lra\GL_{2}(\IF_{q})$ is an isomorphism defined and  fixed by the choice of an $\IF_{q}$-basis of $\IF_{q^{2}}$ and $T'$ is defined to be $\mathbf{d'}( \mu_{q+1} )$. We will often identify $T$ with $\mu_{q-1}$ and $T'$ with $\mu_{q+1}$ via these isomorphisms. 

We let $S_\ell$ and $S_{\ell}'$ denote Sylow $\ell$-subgroups of $T$ and $T'$ respectively. Then the tori decompose into direct products
$ T = S_{\ell} \times T_{\ell'}$, and $ T' = S_{\ell}' \times T_{\ell'}'$.  
If $q = 3$ then $T = Z$ so $C_G(T) = N_G(T) = G$. Otherwise, $N_G(T) =  \langle T, \sigma \rangle=:N$, where $\sigma := \left( \begin{smallmatrix} 0 & -1 \\ 1 & 0 \end{smallmatrix} \right)$
is an element of $G$ of order 4. 
Similarly, $N':=N_{G}(T')= \langle T', \sigma' \rangle$, where $\sigma' \in G$ is an element  of order 4 such that $\sigma'^2 = -I_2$.

\begin{lem}[{Some $\ell$-subgroups of $G$} {\cite[Theorem 1.4.3]{BonBook}}]{\ }
	\label{lem:omnibus}\label{lem:omnibusl=2}
	\begin{enumerate}[{\,\,\rm(a)}] \setlength{\itemsep}{2pt}
		\item[\rm(a)] If $\ell$ is odd, then the following assertions hold.	
		\begin{itemize}
			\item[\rm(i)]  If $\ell \mid q-1$ then $S_\ell\in\Syl_{\ell}(G)$ and for any non-trivial $\ell$-subgroup $Q$ of $T$, $C_{G}(Q) = T$ and $N_{G}(Q) = N$.
			\item[\rm(ii)]  If $\ell \mid q+1$ then	$S_\ell'\in\Syl_{\ell}(G)$ and for any non-trivial $\ell$-subgroup $Q$ of $T'$, $C_{G}(Q) = T'$ and $N_{G}(Q) = N'$.
		\end{itemize}
		\item[\rm(b)]  If  $\ell=2$ and $q\equiv 3\pmod{8}$, then the following assertions hold.
		\begin{itemize}
			\item[\rm(i)] $S_{2}=Z\cong C_{2}$, $S_{2}'\cong C_{4}$, and  $Q:=\langle S_{2}',\sigma'\rangle\cong \mathcal{Q}_{8}$ is a Sylow $2$-subgroup of $G$. Moreover, 
			$$N_{G}(Q)/Q\cong C_{3},\quad N_{G}(S_{2})=G,\quad N_{G}(S_{2}')=N'\,.$$
			\item[\rm(ii)]  A Sylow $2$-subgroup of $\overline{G}$ is $\overline{Q}:=Q/Z\cong C_{2}\times C_{2}$. Moreover,
			$$N_{\overline{G}}(\overline{Q})\cong \fA_{4},\quad  N_{\overline{G}}(S_{2}'/Z)= \langle T'/Z,\sigma'Z\rangle \cong \mathcal{D}_{q+1}.$$
		\end{itemize} 
		\item[\rm(c)]  If  $\ell=2$ and $q\equiv -3\pmod{8}$, then the following assertions hold.
		\begin{itemize}
			\item[\rm(i)]  $S_{2}\cong C_{4}$, $S_{2}'=Z\cong C_{2}$, and  $Q:=\langle S_{2},s\rangle\cong \mathcal{Q}_{8}$ is a Sylow $2$-subgroup of $G$. Moreover, 
			$$N_{G}(Q)/Q\cong C_{3},\quad N_{G}(S_{2})=N, \quad N_{G}(S_{2}')=G\,.$$
			\item[\rm(ii)] A Sylow $2$-subgroup of $\overline{G}$ is $\overline{Q}:=Q/Z\cong C_{2}\times C_{2}$. Moreover,
			$$N_{\overline{G}}(\overline{Q})\cong \fA_{4}, \quad  N_{\overline{G}}(S_{2}/Z)= \langle T/Z,sZ\rangle \cong \mathcal{D}_{q-1}.$$
		\end{itemize} 
	\end{enumerate}
\end{lem}

\begin{proof}
	Part (a) is proved in \cite[Theorem 1.4.3 (a),(b)]{BonBook} for $Q = S_{\ell}$ or $S_{\ell}'$, but the proof holds more generally for any non-trivial $\ell$-subgroup $Q$ of $T$ when $\ell \mid q-1$, and for any non-trivial  $\ell$-subgroup $Q$ of $T'$ when $\ell \mid q+1$. 
	The claims about $G=\SL_{2}(q)$ in (b) and (c) are given by \cite[Theorem 1.4.3 (c), (d) and their proofs]{BonBook}. The claims about $\overline{G}=G/Z$ follow immediately. 
\end{proof}

\begin{lem}\label{lem:B_0}\label{NG(D_1)}
	Assume $\ell$ is odd and let $\bB$ be an $\ell$-block of $kG$ for $G=\SL_{2}(q)$  ($q$ odd)  with a non-trivial cyclic defect group $D\cong C_{\ell^{n}}$. Then the indecomposable capped endo-permutation $kD$-module $W(\bB)$ parametrising the source algebra of $\bB$ (in the sense of \S\ref{subsec:CyclicBlocks}) is the trivial $kD$-module. 
\end{lem}

\begin{proof}
	Let $D_{1}$ be the unique subgroup of order $\ell$ of $D$. Then $C_G(D)=C_G(D_1)$ by Lemma~\ref{lem:omnibus}. Since $D\unlhd C_G(D)$ it follows from Clifford theory that $D$ acts trivially on the simple $kC_G(D_{1})$-modules, because their restriction to $D$ is semisimple. Thus the simple $kC_G(D_{1})$-modules are trivial source modules. Now, by definition (see~\S\ref{subsec:CyclicBlocks}) $W(\bB)$ is a source of such a simple module, hence $W(\bB)=k$. 
\end{proof}

\vspace{2mm}
\subsection{The ordinary character table of $\SL_{2}(q)$}
For convenience, we reproduce the character table of $G=\SL_{2}(q)$ here, using the notation of \cite[Table 5.4]{BonBook} up to small changes.
\par
We set  $\Gamma :=  [(\mu_{q-1} \setminus \{\pm 1\})/\equiv]$, $\Gamma' :=  [(\mu_{q+1} \setminus \{\pm 1\})/\equiv]$ and  write $\Gamma_{\ell'}$, $\Gamma'_{\ell'}$ for the $\ell'$-elements in $\Gamma$ and $\Gamma'$ respectively.  We fix a non-square element $z_0 \in \mathbb{F}_q^\times$ (this is possible as $q$ is odd), and set $ u = \left( \begin{smallmatrix} 1	& 1 \\ 0 & 1 \end{smallmatrix} \right)$ and $u_- = \left(\begin{smallmatrix} 1	& z_0 \\ 0 & 1 \end{smallmatrix} \right)$. Then a set of representatives of the conjugacy classes of $G$ is given by 
\[ \{I_2, -I_2\} \cup \{u_+, u_-, -u_+, -u_-\}\cup \{ \mathbf{d}(a) \mid a \in \Gamma \} \cup \{\mathbf{d}'(\xi) \mid \xi \in \Gamma' \}. \]
Next, we recall that all ordinary irreducible characters of $G$ arise as constituents of characters induced from $T$ by Harish-Chandra induction or induced from $T'$ by Deligne-Lusztig induction, which we denote by 
\begin{align*}
	R : \IZ \Irr(T) \longrightarrow \IZ \Irr(G)\qquad \mbox{ and }\qquad R' : \IZ \Irr(T') \longrightarrow \IZ \Irr(G), 
\end{align*}
respectively. Then, we denote the trivial character of $G$ by $1_G$ and the Steinberg character by~$\St$. We let $\alpha_0$ (respectively $\theta_0$) denote the unique character of $T$ (respectively $T'$) of order 2. Then 
\[R(\alpha_0) = R_+(\alpha_0) + R_-(\alpha_0) \qquad \text{ and }\qquad R'(\theta_0) = R'_{+}(\theta_0) + R'_{-}(\theta_0),\]
where $R_{\pm}(\alpha_0), R'_{\pm}(\theta_0)\in\Irr(G)$. The remaining $q-3$ non-trivial characters $\alpha \in \Irr(T)$ satisfy $R(\alpha) = R(\alpha^{-1}) \in \Irr(G)$, giving us $\frac{q-3}{2}$ more irreducible characters in $\Irr(G)$. Similarly, the remaining $q-1$ non-trivial characters $\theta \in \Irr(T')$ satisfy $R'(\theta) = R'(\theta^{-1}) \in \Irr(G)$, giving us the final $\frac{q-1}{2}$ irreducible characters of $G$. In summary, we have 
\begin{align*} 
	\Irr(G) = & \hspace{1ex}  \{ 1_G, \St, R_\pm(\alpha_0), R'_\pm(\theta_0) \}  \hspace{2ex}  \cup   \hspace{2ex}  \{ R(\alpha) \mid \alpha \in [T^\wedge/\equiv], \alpha^2 \neq 1\} \\
	& \hspace{6ex}  \cup \hspace{2ex}  \{ R'(\theta) \mid \theta \in [T'^\wedge / \equiv], \theta^2 \neq 1\}.
\end{align*}
With this notation the character table of $G=\SL_{2}(q)$ is then as given in Table~\ref{tab:CTBLsl2}, where  we let $q_0 := q$ if $q \equiv 1\pmod{4}$, and $q_0 := -q$ if $q \equiv 3\pmod{4}$.

\renewcommand*{\arraystretch}{1.4}
\begin{longtable}{|c||c|c|c|c|}
	\caption{Character table of $\SL_2(q)$.}\label{tab:CTBLsl2}	
	\\  \hline 
	\begin{tabular}{c}
		\textbf{Conjugacy class } \\
		\textbf{representative $g$}
	\end{tabular}
	&
	\begin{tabular}{c}
		\textbf{$\varepsilon I_2$ } \\
		\textbf{	$ \varepsilon \in \{ \pm 1\}$}
	\end{tabular}
	& 
	\begin{tabular}{c}
		\textbf{$\textbf{d}(a)$} \\
		\textbf{	$ a \in \Gamma$}
	\end{tabular}
	&
	\begin{tabular}{c}
		\textbf{$\textbf{d}'(\xi)$} \\
		\textbf{	$ \xi  \in \Gamma'$}
	\end{tabular}
	&
	\begin{tabular}{c}
		\textbf{$ \varepsilon u_{\tau}$} \\
		\textbf{	$ \varepsilon, \tau \in \{ \pm 1\}$}
	\end{tabular}
	\\ \hline \hline 			
	
	No. of classes 
	& 2
	& $\frac{(q-1) - 2}{2}$
	& $\frac{(q+1) - 2}{2}$
	& 4
	\\ \hline 
	
	Order of $g$
	& $o(\varepsilon)$
	& $o(a)$
	& $o(\xi)$
	& $p.o(\varepsilon)$
	\\ \hline 
	
	Class size
	& 1
	& $q(q+1)$
	& $q(q-1)$
	& $\frac{q^2 -1}{2}$
	\\ \hline 
	
	$C_G(g)$
	& $G$
	& $T$
	& $T'$
	& $Z \times U$
	\\ \hline \hline 
	
	$1_G$
	& $1$
	& $1$
	& $1$
	& $1$
	\\ \hline

	$\St$
	& $q$
	& $1$
	& $-1$
	& $0$
	\\ \hline
	
	$R(\alpha)$, $\alpha^2 \neq 1$
	& $(q+1) \alpha(\varepsilon)$
	& $\alpha(a) + \alpha(a)^{-1}$
	& $0$
	& $\alpha(\varepsilon)$
	\\ \hline
	
	$R'(\theta)$, $\theta^2 \neq 1$
	& $(q-1)  \theta(\varepsilon)$
	& $0$
	& $-\theta(\xi) - \theta(\xi)^{-1}$
	& $-\theta(\varepsilon)$
	\\ \hline

	$R_{\pm} (\alpha_0)$ 
	& $\frac{(q+1)\alpha_0(\varepsilon)}{2} $
	& $\alpha_0(a)$
	& $0$
	& $\alpha_0(\varepsilon) \frac{1 \pm \tau \sqrt {q_0}}{2}$
	\\ \hline

	$R_{\pm}' (\theta_0)$ 
	& $\frac{(q-1)\theta_0(\varepsilon)}{2}$
	& $0$
	& $-\theta_0(\xi)$
	& $\theta_0(\varepsilon) \frac{-1 \pm \tau \sqrt {q_0}}{2}$
	\\ \hline
\end{longtable}

\noindent\textbf{Convention.} We  identify the $\IC$-characters of $\overline{G} = \PSL_{2}(q)$ with the $\IC$-characters of $G = \SL_{2}(q)$ with the centre $Z$ in their kernel. As a consequence, when $\ell=2$ we  label the $\IC$-characters and the $\ell$-blocks of $\PSL_{2}(q)$ using the corresponding labelling in $\SL_{2}(q)$.

\vspace{2mm}
\subsection{Characters and conjugacy classes of $N$ and $N'$}
\label{subsec:NandN'}

We use the notation for the ordinary characters and conjugacy classes of $N$ and $N'$ given in \cite[Sections 6.2.1 and 6.2.2]{BonBook} and let 
\[\Irr(N) = \{1_N, \varepsilon, \chi^\pm_{\alpha_0}\} \cup \{ \chi_{\alpha} \mid \alpha \in [T^\wedge / \equiv], \alpha^2 \neq 1\} ,\]
\[\Irr(N') = \{1_{N'}, \varepsilon', \chi'^\pm_{\alpha_0}\} \cup \{ \chi'_{\theta} \mid \theta \in [T'^\wedge / \equiv], \theta^2 \neq 1\} .\]
Let $\sigma_+ := \sigma$ and let $\sigma_- := \sigma \mathbf{d}(z_1)$ for some non-square element $z_1 \in \IF_q^\times$. Let $\sigma'_+ := \sigma'$ and $\sigma'_- := \sigma \mathbf{d'}(\xi_0)$ for some non-square element $\xi_0 \in \mu_{q+1}$. Then a set of representatives of the conjugacy classes of $N$ is given by 
\[ \{I_2, -I_2\} \cup \{\sigma_+, \sigma_-\}\cup \{ \mathbf{d}(a) \mid a \in \Gamma \},  \] 
and a set of representatives of the conjugacy classes of $N'$ is given by
\[ \{I_2, -I_2\} \cup \{\sigma'_+, \sigma'_-\}\cup \{ \mathbf{d'}(\xi) \mid \xi \in \Gamma' \}.  \] 
We refer the reader to \cite[Tables 6.2 and 6.3]{BonBook} for the character tables of $N$ and $N'$.


\vspace{6mm}
\section{$\SL_2(q)$ with  $2\neq\ell\mid (q-1)$}
\label{sec:lmidq-1}

In this section we assume that $G = \SL_2(q)$ with $q$ odd and $2 \neq \ell \mid q-1$.  By Lemma~\ref{lem:omnibusl=2}(a)(i) the Sylow $\ell$-subgroups of $G$ are cyclic so all blocks have cyclic defect groups.  Trivial source modules in blocks of defect zero are just PIMs of $kG$, and the trivial source modules  lying in blocks with a non-trivial cyclic defect group are easily dealt with via Remark~\ref{rem:tsmodulecyclicdef} and  \cite[Theorem~5.3]{HL20}.  

\begin{nota}\label{rem:lmidq-1}
	In order to describe $\Triv_{\ell}(G)$ according to Convention~\ref{conv:tsctbl} we adopt the following notation. We fix $Q_{n+1}:=S_{\ell}\cong C_{\ell^{n}}$ and for each $1\leq i\leq n$ we let $Q_{i}$ denote the unique cyclic subgroup of $Q_{n+1}$ of order $\ell^{i-1}$. The chain of subgroups
	 \[\{1\} = Q_1 \leq \dots  \leq Q_{n+1} \in \Syl_{\ell}(G)\]
is then our fixed set of representatives for the conjugacy classes of $\ell$-subgroups of $G$.  We fix the following set of representatives for  the $\ell'$-conjugacy classes of $G$:
$$[G]_{\ell'}:=\{\pm I_{2}\}\cup \{\mathbf{d}(a)\mid a\in\Gamma_{\ell'}\}\cup  \{\mathbf{d}'(\xi)\mid \xi\in\Gamma'_{\ell'}\}\cup\{\varepsilon u_{\tau}\mid \varepsilon,\tau\in\{\pm1\}\}\,.$$
For any $2\leq v\leq n+1$, $1\leq i\leq n+1$, the columns of $T_{i,v}$ are labelled by a set of representatives for the $\ell'$-conjugacy classes of $\overline N_{v} = N_G(Q_{v})/Q_{v}=N/Q_{v}$ as  $N_G(Q_v) = N_G(T) = N$ for each $2\leq v \leq n+1$ by Lemma~\ref{lem:omnibus}(a)(i).  However, since $Q_{v}$ is an $\ell$-group we will simply label the columns of $T_{i,v}$ by the following fixed set of representatives for  the $\ell'$-conjugacy classes of $N$:
$$[N]_{\ell'}:=\{\pm I_{2}\}\cup  \{\mathbf{d}(a)\mid a\in\Gamma_{\ell'}\}\cup \{\sigma_{\tau}\mid\tau\in\{\pm1\}\}\,.$$	
Moreover, in order to describe the exceptional characters occurring  as constituents of the trivial source characters, for each $0\leq  i \leq n$ we fix
 \[
\pi_{q,i} :=\frac{(q-1)_\ell\cdot \ell^{-i} - 1}{2}\,,
 \]
we let $\pi_q := \pi_{q,0}$, and note that 
$\pi_{q,n}=0$. These numbers naturally come from the classification of the trivial source modules in cyclic blocks in \cite{HL20}. 
\end{nota}

\vspace{2mm}
\subsection{The $\ell$-blocks and trivial source characters of $G$}


\begin{lem} {\label{lem:blockslmidq-1}}
	When $2 \neq \ell \mid q-1$ the $\ell$-blocks of $G$, their defect groups and their Brauer trees with type function are as given in Table~\ref{tab:blockslmidq-1}.
\end{lem}

\begin{longtable}{c||c|c|c}
	\caption{The $\ell$-blocks of $\SL_2(q)$ when $2 \neq \ell \mid q-1$.}	
	\label{tab:blockslmidq-1}
	\\  
	Block
	&	\begin{tabular}{c}
		Number of Blocks \\
		(Type)
	\end{tabular}
	& 	\begin{tabular}{c}
		Defect \\
		Groups
	\end{tabular}
	& 
	\begin{tabular}{c}
		Brauer Tree with Type  \\
		Function/ $\Irr(\bB)$
	\end{tabular}
	\\ \hline \hline 			
	
	$\mathbf{B}_0(G)$ 
	& 	\begin{tabular}{c}
		1 \\
		\small{(Principal)} 
	\end{tabular} 
	& 	$C_{\ell^n}$ 
	&	\begin{tabular}{c}
		$$ \xymatrix@R=0.0000pt@C=30pt{
			{_+} & {_-}& {_+}\\
			{\Circle} \ar@{-}[r] 
			& {\CIRCLE} \ar@{-}[r] 
			&{\Circle} \\
			{^{1_G}}&{^{\Xi}}&{^{\St}}
		}$$\\
	{\footnotesize $\Xi:=\sum\limits_{\eta \in [S_{\ell}^\wedge / \equiv] \setminus \{1\}}R(\eta)$}\\
	\end{tabular}
	\\ \hline 
	
	\begin{tabular}{c}
		$A_{\alpha_0}$ \\
		\footnotesize{($\alpha_0 \in T_{\ell'}^\wedge \setminus \{1\}$}, 
		\footnotesize{$\alpha_0^2 = 1$)} 
	\end{tabular} 
	& 	\begin{tabular}{c}
		1 \\
		\small{(Quasi-isolated)} 
	\end{tabular}
	& 	$C_{\ell^n}$ 	
	& 	\begin{tabular}{c}
		$$  \xymatrix@R=0.0000pt@C=30pt{	
			{_+} & {_-}& {_+}\\
			{\Circle}  \ar@{-}[r]
			& {\CIRCLE}  \ar@{-}[r]  
			&{\Circle} \\
			{^{R_+(\alpha_0)}}&{^{\Xi_{\alpha_0}}}&{^{R_-(\alpha_0)}}
		}$$\\
		{\footnotesize $\Xi_{\alpha_0}   := \sum\limits_{\eta \in [S_{\ell}^\wedge / \equiv] \setminus \{1\}} R(\alpha_0\eta)$ }\\
	\end{tabular}
	\\ \hline  
	
	\begin{tabular}{c}
		$A_\alpha$ \\
		\footnotesize{($\alpha \in [T_{\ell'}^\wedge / \equiv]$}, 
		\footnotesize{$\alpha^2 \neq 1$)}
	\end{tabular}
	& 	\begin{tabular}{c}
		$\frac{(q-1)_{\ell'} - 2}{2}$\\
		\small{(Nilpotent)} 
	\end{tabular}	
	& 	$C_{\ell^n}$ 		
	& 	\begin{tabular}{c}
		$  \xymatrix@R=0.0000pt@C=30pt{	
			{_+} & {_-}\\
			{\Circle} \ar@{-}[r] 
			& {\CIRCLE}    \\
			{^{R(\alpha)}}&{^{\Xi_\alpha}}
		}$ \\
		{\footnotesize $\Xi_{\alpha}  := \sum\limits_{\eta \in S_{\ell}^\wedge \setminus \{1\}} R(\alpha\eta)$}\\
	\end{tabular}
	\\ \hline 
	
	\begin{tabular}{c}
		$A'_{\theta_0,\pm}$ \\
		\footnotesize{($\theta_0 \in T_{\ell'}'^\wedge \setminus \{1\}$}, 
		\footnotesize{$\theta_0^2 = 1$)}
	\end{tabular}
	& 	\begin{tabular}{c}
		2 \\
		\small{(Defect zero)}
	\end{tabular}
	& 	$\{1\}$ 
	& 	\begin{tabular}{c}
		$\Irr(A'_{\theta_0,+}) = \{R_{+}'(\theta_0)\}$ \\
		$\Irr(A'_{\theta_0,-}) = \{R_{-}'(\theta_0)\}$
	\end{tabular}
	\\ 	\hline

	\begin{tabular}{c}
		$A'_\theta$  \\
		\footnotesize{($\theta \in [T_{\ell'}'^\wedge / \equiv]$}, 
		\footnotesize{$\theta^2 \neq 1$)} 
	\end{tabular}
	& 	\begin{tabular}{c}
		$\frac{(q+1)_{\ell'} - 2}{2}$ \\
		\small{(Defect zero)}
	\end{tabular} 
	& 	$\{1\}$ 
	& 	$\Irr(A'_\theta) = \{R'(\theta)\}$
	\\ \hline 
\end{longtable}

\begin{proof}
	Apart from the type functions of the Brauer trees, the information in Table \ref{tab:blockslmidq-1} can be found in \cite[Chapters 8 and 9]{BonBook}. The type function of $\sigma(\bB_0(G))$ is immediate as the trivial character is positive. Since the $\ell'$-characters of $T$ take the value $1$ on $\ell$-elements, it follows from the character table of $G$ (Table~\ref{tab:CTBLsl2}) that $R_+(\alpha_0)$ and $R_{\alpha}$ are positive. This determines the type functions for $\sigma(A_{\alpha_0})$ and $\sigma(A_{\alpha})$ for each $\alpha \in [T_{\ell'}^\wedge / \equiv]$, $\alpha^2 \neq 1$. 
\end{proof}

\begin{lem}
	\label{lem:tsmodslmidq-1}
	When $2 \neq \ell \mid q-1$ the ordinary characters $\chi_{\widehat{M}}$ of the trivial source $kG$-modules~$M$  are as given in Table~\ref{tab:tsmodslmidq-1}, where for each $1\leq i\leq n$, 
\[\Xi_i = \sum\limits_{j=1}^{\pi_{q,i}} R(\eta_j)\quad\text{ and }\quad  \Xi_{\alpha_0,i} = \sum\limits_{j=1}^{\pi_{q,i}} R(\alpha_0\eta_j)\]
are  sums of  $\pi_{q,i}$ pairwise distinct exceptional characters  in $\bB_{0}(G)$, resp. $A_{\alpha_{0}}$, and  for any $\alpha \in \Irr(T_{\ell'})$ with $\alpha^2 \neq 1$, 
\[\Xi_{\alpha,i} = \sum\limits_{j=1}^{2\pi_{q,i}} R(\alpha\eta_j)\]
is a sum of $2\pi_{q,i}$ pairwise distinct exceptional characters  in $A_{\alpha}$. 
\end{lem} 

\begin{rem}
	For the purpose of our computations, it is not necessary to know precisely which exceptional characters occur as constituents of the trivial source characters, because they all take the same values at the $\ell'$-elements. 
\end{rem}

\begin{longtable}{|c|c|c|}
	\caption{Trivial source characters of $\SL_2(q)$ when $2 \neq \ell \mid q-1$.}	
	\label{tab:tsmodslmidq-1}
	\\  \hline 
	\begin{tabular}{c}
		Vertices of $M$
	\end{tabular}
	& \begin{tabular}{c}
		Character $\chi_{\widehat{M}}$
	\end{tabular}
& \begin{tabular}{c}
	Block containing $M$
\end{tabular}
	\\ \hline \hline 
	
	\multirow{5}{*}{$\{1\}$}
	&$1_G + \Xi$,  
	$\St + \Xi$
	&
	$\mathbf{B}_0(G)$
	\\
	
	&
	$R_\pm(\alpha_0) + \Xi_{\alpha_0}$
	& 
	$A_{\alpha_0}$
	\\
	
	&
	$R'_\pm(\theta_0)$
	&
	$A'_{\theta_0,\pm}$
	\\
	
	&
	$R(\alpha) + \Xi_{\alpha}$
	&
	$A_{\alpha}$, ($\alpha \in [T_{\ell'}^\wedge / \equiv]$, $\alpha^2 \neq 1$)
	\\
	
	&
	$R'(\theta)$
	&
	$A'_{\theta}$, ($\theta \in [T_{\ell'}'^\wedge / \equiv]$, $\theta^2 \neq 1$)
	\\
	\hline \hline

	\multirow{3}{*}{
	\begin{tabular}{c}
		$C_{\ell^i}$ \\
		($1 \leq i < n$)
	\end{tabular}}
	&
	$1_G + \Xi_i$, $\St + \Xi_i$
	&
	$\mathbf{B}_0(G)$
	\\

	&
	$R_\pm(\alpha_0) + \Xi_{\alpha_0, i}$
	&
	$A_{\alpha_0}$
	\\
	
	&
	$R(\alpha) + \Xi_{\alpha, i}$
	&
	$A_{\alpha}$, ($\alpha \in [T_{\ell'}^\wedge / \equiv]$, $\alpha^2 \neq 1$)
	\\
	\hline \hline
	
\multirow{3}{*}{
		$C_{\ell^n}$}
&
$1_G$, $\St$
&
$\mathbf{B}_0(G)$
\\

&
$R_\pm(\alpha_0)$
&
$A_{\alpha_0}$
\\

&
$R(\alpha)$
&
$A_{\alpha}$, ($\alpha \in [T_{\ell'}^\wedge / \equiv]$, $\alpha^2 \neq 1$)
\\
\hline 	
\end{longtable}

%
%
%
%
%
%
%
%
%

\begin{proof}
The ordinary characters of the PIMs lying in blocks of defect zero are immediate from Table~\ref{tab:blockslmidq-1}, and the characters of the PIMs lying in blocks with a non-trivial cyclic defect group can also be read off from Table~\ref{tab:blockslmidq-1} using Remark~\ref{rem:tsmodulecyclicdef}(a).
\par
The trivial source $kG$-modules with a non-trivial vertex all belong to $\ell$-blocks with a non-trivial cyclic defect group. By Lemma~\ref{NG(D_1)} the module $W(\textbf{B})$ is trivial for all $\ell$-blocks $\mathbf{B}$ of~$G$, so we obtain the trivial source characters as follows.  The characters $\chi_{\widehat{M}}$ of trivial source $kG$-modules $M$ with full vertex $C_{\ell^{n}}$ are therefore directly identified using the Brauer trees from Table~\ref{tab:blockslmidq-1} and Remark~\ref{rem:tsmodulecyclicdef}(b). 
The characters $\chi_{\widehat{M}}$ of trivial source modules $M\in\TS(G;C_{\ell^{i}})$ with $0< i<n$ are obtained from the classification of the trivial source modules in blocks with cyclic defect groups in \cite{HL20}. More precisely, if $M$ belongs to $\bB_{0}(G)$ or $A_{\alpha_{0}}$ then  \cite[Theorem~5.3(b)(2) and Theorem~A.1(d)]{HL20} yield the following characters $\chi_{\widehat{M}}$, 
\[
\begin{cases}
   1_{G}+\Xi_i\text{ and }\St +~\Xi_i   & \text{in  }\bB_{0}(G), \\
   R_+(\alpha_0) \text{ and }R_-(\alpha_0)  & \text{in  }A_{\alpha_0},\\
\end{cases}
\]
and if $M$ belongs to a nilpotent block $A_{\alpha}$  $(\alpha \in [T_{\ell'}^\wedge / \equiv], \alpha^2 \neq 1)$, then $\chi_{\widehat{M}}=R(\alpha) + \Xi_{\alpha, i}$ by \cite[Theorem~7.1(a)]{KL21}. 
\end{proof}

\vspace{2mm}
\subsection{The $\ell$-blocks and trivial source characters of $N$}
\label{subsec:N}

\begin{lem}
	\label{lem:tsmodsN}
	When $2 \neq \ell \mid q-1$ the Brauer correspondents in $N$ of the $\ell$-blocks of $G$ with non-trivial defect groups are as given in Table~\ref{tab:blocksN}, where the labelling of the non-principal blocks of $N$ comes from \cite[Section 7.1.2]{BonBook}. 
	The ordinary characters $\chi_{\widehat{f(M)}}$ of the $kN$-Green correspondents $f(M)$ of  the trivial source $kG$-modules $M$ with a non-trivial vertex  are as given  in Table~\ref{tab:tsmodsN}, 
	where for each $1\leq i\leq n$, 
	\[\Xi^{N}_i = \sum\limits_{j=1}^{\pi_{q,i}} \chi_{\eta_j},\quad\text{ and }\quad  \Xi^{N}_{\alpha_0, i} = \sum\limits_{j=1}^{\pi_{q,i}} \chi_{\alpha_0\eta_j}\]
are  sums of  $\pi_{q,i}$ pairwise distinct exceptional characters  in $\Irr(\bB_{0}(N))$, resp. $\Irr(\mathcal{O}Nb_{\alpha_0})$, and  for any $\alpha \in \Irr(T_{\ell'})$ with $\alpha^2 \neq 1$, 
\[\Xi^{N}_{\alpha, i} = \sum\limits_{j=1}^{2\pi_{q,i}} \chi_{\alpha\eta_j}\]
is a sum of $2\pi_{q,i}$ pairwise distinct exceptional characters  in $\Irr(\mathcal{O}N( b_{\alpha} + b_{\alpha^{-1}}))$. 	
\end{lem}

\begin{longtable}{c||c|c|c}
	\caption{Brauer correspondents in $N$ of the $\ell$-blocks of $G$, when $2 \neq \ell \mid q-1$.}	
	\label{tab:blocksN}
	\\  
	\begin{tabular}{c} 
		Block of $G$ \\
		$\textbf{B}$
	\end{tabular}
	&	\begin{tabular}{c}
		Brauer correspondent \\
		$\textbf{b}$
	\end{tabular}
	& 	\begin{tabular}{c}
		Defect \\
		Groups
	\end{tabular}
	& 
	\begin{tabular}{c}
		Brauer Tree $\sigma(\textbf{b})$ \\
		with Type Function
	\end{tabular}
	\\ \hline \hline 			
	
	$\mathbf{B}_0(G)$ 
	& 	$\mathbf{B}_0(N)$  
	& 	$C_{\ell^n}$ 
	&	\begin{tabular}{c}
		$$ \xymatrix@R=0.0000pt@C=30pt{
			{_+} & {_-}& {_+}\\
			{\Circle} \ar@{-}[r]
			& {\CIRCLE} \ar@{-}[r] 
			&{\Circle} \\
			{^{1_N}}&{^{\Xi^N}}&{^{\varepsilon}}
		}$$\\
		{\footnotesize $\Xi^{N} := \sum\limits_{\eta \in [S_{\ell}^\wedge / \equiv] \setminus \{1\}} \chi_{\eta}$}\\
	\end{tabular}
	\\ \hline

	\begin{tabular}{c}
		$A_{\alpha_0}$\\
		\footnotesize{($\alpha_0 \in T_{\ell'}^\wedge \setminus \{1\}$}, 
		\footnotesize{$\alpha_0^2 = 1$)} 
	\end{tabular} 
	&
	\begin{tabular}{c}
		$k\otimes_{\cO}\mathcal{O}Nb_{\alpha_0}$
	\end{tabular} 
	& 	$C_{\ell^n}$ 	
	& 	\begin{tabular}{c}
		$$  \xymatrix@R=0.0000pt@C=30pt{	
			{_+} & {_-}& {_+}\\
			{\Circle}  \ar@{-}[r]
			& {\CIRCLE}  \ar@{-}[r] 
			&{\Circle} \\
			{^{\chi^+_{\alpha_0}}}&{^{\Xi^{N}_{\alpha_0}}}&{^{\chi^-_{\alpha_0}}}
		}$$\\
	{\footnotesize	$\Xi^{N}_{\alpha_0}  := \sum\limits_{\eta \in [S_{\ell}^\wedge / \equiv] \setminus \{1\}} \chi_{\alpha_0\eta}$  }\\
	\end{tabular}
	\\ \hline  
	
	\begin{tabular}{c}
		$A_{\alpha}$ \\
		\footnotesize{($\alpha \in [T_{\ell'}^\wedge / \equiv]$}, 
		\footnotesize{$\alpha^2 \neq 1$)}
	\end{tabular}
	&
	\begin{tabular}{c}
		$k\otimes_{\cO}\mathcal{O}N( b_{\alpha} + b_{\alpha^{-1}})$
	\end{tabular}
	& 	$C_{\ell^n}$ 		
	& 	\begin{tabular}{c}
		$  \xymatrix@R=0.0000pt@C=30pt{	
			{_+} & {_-}\\
			{\Circle} \ar@{-}[r] 
			& {\CIRCLE}    \\
				{^{\chi_\alpha}}&{^{\Xi^{N}_{\alpha}}}
		}$\\
	{\footnotesize	$\Xi^{N}_{\alpha}  := \sum\limits_{ \eta \in S_{\ell}^\wedge \setminus \{1\}} \chi_{\alpha\eta}$ }\\  
	\end{tabular}
	\\ \hline 
	
\end{longtable}

\begin{longtable}{|c|l|l|}
	\caption{The trivial source characters of the $kN$-Green correspondents when $2 \neq \ell \mid q-1$.}	
	\label{tab:tsmodsN}
	\\  \hline 
	\begin{tabular}{c}
		Vertices of $M$
	\end{tabular}
	& \begin{tabular}{c}
		Character $\chi_{\widehat{M}}$
	\end{tabular}
	& \begin{tabular}{c}
		Character $\chi_{\widehat{f(M)}}$ of the  \\Green Correspondent 
	\end{tabular}
	\\ \hline \hline

	
	\begin{tabular}{c}
		$C_{\ell^i}$ \\
		($1 \leq i < n$)
	\end{tabular} 
	& 	\begin{tabular}{l}
		$1_G + \Xi_i$\\ 
		$\St + \Xi_i$\\ 
		$R_\pm(\alpha_0) + \Xi_{\alpha_0, i}$ \\ 
		$R(\alpha) + \Xi_{\alpha, i}$ ($\alpha \in [T_{\ell'}^\wedge / \equiv]$, $\alpha^2 \neq 1$) 
	\end{tabular}
	& \begin{tabular}{l}
		$1_N + \Xi^N_i$\\ 
		$\varepsilon + \Xi^N_i$\\ 
		$\chi^\pm_{\alpha_0} + \Xi^{N}_{\alpha_0, i}$ \\ 
		$\chi_\alpha + \Xi^{N}_{\alpha, i}$
	\end{tabular}
	
	\\ 	\hline 	 \hline 
	
	$C_{\ell^n}$
	& 
	\begin{tabular}{l}
		$1_G$\\
		$\St$\\
		$R_\pm(\alpha_0)$ \\
		$R(\alpha)$ ($\alpha \in [T_{\ell'}^\wedge / \equiv]$, $\alpha^2 \neq 1$)
	\end{tabular}
	&
	\begin{tabular}{l}
		$1_N$ \\
		$\varepsilon$ \\
		$\chi^\pm_{\alpha_0}$ \\
		$\chi_{\alpha}$
	\end{tabular}
	\\  \hline 
	
\end{longtable}

\begin{proof}	
	
	The partitioning of the ordinary characters of $N$ into $\ell$-blocks can be determined by examining the values of the central characters of $N$ modulo $\ell$ using the character table \cite[Table 6.2]{BonBook}: 
	\begin{itemize}
		\item $\Irr(\mathbf{B}_0(N))$ contains $1_N$, $\varepsilon$ and $\chi_{\eta}$ for all $\eta \in \Irr(S_{\ell}) \setminus \{1\}$, 
		\item $\Irr(\mathcal{O}Nb_{\alpha_0})$ contains $\chi^\pm_{\alpha_0}$ and $\chi_{\alpha_0 \eta}$ for all $\eta \in \Irr(S_{\ell}) \setminus \{1\}$, and
		\item for each $\alpha \in \Irr(T_{\ell'})$ with $\alpha^2 \neq 1$, there exists a block containing $\chi_{\alpha \eta}$ for all $\eta \in \Irr(S_{\ell})$.
	\end{itemize}
	
	All blocks of $N$ have maximal normal defect groups so the Brauer trees are star-shaped with the exceptional vertex in the middle (see e.g. \cite[Proposition 6.5.4]{BensonBookI}). Since $1_{N}$, $\varepsilon$, $\chi_{\alpha_0}$ and $\chi_{\alpha}$ for $\alpha \in \Irr(T_{\ell'})$ are all $\ell$-rational, these characters are non-exceptional. This fully determines the shape of the Brauer trees given in  Table~\ref{tab:blocksN}. 
	The type function of $\bB_{0}(N)$ is immediate as the trivial character is positive. It is also easy to see from the character table of $N$ that $\chi^\pm_{\alpha_0}$ and $\chi_{\alpha}$ (for each $\alpha \in \Irr(T_{\ell'})$ with $\alpha^2 \neq 1$) are positive because $\alpha_0$ and $\alpha$ are $\ell'$-characters so they take the value $1$ on all $\ell$-elements. This determines the type function for all remaining blocks. 
	\par
	The characters of the trivial source $kN$-modules in  the third column of Table~\ref{tab:tsmodsN}  are obtained via Remark~\ref{rem:tsmodulecyclicdef},  \cite[Theorem~5.3(b)(2) and Theorem~A.1(d)]{HL20}  and \cite[Theorem~7.1(a)]{KL21} exactly as in the proof of Lemma~\ref{lem:tsmodslmidq-1}. 
	\par
	The correspondence between the characters  $\chi_{\widehat{M}}$ of the trivial source $kG$-modules with a non-trivial vertex and the characters $\chi_{\widehat{f(M)}}$ of their $kN$-Green correspondents is clear. Indeed, firstly the Green correspondence and Brauer correspondence commute. Secondly,  in all cases we have a Morita equivalence between $\bB$ and $\bb$ that  respects the labelling of characters by the discussion in \cite[Section 9.3]{BonBook}, and  thirdly the Morita equivalence is a source-algebra equivalence as $W(\bB)=k$ and the type function also respects the labelling, so trivial source modules are preserved. 
\end{proof}

\vspace{2mm}

\subsection{The trivial source character table of $G$}
\label{sec:tschartablelmidq-1}

We fix one last piece of notation for the tables below and set  
\[\kappa := \left\{ \begin{array}{ll} 0 & \mbox{ if } q \equiv 1 \pmod{4} \\ 1 & \mbox{ if } q \equiv 3 \pmod{4} .\end{array} \right. \]

\vspace{2mm}

\begin{thm}\label{thm:l|q-1}
	Suppose that $2 \neq \ell \mid q-1$. With notation as in Notation \ref{rem:lmidq-1}, 
	the trivial source character table $\Triv_{\ell}(G)= [T_{i,v}]_{1\leq i,v\leq n+1}$ is given as follows:
	\begin{enumerate}[{\,\,\rm(a)}] \setlength{\itemsep}{2pt}
		\item $T_{i,v} = \mathbf{0}$ if $v > i$;
		\item the matrices $T_{i,1}$ are as given in Table \ref{tab:lmidq-1T_i1} for each $1 \leq i \leq n+1$;
	        \item the matrices $T_{i,i}$ are as given in Table \ref{tab:lmidq-1T_ii} for each $2 \leq i \leq n+1$; and 
		\item $T_{i,v} = T_{i,i}$ for all $2 \leq v < i \leq n+1$.
	\end{enumerate} 
\end{thm}

\noindent

\begin{landscape}
	
	\renewcommand*{\arraystretch}{1.5}
	\begin{longtable}{c|c||c|c|c|c|}
		\caption{$T_{i,1}$ for $1 \leq i \leq n+1$.}	
		\label{tab:lmidq-1T_i1}
		\\  \cline{2-6}
		
		&&
		\begin{tabular}{c}
			\textbf{$\varepsilon I_2$ } \vspace{-.8ex}  \\
			\textbf{	$ \varepsilon \in \{ \pm 1\}$}
		\end{tabular}
		& 
		\begin{tabular}{c}
			$\mathbf{d}(a)$ \vspace{-.8ex}  \\
			\footnotesize{$ a \in \Gamma_{\ell'}$}
		\end{tabular}
		&
		\begin{tabular}{c}
			$\mathbf{d}'(\xi)$ \vspace{-.8ex}  \\
			\footnotesize{$ \xi \in \Gamma'_{\ell'}$}
		\end{tabular}
		&
		\begin{tabular}{c}
			\textbf{$ \varepsilon u_{\tau}$} \vspace{-.8ex}  \\
			\textbf{	$ \varepsilon, \tau \in \{ \pm 1\}$}
		\end{tabular}
		\\ \hline \hline

		&$1_G + \Xi$
		& $1 + (q+1)\pi_q$
		& $1 + 2\pi_q$
		& $1$
		& $1 + \pi_q$
		\\ \cline{2-6}

		&$\St +  \Xi$
		& $q + (q+1)\pi_q$
		& $1 + 2\pi_q$
		& $-1$
		& $\pi_q$
		\\ \cline{2-6}
		
		&\begin{tabular}{c}
			$R_\pm(\alpha_0) + \Xi_{\alpha_0}$
		\end{tabular}
		& $\varepsilon^{\kappa}\left(\frac{q+1}{2}\right)(1 + 2\pi_q)$
		& $\alpha_0(a)(1 + 2\pi_q)$ 
		& $0$
		& $\varepsilon^{\kappa} \left(\frac{1 \pm \tau \sqrt{q_0}}{2} + \pi_q \right) $
		\\ \cline{2-6}
		
		$T_{1,1}$
		&\begin{tabular}{c}
			$R(\alpha) + \Xi_{\alpha}$\vspace{-1ex} 	\\
			\tiny{$\left(\alpha \in [T_{\ell'}^\wedge / \equiv], \alpha^2 \neq 1\right)$}
		\end{tabular}
		& $\alpha(\varepsilon) (q+1)(1 + 2\pi_q)$
		& $(\alpha(a) + \alpha(a^{-1}))(1 + 2\pi_q)$
		& $0$
		& $\alpha(\varepsilon)(1 + 2\pi_q)$
		\\ \cline{2-6}
		
		&$R'_\pm(\theta_0)$
		& $\varepsilon^{\kappa+1} \left( \frac{q-1}{2}\right) $
		& $0$
		& $-\theta_0(\xi)$
		& $\varepsilon^{\kappa + 1} \left(\frac{-1 \pm \tau \sqrt{q_0}}{2} \right) $
		\\ \cline{2-6}

		&\begin{tabular}{c}
			$R'(\theta)$ \vspace{-.8ex} \\
			\tiny{$\left(\theta \in [T_{\ell'}'^\wedge / \equiv], \theta^2 \neq 1\right)$}
		\end{tabular}
		& $\theta(\varepsilon) (q-1)$
		& $0$
		& $-\theta(\xi) - \theta(\xi^{-1})$
		& $-\theta(\varepsilon)$
		\\ 
		\hline
		\hline

		\multirow{4}{*}[-2ex]{
			\begin{tabular}{c}
				$T_{ i,1}$\\
				$(1 \leq i < n)$
			\end{tabular}}
		&	$1_G + \Xi_{i-1}$
		& $1 + (q+1)\pi_{q,i-1}$
		& $1 + 2\pi_{q,i-1}$
		& $1$
		& $1 + \pi_{q,i-1}$
		\\ \cline{2-6}

		&		
		$\St +  \Xi_{i-1}$
		& $q + (q+1)\pi_{q,i-1}$
		& $1 + 2\pi_{q,i-1}$
		& $-1$
		& $\pi_{q,i-1}$
		\\ \cline{2-6}

		&\begin{tabular}{c}
			$R_\pm(\alpha_0) + \Xi_{\alpha_0, i-1}$
		\end{tabular}
		& $\varepsilon^{\kappa}\left(\frac{q+1}{2}\right)(1 + 2\pi_{q,i-1})$
		& $\alpha_0(a)(1 + 2\pi_{q,i-1})$
		& $0$
		& $\varepsilon^{\kappa} \left(\frac{1 \pm \tau \sqrt{q_0}}{2} + \pi_{q,i-1} \right) $
		\\ \cline{2-6}

		&\begin{tabular}{c}
			$R(\alpha) + \Xi_{\alpha,i-1}$	\vspace{-1ex} 		\\
			\tiny{$\left(\alpha \in [T_{\ell'}^\wedge / \equiv], \alpha^2 \neq 1\right)$}
		\end{tabular}
		& $\alpha(\varepsilon)(q+1)(1 + 2\pi_{q,i-1})$
		& $(\alpha(a) + \alpha(a^{-1}))(1 + 2\pi_{q,i-1})$
		& $0$
		& $\alpha(\varepsilon) (1 + 2\pi_{q,i-1})$
		\\ 
		\hline  \hline 
		
	\multirow{4}{*}[-2ex]{
			$T_{n+1,1}$}
		&$1_G$
		& $1$
		& $1$
		& $1$
		& $1$
		\\ \cline{2-6}

		&$\St$
		& $q$
		& $1$
		& $-1$
		& $0$
		\\ \cline{2-6}

		&\begin{tabular}{c}
			$R_\pm(\alpha_0)$ 
		\end{tabular}
		& $\varepsilon^{\kappa}\left(\frac{q+1}{2}\right)$
		& $\alpha_0(a)$
		& $0$
		& $\varepsilon^{\kappa} \left(\frac{1 \pm \tau \sqrt{q_0}}{2}\right)$
		\\ \cline{2-6}

		&\begin{tabular}{c}
			$R(\alpha)$ \vspace{-1ex} \\
			\tiny{$\left(\alpha \in [T_{\ell'}^\wedge / \equiv], \alpha^2 \neq 1\right)$}
		\end{tabular}
		& $\alpha(\varepsilon)(q+1)$
		& $\alpha(a) + \alpha(a^{-1})$
		& $0$
		& $\alpha(\varepsilon)$
		\\ \hline

	\end{longtable}
\end{landscape}


\renewcommand*{\arraystretch}{1.5}
\begin{longtable}{|c||c|c|c|c|}
	\caption{$T_{i,i}$ for $2 \leq  i \leq n+1$.} 
	\label{tab:lmidq-1T_ii}
	\\  \cline{1-5}
	
	&
	\begin{tabular}{c}
		\textbf{$\varepsilon I_2$ } \\
		\textbf{	$ \varepsilon \in \{ \pm 1\}$}
	\end{tabular}
	& 
	\begin{tabular}{c}
		$\mathbf{d}(a)$ \\
		\footnotesize{$a \in \Gamma_{\ell'}$}
	\end{tabular}
	&
	\begin{tabular}{c}
		\textbf{$\sigma_{\tau}$} \\
		\textbf{	$\tau \in \{ \pm 1\}$}
	\end{tabular}
	\\ \hline \hline 			
	
	$1_G + \Xi_{i-1}$
	& $1 + 2\pi_{q, i-1}$
	& $1 + 2\pi_{q, i-1}$
	& $1$
	\\ \hline

	$\St +  \Xi_{i-1}$
	& $1 + 2\pi_{q, i-1}$
	& $1 + 2\pi_{q, i-1}$
	& $-1$
	\\ \hline

	\begin{tabular}{c}
		$R_\pm(\alpha_0) + \Xi_{\alpha_0, i-1}$
	\end{tabular}
	& $\varepsilon^{\kappa}\left(1 + 2\pi_{q, i-1}\right)$
	&  $\alpha_0(a)(1 + 2\pi_{q, i-1})$
	&  $\pm \tau \sqrt{-1^{\kappa}}$
	\\ \hline

	\begin{tabular}{c}
		$R(\alpha) + \Xi_{\alpha, i-1}$		\vspace{-1ex} 	\\
		\tiny{$\left(\alpha \in [T_{\ell'}^\wedge / \equiv], \alpha^2 \neq 1\right)$}
	\end{tabular}
	& $2\alpha(\varepsilon) (1 + 2\pi_{q, i-1})$
	& $(\alpha(a) + \alpha(a^{-1}))(1 + 2\pi_{q, i-1})$
	& $0$
	\\ \hline 
	
	\hline

\end{longtable}

\begin{proof}
By Convention~\ref{conv:tsctbl} the labels for the rows of $\Triv_{\ell}(G)$ are the ordinary characters of the trivial source $kG$-modules determined in Lemma~\ref{lem:tsmodslmidq-1}. 
\begin{enumerate}[(a)]
   \item It follows from Remark~\ref{rem:tsctbl}(d) and Notation~\ref{rem:lmidq-1} that $T_{i,v} = \mathbf{0}$ whenever $v > i$. 
   \item By Remark~\ref{rem:tsctbl}(d), the values in  $T_{i,1}$  for $1 \leq i \leq n+1$ (Table \ref{tab:lmidq-1T_i1})  are calculated by evaluating the character of each trivial source module given in Table \ref{tab:tsmodslmidq-1} at the relevant representative of an $\ell'$-conjugacy class of $G$ using the character table of $G$ (Table \ref{tab:CTBLsl2}). 
\item By Convention~\ref{conv:tsctbl},  the values in $T_{i,i}$ for $2 \leq i \leq n+1$ (Table \ref{tab:lmidq-1T_ii}) are given by the values of the species~$\tau_{Q_{i},s}^{G}$, with $s$ running through  $[\overline{N}_{i}]_{\ell'}$ (identified here with $[N]_{\ell'}$), evaluated at the  trivial source modules $[M]\in\TS(G;Q_{i})$.  By definition of the species and Proposition~\ref{prop:omnts}(d) these are calculated by evaluating the ordinary character of the $kN$-Green correspondent given in Table~\ref{tab:tsmodsN} of the trivial source $kG$-module labelling the relevant row, at the representatives of the $\ell'$-conjugacy classes of $N$ using the character table of $N$ given in  \cite[Table 6.2]{BonBook}.
    \item For  each $2 \leq v \leq i \leq n+1$, by Convention~\ref{conv:tsctbl}, the matrix $T_{i,v}$ consists of the values of the species~$\tau_{Q_{v},s}^{G}$, with $s$ running through  $[\overline{N}_{v}]_{\ell'}$  (identified here with $[N]_{\ell'}$), evaluated at the trivial source modules $[M]\in\TS(G;Q_{i})$. However, by definition of the species, $\tau_{Q_{v},s}^{G}([M])=\chi_{\widehat{M[Q_v]}}(s)$ and  Lemma~\ref{lem:equalgreencorrs} shows that $M[Q_v]$ is the $kN$-Green correspondent of $M$. Hence $T_{i,v} = T_{i,i}$ for all $2 \leq v < i \leq n+1$. 
\end{enumerate}
\end{proof}

\normalsize{}


\vspace{6mm}
\section{$\SL_2(q)$ with  $2\neq\ell\mid (q+1)$}
\label{sec:lmidq+1} 

In this section we assume that $G = \SL_2(q)$ with $q$ odd and $2 \neq \ell \mid q+1$.   By Lemma~\ref{lem:omnibusl=2}(a)(i) the Sylow $\ell$-subgroups of $G$ are also cyclic and we proceed as in Section~\ref{sec:lmidq-1}. 

\begin{nota}\label{rem:lmidq+1}
	We now adopt notation analogous to Notation~\ref{rem:lmidq-1}, in order to describe $\Triv_{\ell}(G)$ according to Convention~\ref{conv:tsctbl}. Here, we fix $Q_{n+1}:=S'_{\ell}\cong C_{\ell^{n}}$. Then, as before, for each $1\leq i\leq n$ we let $Q_{i}$ denote the unique cyclic subgroup of $Q_{n+1}$ of order $\ell^{i-1}$ and  
	\[\{1\} = Q_1 \leq \dots  \leq Q_{n+1} \in \Syl_{\ell}(G)\]
	is our fixed set of representatives for the conjugacy classes of $\ell$-subgroups of $G$.  We keep the same set of representatives for  the $\ell'$-conjugacy classes of $G$:
	$$[G]_{\ell'}:=\{\pm I_{2}\}\cup \{\mathbf{d}(a)\mid a\in\Gamma_{\ell'}\}\cup  \{\mathbf{d}'(\xi)\mid \xi\in\Gamma'_{\ell'}\}\cup\{\varepsilon u_{\tau}\mid \varepsilon,\tau\in\{\pm1\}\}\,,$$
	and we fix the following set of representatives for the $\ell'$-conjugacy classes of $N'$: 
	$$[N']_{\ell'}:=\{\pm I_{2}\}\cup  \{\mathbf{d}'(\xi)\mid \xi\in\Gamma'_{\ell'}\}\cup\{\sigma_{\tau}\mid \tau\in\{\pm1\}\}\,.$$
	By the same arguments as in Notation \ref{rem:lmidq-1}, for any $2\leq v\leq n+1$ and any $1\leq i\leq n+1$ we can label the columns of $T_{i,v}$ by this fixed set of representatives for  the $\ell'$-conjugacy classes of $N'$. Finally, for each $0\leq  i \leq n$ we fix
	\[
	\pi'_{q,i} :=\frac{(q+1)_\ell \cdot(1- \ell^{-i})}{2}\,,
	\]
	\[
\pi''_{q,i} := \frac{(q+1)_{\ell} - 1}{2} - \pi'_{q, i} = \frac{(q+1)_\ell \cdot\ell^{-i} - 1}{2}\,,
\]
	and let $\pi'_q := \pi'_{q,n}$. \\
Again, these numbers naturally come from the classification of the trivial source modules in blocks with cyclic defect goups in \cite{HL20}.	
\end{nota}

\vspace{2mm}
\subsection{The $\ell$-blocks and trivial source characters of $G$}

\begin{lem} {\label{lem:blockslmidq+1}}
	When $2 \neq \ell \mid q+1$ the $\ell$-blocks of $G$, their defect groups and their Brauer trees with type function are as given in Table~\ref{tab:blockslmidq+1}.
\end{lem}

\begin{longtable}{c||c|c|c}
	\caption{The $\ell$-blocks of $\SL_2(q)$ when $2 \neq \ell \mid q+1$.}	
	\label{tab:blockslmidq+1}
	\\  
	Block
	&	\begin{tabular}{c}
		Number of Blocks \\
		(Type)
	\end{tabular}
	& 	\begin{tabular}{c}
		Defect \\
		Groups
	\end{tabular}
	& 
	\begin{tabular}{c}
		Brauer Tree with Type\\
		Function /  $\Irr(\bB)$
	\end{tabular}
	\\ \hline \hline 
	
	$\mathbf{B}_0(G)$ 
	& 	\begin{tabular}{c}
		1 \\
		\small{(Principal)}
	\end{tabular}
	& 	$C_{\ell^n}$ 
	&	\begin{tabular}{c}
		$$ \xymatrix@R=0.0000pt@C=30pt{	
			{_+} & {_-}& {_+}\\
			{\Circle} \ar@{-}[r]  
			& {\Circle} \ar@{-}[r] 
			&{\CIRCLE} \\
			{^{1_G}}&{^{\St}}&{^{\Xi'}}
		}$$ \\
	{\footnotesize $\Xi':= \sum\limits_{\eta \in [S_{\ell}'^\wedge / \equiv] \setminus \{1\}} R'(\eta)$}\\
	\end{tabular}
	\\ \hline

	\begin{tabular}{c}
		$A'_{\theta_0}$ \\
		\small{($\theta_0 \in T_{\ell'}'^\wedge \setminus \{1\}, \theta_0^2 = 1$)}
	\end{tabular} 
	& 	\begin{tabular}{c}
		1 \\
		\small{(Quasi-isolated)}
	\end{tabular}
	& 	$C_{\ell^n}$ 		
	& 	\begin{tabular}{c}
		$$  \xymatrix@R=0.0000pt@C=30pt{	
			{_-} & {_+}& {_-}\\
			{\Circle}  \ar@{-}[r]
			& {\CIRCLE}  \ar@{-}[r]
			&{\Circle} \\
			{^{R_+'(\theta_0)}}&{^{\Xi'_{\theta_0}}}&{^{R_-'(\theta_0)}}
		}$$ \\
	{\footnotesize $\Xi'_{\theta_0}:= \sum\limits_{\eta \in [S_{\ell}'^\wedge / \equiv] \setminus \{1\}} R'(\theta_0\eta)$}\\
	\end{tabular}
	\\ \hline

	\begin{tabular}{c}
		$A'_\theta$  \\
		\small{($\theta \in [T_{\ell'}'^\wedge / \equiv], \theta^2 \neq 1$) }
	\end{tabular}
	& \begin{tabular}{c}
		$\frac{(q+1)_{\ell'} - 2}{2}$ \\
		\small{(Nilpotent)}
	\end{tabular}
	& $C_{\ell^n}$ 		
	& 	\begin{tabular}{c}
		$$  \xymatrix@R=0.0000pt@C=50pt{	
			{_-} & {_+}\\
			{\Circle}  \ar@{-}[r]
			& {\CIRCLE}    \\
			{^{R'(\theta)}}&{^{\Xi'_{\theta}}}
		} $$ \\
	{\footnotesize $\Xi'_{\theta}:= \sum\limits_{\eta \in S_{\ell}'^\wedge \setminus \{1\}} R'(\theta \eta)$}\\
	\end{tabular}
	\\ \hline 
	
	\begin{tabular}{c}
		$A_{\alpha_0,\pm}$ \\
		\small{	($\alpha_0 \in T_{\ell'}^\wedge \setminus \{1\}, \alpha_0^2 = 1$) }
	\end{tabular} 
	& \begin{tabular}{c}
		2 \\
		\small{(Defect zero)}
	\end{tabular}
	& 	$\{1\}$ 
	& 	\begin{tabular}{c}
		$\Irr(A_{\alpha_0,+}) = \{R_{+}(\alpha_0)\}$ \\
		$\Irr(A_{\alpha_0, -}) = \{R_{-}(\alpha_0)\}$ 
	\end{tabular}
	\\ \hline

	\begin{tabular}{c}
		$A_\alpha$ \\
		\small{($\alpha \in [T_{\ell'}^\wedge / \equiv], \alpha^2 \neq 1$)} 
	\end{tabular}
	& 	\begin{tabular}{c}
		$\frac{(q-1)_{\ell'} - 2}{2}$ \\
		\small{(Defect zero)}
	\end{tabular}
	& $\{1\}$ 
	& $\Irr(A_\alpha) = \{R(\alpha)\}$	
	\\ \hline 	
\end{longtable}

\begin{proof}
	As in Lemma~\ref{lem:blockslmidq-1}, the information in Table \ref{tab:blockslmidq+1} can be found in \cite[Chapters~8 and~9]{BonBook}, except for the type functions of the Brauer trees.  The type function of $\sigma(\bB_{0}(G))$ is immediate as the trivial character is positive. Here, $R_\pm'(\theta_0)$ and $R'(\theta)$ are negative for each $\theta \in T_{\ell'}'^\wedge \setminus \{1\}$, $\theta^2 \neq 1$, because $\theta_0$ and $\theta$ are $\ell'$-characters and take the value $1$ on $\ell$-elements. This determines the type functions of the remaining blocks. 
	%
\end{proof}

\begin{lem}
	\label{lem:verticeschainandtsmodslmidq+1}
	When $2 \neq \ell \mid q+1$  the ordinary characters $\chi_{\widehat{M}}$ of the trivial source $kG$-modules~$M$  are as given in Table~\ref{tab:tsmodslmidq+1}, where for each $1 \leq i < n$, 
		\[ \Xi'_{i}   := \sum\limits_{j=1 }^{\pi'_{q,i}} R'(\eta_j)\quad\text{ and }\quad 
		\Xi'_{\theta_0, i}   := \sum\limits_{j=1}^{\pi'_{q,i}} R'(\theta_0\eta_j)  
	\]
	are sums of  $\pi'_{q,i}$ pairwise distinct exceptional characters  in $\Irr(\bB_{0}(G))$, resp. $\Irr(A'_{\theta_{0}})$, and  for any $\theta \in \Irr(T'_{\ell'})$ with $\theta^2 \neq 1$,  
	\[\Xi'_{\theta, i}   := \sum\limits_{j=1}^{2\pi'_{q,i} } R'(\theta\eta_j) \]		
		is a sum of $2\pi'_{q,i}$ pairwise distinct exceptional characters  in $\Irr(A'_{\theta})$.
\end{lem}

\begin{longtable}{|c|c|c|}
	\caption{Trivial source characters of $\SL_2(q)$ when $2 \neq \ell \mid q+1$.}	
	\label{tab:tsmodslmidq+1}
	\\  \hline 
	\begin{tabular}{c}
		Vertices of $M$
	\end{tabular}
	& \begin{tabular}{c}
		Character $\chi_{\widehat{M}}$
	\end{tabular}
	& \begin{tabular}{c}
		Block containing $M$
	\end{tabular}
	\\ \hline \hline 
	
	\multirow{5}{*}{$\{1\}$}
	&	$1_G + \St$, $\St + \Xi'$ 
	&
	$\mathbf{B}_0(G)$
	\\
	
	&$R'_\pm(\theta_0) + \Xi'_{\theta_0}$
	&$A'_{\theta_0}$
	\\
	
	&$R_\pm(\alpha_0)$
	&$A_{\alpha_0,\pm}$
	\\
	
	&$R'(\theta) + \Xi'_{\theta}$
	&$A'_{\theta}$, ($\theta \in [T_{\ell'}'^\wedge / \equiv]$, $\theta^2 \neq 1$)
	\\

	&$R(\alpha)$
	&$A_{\alpha}$, ($\alpha \in [T_{\ell'}^\wedge / \equiv]$, $\alpha^2 \neq 1$)
	\\
	
	\hline \hline 

	\multirow{3}{*}{
		\begin{tabular}{c}
			$C_{\ell^i}$ \\
			($1 \leq i < n$)
	\end{tabular}}
	&$1_G + \St + \Xi'_i$, 	
	 $\St + \Xi'_i$			
	&$\mathbf{B}_0(G)$
	\\

	&$\Xi'_{\theta_0, i}$, 
	 $\Xi'_{\theta_0, i}$  
	&$A'_{\theta_0}$
	\\
	
	&$\Xi'_{\theta, i}$ 
	&$A'_{\theta}$, ($\theta \in [T_{\ell'}'^\wedge / \equiv]$, $\theta^2 \neq 1$) 
	\\
	\hline \hline
	
	\multirow{3}{*}{
		$C_{\ell^n}$}
	&$1_G$, $\Xi'$
	&$\mathbf{B}_0(G)$
	\\

	&$\Xi'_{\theta_0}$, $\Xi'_{\theta_0}$
	&$A'_{\theta_0}$
	\\
	
	&$\Xi'_{\theta}$
	&$A'_{\theta}$, ($\theta \in [T_{\ell'}'^\wedge / \equiv]$, $\theta^2 \neq 1$)
	\\
	\hline 	
\end{longtable}

%
%
%
%
%
%
%
%
%
%
%

\begin{proof}
	The arguments are as in Lemma \ref{lem:tsmodslmidq-1}. We can read the ordinary characters of all PIMs from Table \ref{tab:blockslmidq+1} (with the help of Remark~\ref{rem:tsmodulecyclicdef}(a) for the PIMs in $\ell$-blocks not of defect zero). Since $W(\mathbf{B})$ is trivial for all cyclic $\ell$-blocks of $G$, the characters of trivial source modules with full vertex can be obtained from Table~\ref{tab:blockslmidq+1} and Remark~\ref{rem:tsmodulecyclicdef}(b). 
	
	Once again the trivial source modules $M\in\TS(G;C_{\ell^{i}})$ with $1< i<n$ require the most work. Such modules in $\bB_{0}(G)$ are given by \cite[Theorem~5.3(b)(2) and (3)]{HL20}, and those in $A'_{\theta_0}$ are given by \cite[Theorem~5.3(b)(7)]{HL20}. Their characters can then be identified using \cite[Theorem~A.1(d)]{HL20}. Finally, if $M\in\TS(G;C_{\ell^{i}})$ belongs to a nilpotent block $A'_{\theta}$  ${(\theta \in [T_{\ell'}'^\wedge / \equiv]}, {\theta^2 \neq 1})$, then $\chi_{\widehat{M}}=R(\alpha) + \Xi_{\alpha, i}$ by \cite[Theorem~7.1(a)]{KL21}. 
\end{proof}

\vspace{2mm}
\subsection{The $\ell$-blocks and trivial source characters of $N'$}
\label{subsec:N'}
%
%


\begin{lem}
	\label{lem:tsmodsN'}
	When $2 \neq \ell \mid q+1$ the Brauer correspondents in $N'$ of the $\ell$-blocks of $G$ with non-trivial defect groups are as given in Table \ref{tab:blocksN'}, where the labelling of the non-principal blocks of $N'$ comes from \cite[Section 7.1.2]{BonBook}. 	The ordinary characters $\chi_{\widehat{f(M)}}$ of the $kN'$-Green correspondents $f(M)$ of the trivial source $kG$-modules $M$ with a non-trivial vertex are as given in Table~\ref{tab:tsmodsN'}, where for each $1 \leq i < n$,
	\[ \Xi^{N'}_i   := \sum\limits_{j=1}^{\pi''_{q,i}} \chi'_{\eta}, 
	\quad\text{ and }\quad 
	\Xi^{N'}_{\theta_0, i}  := \sum\limits_{j=1}^{\pi''_{q,i}} \chi'_{\theta_0\eta},  
	\]
	are sums of $\pi''_{q,i}$ pairwise distinct exceptional characters in $\Irr(\bB_0(N'))$, resp. $\Irr(\mathcal{O}N'b'_{\theta_0})$, and for any $\theta \in \Irr(T'_{\ell'})$ with $\theta^2 \neq 1$, 
	\[\Xi^{N'}_{\theta, i}   := \sum\limits_{j=1}^{2\pi''_{q,i}} \chi'_{\theta\eta} \]
	is a sum of $2 \pi''_{q,i}$ pairwise distinct exceptional characters of $\Irr(\mathcal{O}N'(b'_{\theta} + b'_{\theta^{-1}}))$. 
\end{lem}

\begin{longtable}{c||c|c|c}
	\caption{Brauer correspondents in $N'$ of the $\ell$-blocks of $G$, when $2 \neq \ell \mid q+1$.}	
	\label{tab:blocksN'}
	\\  
	\begin{tabular}{c} 
		Block of $G$ \\
		$\textbf{B}$
	\end{tabular}
	&	\begin{tabular}{c}
		Brauer correspondent \\
		$\textbf{b}$
	\end{tabular}
	& 	\begin{tabular}{c}
		Defect \\
		Groups
	\end{tabular}
	& 
	\begin{tabular}{c}
		Brauer Tree $\sigma(\textbf{b})$ \\
		with Type Function
	\end{tabular}
	\\ \hline \hline 			
	
	$\mathbf{B}_0(G)$ 
	& 	$\mathbf{B}_0(N')$  
	& 	$C_{\ell^n}$ 
	&	\begin{tabular}{c}
		$$ \xymatrix@R=0.0000pt@C=30pt{
			{_+} & {_-}& {_+}\\
			{\Circle} \ar@{-}[r]
			& {\CIRCLE} \ar@{-}[r] 
			&{\Circle} \\
			{^{1_{N'}}}&{^{\Xi^{N'}}}&{^{\varepsilon'}}
		}$$\\
		{\footnotesize $\Xi^{N'}:= \sum\limits_{\eta \in [S_{\ell}'^\wedge / \equiv] \setminus \{1\}} \chi'_{\eta}$}\\
	\end{tabular}
	\\ \hline

	\begin{tabular}{c}
		$A'_{\theta_0}$\\
		\small{($\theta_0 \in {T'_{\ell'}}^\wedge \setminus \{1\}$}, \small{$\theta_0^2 = 1$)} 
	\end{tabular} 
	&
	\begin{tabular}{c}
		$k\otimes_{\cO}\mathcal{O}N'b'_{\theta_0}$
	\end{tabular} 
	& 	$C_{\ell^n}$ 	
	& 	\begin{tabular}{c}
		$$  \xymatrix@R=0.0000pt@C=30pt{	
			{_+} & {_-}& {_+}\\
			{\Circle}  \ar@{-}[r]
			& {\CIRCLE}  \ar@{-}[r] 
			&{\Circle} \\
			{^{{\chi'^{+}_{\theta_0}}}}&{^{\Xi^{N'}_{\theta_0}}}&{^{{\chi'^-_{\theta_0}}}}
		}$$\\
	{\footnotesize $\Xi^{N'}_{\theta_0}:= \sum\limits_{\eta \in [S_{\ell}'^\wedge / \equiv] \setminus \{1\}} \chi'_{\theta_0 \eta}$}\\
	\end{tabular}
	\\ \hline  
	
	\begin{tabular}{c}
		$A'_{\theta}$ \\
		\small{($\theta \in [{T'_{\ell'}}^\wedge / \equiv]$}, \small{$\theta^2 \neq 1$)}
	\end{tabular}
	&
	\begin{tabular}{c}
		$k\otimes_{\cO}\mathcal{O}N'( b'_{\theta} + b'_{\theta^{-1}})$
	\end{tabular}
	& 	$C_{\ell^n}$ 		
	& 	\begin{tabular}{c}
		$  \xymatrix@R=0.0000pt@C=30pt{	
			{_+} & {_-}\\
			{\Circle} \ar@{-}[r] 
			& {\CIRCLE}    \\
			{^{\chi'_\theta}}&{^{\Xi^{N'}_{\theta}}}
		}$\\
	{\footnotesize $\Xi^{N'}_{\theta}:= \sum\limits_{\eta \in S_{\ell}'^\wedge \setminus \{1\}} \chi'_{\theta \eta}$}\\
	\end{tabular}
	\\ \hline 
	
\end{longtable}

\newpage 
\begin{longtable}{|c|l|l|}
	\caption{The trivial source characters of the $kN'$-Green correspondents when $2 \neq \ell \mid q+1$.}	
	\label{tab:tsmodsN'}
	\\  \hline 
	\begin{tabular}{c}
		Vertices of $M$
	\end{tabular}
	& \begin{tabular}{c}
		Character $\chi_{\widehat{M}}$
	\end{tabular}
	& \begin{tabular}{c}
		Character $\chi_{\widehat{f(M)}}$ of the \\
		Green correspondent 
	\end{tabular}
	\\ \hline \hline

	
	\begin{tabular}{c}
		$C_{\ell^i}$ \\
		($1 \leq i < n$)
	\end{tabular} 
	& 	\begin{tabular}{l}
		$1_G + \St + \Xi'_i$ \\ 
		$\St + \Xi'_i$ \\ 
		$\Xi'_{\theta_0, i}$ \\ 
		$\Xi'_{\theta_0, i}$ \\ 
		$\Xi'_{\theta, i}$ ($\theta \in [{T'_{\ell'}}^\wedge / \equiv]$, $\theta^2 \neq 1$) 
	\end{tabular}
	& \begin{tabular}{l}
		$1_{N'} + \Xi^{N'}_i$\\ 
		$\varepsilon' + \Xi^{N'}_i$\\ 
		$\chi'^{+}_{\theta_0} + \Xi^{N'}_{\theta_0, i}$ \\ 
		$\chi'^{-}_{\theta_0} + \Xi^{N'}_{\theta_0, i}$ \\ 
		$\chi'_\theta + \Xi^{N'}_{\theta, i}$ 
	\end{tabular}
	
	\\ 	\hline \hline 
	
	$C_{\ell^n}$
	& 
	\begin{tabular}{l}
		$1_G$ \\
		$\Xi'$ \\
		$\Xi'_{\theta_0}$ \\
		$\Xi'_{\theta_0}$ \\
		$\Xi'_{\theta}$ ($\theta \in [T_{\ell'}'^\wedge / \equiv]$, $\theta^2 \neq 1$)
	\end{tabular}
	&
	\begin{tabular}{l}
		$1_{N'}$ \\
		$\varepsilon'$ \\
		$\chi'^{+}_{\theta_0}$ \\
		$\chi'^{-}_{\theta_0}$ \\
		$\chi'_\theta$
	\end{tabular}
	\\  \hline 	
	
\end{longtable}

\begin{proof}
	The distribution of the characters of $N'$ into blocks can be determined by examining the values of the central characters modulo $\ell$ using the character table of $N'$ \cite[Table 6.3]{BonBook}: 
	\begin{itemize}
		\item the principal block contains $1_{N'}$, $\varepsilon'$ and $\chi'_{\eta}$ for each $\eta \in S_{\ell}^{'\wedge} \setminus \{1\}$,
		\item the second block contains $\chi'^\pm_{\theta_0}$ and $\chi'_{\theta_0 \eta}$ for each $\eta \in S_{\ell}^{'\wedge} \setminus \{1\}$, and 
		\item for each $\theta \in T_{\ell'}^{\wedge}$ with $\theta^2 \neq 2$, there is a block containing the characters $\chi'_{\theta \eta}$ for all $\eta \in S_{\ell}^{'\wedge}$.		
	\end{itemize} 

	As for $N$, every block of $N'$ has maximal normal defect groups so all Brauer trees are star-shaped with exceptional character in the centre. As $1_{N'}$, $\varepsilon'$ and $\chi'^\pm_{\theta_0}$ are rational valued, they are in particular $\ell$-rational characters and so they are not exceptional. Furthermore, $\chi'_{\theta}$ is $\ell$-rational (and therefore non-exceptional) only if $\theta \in \Irr(T')$ is an $\ell'$-character. This determines the shape of the Brauer trees given in the table. Since the trivial character is positive, the type function of the principal block is immediate. Since $\theta_0$ and $\theta \in \Irr(T'_{\ell'})$ with $\theta^2 \neq 2$ are linear $\ell'$-characters, they take the value $1$ on $\ell$-elements. It follows from the character table of $N'$ that $\chi'^\pm_{\theta_0}$ and $\chi'_{\theta}$ are positive for every $\theta \in \Irr(T_{\ell'})$. 
	
	The module $W(\textbf{b})=W(\bB)$ is trivial for each block $\textbf{b}$ of $N'$ (see §\ref{subsec:CyclicBlocks}). The ordinary characters of the trivial source $kN'$-modules are obtained once again from the classification described in Remark \ref{rem:tsmodulecyclicdef}. In this case, for the listed characters of the trivial source $kN'$-modules with non-full vertex we use \cite[Theorem~5.3(b)(2),(7) and Theorem~A.1(d)]{HL20} and \cite[Theorem 7.1(a)]{KL21}.  
	
	Lastly, for any $1 \leq i \leq n$, if $[M]\in\TS(G;C_{\ell^{i}})$ is such that $\chi_{\widehat{M}}$ has $1_G$ as a constituent, then $[f(M)]\in\TS(N';C_{\ell^{i}})$ and $\chi_{\widehat{f(M)}}$ must have $1_{N'}$ as a constituent. This determines the correspondence of characters in the principal blocks of $G$ and $N'$. The two trivial source modules $[M],[M']\in\TS(G;C_{\ell^{i}})$ in $A'_{\theta_0}$ have the same ordinary character so up to the labelling of characters, the correspondences are also determined in this case. Finally, there is a unique choice for $\chi_{\widehat{f(M)}}$ for any trivial source module $[M]\in\TS(G;C_{\ell^{i}})$ in a nilpotent block of $G$ because, up to isomorphism,  such blocks contain a unique trivial source module with a given vertex $C_{\ell^{i}}$. 
\end{proof}

\vspace{2mm}

\subsection{The trivial source character table of $G$}
\label{sec:tschartablelmidq+1}

Recall that 
\[\kappa = \left\{ \begin{array}{ll} 0 & \mbox{ if } q \equiv 1 \pmod{4} \\ 1 & \mbox{ if } q \equiv 3 \pmod{4} .\end{array} \right. \]

\vspace{2mm}

\begin{thm}\label{thm:l|q+1}
	Suppose that $2 \neq \ell \mid q+1$. With the notation as in Notation \ref{rem:lmidq+1}, the trivial source character table $\Triv_{\ell}(G)= [T_{i,v}]_{1\leq i,v\leq n+1}$ is given as follows:
       \begin{enumerate}[{\,\,\rm(a)}] \setlength{\itemsep}{2pt}
		\item $T_{i,v} = \mathbf{0}$ if $v > i$;
		\item the matrices $T_{i,1}$ are as given in Table \ref{tab:lmidq+1T_i1} for each $1 \leq i \leq n+1$; 
		\item the matrices $T_{i,i}$ are as given in Table \ref{tab:lmidq+1T_ii} for each $2 \leq i \leq n$;  
		\item the matrix $T_{n+1,n+1}$ is as given in Table \ref{tab:lmidq+1T_n+1n+1}; and 
		\item $T_{i,v} = T_{i,i}$ for all $2 \leq v \leq i \leq n+1$.
	\end{enumerate} 	
\end{thm}

\begin{landscape}
	
	
	\renewcommand*{\arraystretch}{1.5}
	\begin{longtable}{c|c||c|c|c|c|}
		\caption{$T_{i,1}$ for $1 \leq i \leq n+1$.}	
		\label{tab:lmidq+1T_i1}
		\\  \cline{2-6}
		
		&&
\begin{tabular}{c}
	\textbf{$\varepsilon I_2$ } \vspace{-.8ex}  \\
	\textbf{	$ \varepsilon \in \{ \pm 1\}$}
\end{tabular}
& 
\begin{tabular}{c}
	$\mathbf{d}(a)$ \vspace{-.8ex}  \\
	\footnotesize{$ a \in \Gamma_{\ell'}$}
\end{tabular}
&
\begin{tabular}{c}
	$\mathbf{d}'(\xi)$ \vspace{-.8ex}  \\
	\footnotesize{$ \xi \in \Gamma'_{\ell'}$}
\end{tabular}
&
\begin{tabular}{c}
	\textbf{$ \varepsilon u_{\tau}$} \vspace{-.8ex}  \\
	\textbf{	$ \varepsilon, \tau \in \{ \pm 1\}$}
\end{tabular}
\\ \hline \hline

		& $1_G + \St$ 
		& $1 + q$ 
		& $2$ 
		& $0$ 
		& $1$ 
		\\ \cline{2-6}

		& $\St + \Xi'$ 
		& $ q + (q-1)\pi'_q$ 
		& $1$ 
		& $-1 - 2\pi'_q$ 
		& $-\pi'_q $
		\\ \cline{2-6}
		
		&\begin{tabular}{c}
			$R'_\pm(\theta_0) + \Xi'_{\theta_0}$ 
		\end{tabular}
		& $\varepsilon^{\kappa+1}\left(\frac{q-1}{2}\right)(1+2\pi'_q)$ 
		& $0$
		& $-\theta_0(\xi)(1 + 2 \pi'_q)$
		& $\varepsilon^{\kappa + 1}\left(\frac{-1 \pm \tau \sqrt{q_0}}{2} - \pi'_q\right)$ 
		\\ \cline{2-6}
		
		$T_{1,1} $
		&\begin{tabular}{c}
			$R'(\theta) + \Xi'_{\theta}$ \vspace{-1ex} \\ 
			\tiny{$\left(\theta \in [T_{\ell'}'^\wedge / \equiv], \theta^2 \neq 1\right)$}
		\end{tabular}
		& $\theta(\varepsilon)(q-1)(1 + 2\pi'_q)$
		& $0$
		& $-(\theta(\xi) + \theta(\xi^{-1}))(1 + 2\pi'_q)$
		& $-\theta(\varepsilon)(1+ 2\pi'_q)$
		\\ \cline{2-6}
		
		& $R_\pm(\alpha_0)$ 
		& $\varepsilon^{\kappa} \frac{q+1}{2}$
		& $\alpha_0(a)$
		& $0$ 
		& $\varepsilon^\kappa \left(\frac{1 \pm \tau \sqrt{q_0}}{2}\right)$
		\\ \cline{2-6}

		&\begin{tabular}{c}
			$R(\alpha)$ \vspace{-.8ex} \\	
			\tiny{$\left(\alpha \in [T_{\ell'}^\wedge / \equiv], \alpha^2 \neq 1\right)$}
		\end{tabular}
		& $\alpha(\varepsilon)(q+1)$
		& $\alpha(a) + \alpha(a^{-1})$
		& $0$
		& $\alpha(\varepsilon)$
		\\ 
		\hline
		\hline

		& $1_G + \St + \Xi'_{i-1}$ 
		& $1 + q + (q-1)\pi'_{q,i-1}$
		& $2$ 
		& $-2\pi'_{q,i-1}$ 
		& $1 - \pi'_{q,i-1}$
		\\ \cline{2-6}

		&		
		$\St + \Xi'_{i-1}$ 
		& $q + (q-1)\pi'_{q,i-1}$
		& $1$
		& $-1-2\pi'_{q,i-1}$
		& $-\pi'_{q,i-1}$
		\\ \cline{2-6}

		$T_{ i,1}$
		&\begin{tabular}{c}
			$\Xi'_{\theta_0, i-1}$	
		\end{tabular}
		& $\varepsilon^{\kappa + 1}(q-1)\pi'_{q,i-1}$
		& $0$
		& $-2\theta_0(\xi)\pi'_{q,i-1}$
		& $\varepsilon^\kappa \pi'_{q,i-1}$
		\\ \cline{2-6}

		$(2 \leq i \leq n)$
		&\begin{tabular}{c}
			$\Xi'_{\theta_0,i-1}$	
		\end{tabular}
		& $\varepsilon^{\kappa + 1}(q-1)\pi'_{q,i-1}$
		& $0$
		& $-2\theta_0(\xi)\pi'_{q,i-1}$
		& $\varepsilon^\kappa \pi'_{q,i-1}$
		\\ \cline{2-6}

		&\begin{tabular}{c}
			$\Xi'_{\theta, i-1}$	\vspace{-1ex} \\
			\tiny{$\left(\theta \in [T_{\ell'}'^\wedge / \equiv], \theta^2 \neq 1\right)$}
		\end{tabular}
		& $2\theta(\varepsilon)(q-1)\pi'_{q,i-1}$
		& $0$
		& $-2\left(\theta(\xi) + \theta(\xi^{-1})\right)\pi'_{q,i-1}$ 
		& $-2\theta(\varepsilon)\pi'_{q,i-1}$
		\\ 
		\hline 		
		\hline

		& $1_G$ 
		& $1$
		& $1$
		& $1$
		& $1$
		\\ \cline{2-6}

		& $\Xi'$ 
		& $(q-1)\pi'_q$
		& $0$ 
		& $-2\pi'_q$
		& $-\pi'_q$
		\\ \cline{2-6}
		
		$T_{n+1,1}$
		&\begin{tabular}{c}
			$\Xi'_{\theta_0}$
		\end{tabular}
		& $\varepsilon^{\kappa + 1}(q-1)\pi'_q$
		& $0$
		& $-2\theta_0(\xi)\pi'_q$
		& $\varepsilon^{\kappa}\pi'_q$
		\\ \cline{2-6}

		&\begin{tabular}{c}
			$\Xi'_{\theta_0}$
		\end{tabular}
		& $\varepsilon^{\kappa + 1}(q-1)\pi'_q$
		& $0$
		& $-2\theta_0(\xi)\pi'_q$
		& $\varepsilon^{\kappa}\pi'_q$
		\\ \cline{2-6}

		&\begin{tabular}{c}
			$\Xi'_{\theta}$ \vspace{-1ex} \\
			\tiny{$\left(\theta \in [T_{\ell'}'^\wedge / \equiv], \theta^2 \neq 1\right)$}
		\end{tabular}
		& $2\theta(\varepsilon)(q-1)\pi'_q$
		& $0$
		& $-2\pi'_q(\theta(\xi) + \theta(\xi^{-1}))$
		& $-2\theta(\varepsilon)\pi'_q$
		\\ \hline
		
	\end{longtable}
	
\end{landscape}

\renewcommand*{\arraystretch}{1.5}
\begin{longtable}{|c||c|c|c|c|}
	\caption{$T_{i,i}$ for $2 \leq  i \leq n$.} 
	\label{tab:lmidq+1T_ii}
	\\  \cline{1-5}
	
	&
	\begin{tabular}{c}
		\textbf{$\varepsilon I_2$ } \vspace{-.8ex}\\
		\textbf{	$ \varepsilon \in \{ \pm 1\}$}
	\end{tabular}
	& 
	\begin{tabular}{c}
		$\mathbf{d}'(\xi)$ \vspace{-.8ex}  \\
		\footnotesize{$ \xi \in \Gamma'_{\ell'}$}
	\end{tabular}
	&
	\begin{tabular}{c}
		\textbf{$\sigma'_{\tau}$} \vspace{-.8ex}\\
		\textbf{	$\tau \in \{ \pm 1\}$}
	\end{tabular}
	\\ \hline \hline 			
	
	$1_G + \St + \Xi'_{i-1}$ 
	& $1 + 2\pi''_{q,i-1}$ 
	& $1 + 2\pi''_{q,i-1}$
	& $1$ 
	\\ \cline{1-5}

	$\St + \Xi'_{i-1}$ 
	& $1 + 2 \pi''_{q,i-1}$
	& $1 + 2\pi''_{q,i-1}$
	& $-1$ 
	\\ \cline{1-5}

	\begin{tabular}{c}
		$\Xi'_{\theta_0, i-1}$
	\end{tabular}
	& $\varepsilon^{\kappa + 1}\left(1 + 2\pi''_{q,i-1}\right)$
	& $\theta_0(\xi)\left(1 + 2\pi''_{q,i-1}\right)$
	& $\tau \sqrt{-1^{\kappa + 1}}$
	\\ \cline{1-5}

	\begin{tabular}{c}
		$\Xi'_{\theta_0, i-1}$
	\end{tabular}
	& $\varepsilon^{\kappa + 1}\left(1 + 2\pi''_{q,i-1}\right)$
	& $\theta_0(\xi)\left(1 + 2\pi''_{q,i-1}\right)$
	& $-\tau \sqrt{-1^{\kappa + 1}}$
	\\ \cline{1-5}

	\begin{tabular}{c}
		$\Xi'_{\theta, i-1}$ \vspace{-1ex}\\ 
		\tiny{$\left(\theta \in [T_{\ell'}'^\wedge / \equiv], \theta^2 \neq 1\right)$}
	\end{tabular}
	& $2\theta(\varepsilon)(1 + 2\pi''_{q,i-1})$
	& $\left(\theta(\xi) + \theta(\xi^{-1})\right)(1 + 2\pi''_{q,i-1})$
	& $0$
	\\ \hline 
	
\end{longtable}

\renewcommand*{\arraystretch}{1.5}
\begin{longtable}{|c||c|c|c|c|}
	\caption{$T_{n+1,n+1}$.} 
	\label{tab:lmidq+1T_n+1n+1}
	\\  \cline{1-5}
	
	&
	\begin{tabular}{c}
		\textbf{$\varepsilon I_2$ } \vspace{-.8ex}\\
		\textbf{	$ \varepsilon \in \{ \pm 1\}$}
	\end{tabular}
	& 
	\begin{tabular}{c}
		$\mathbf{d}'(\xi)$ \vspace{-.8ex}  \\
		\footnotesize{$ \xi \in \Gamma'_{\ell'}$}
	\end{tabular}
	&
	\begin{tabular}{c}
		\textbf{$\sigma'_{\tau}$} \vspace{-.8ex}\\
		\textbf{	$\tau \in \{ \pm 1\}$}
	\end{tabular}
	\\ \hline \hline 			
	
	$1_G$ 
	& $1$ 
	& $1$
	& $1$ 
	\\ \cline{1-5}

	$\Xi'$ 
	& $1$
	& $1$
	& $-1$ 
	\\ \cline{1-5}

	\begin{tabular}{c}
		$\Xi'_{\theta_0}$
	\end{tabular}
	& $\varepsilon^{\kappa + 1}$
	& $\theta_0(\xi)$
	& $\tau \sqrt{-1^{\kappa + 1}}$
	\\ \cline{1-5}

	\begin{tabular}{c}
		$\Xi'_{\theta_0}$
	\end{tabular}
	& $\varepsilon^{\kappa + 1}$
	& $\theta_0(\xi)$
	& $-\tau \sqrt{-1^{\kappa + 1}}$
	\\ \cline{1-5}

	\begin{tabular}{c}
		$\Xi'_{\theta}$ \vspace{-1ex}\\ 
		\tiny{$\left(\theta \in [T_{\ell'}'^\wedge / \equiv], \theta^2 \neq 1\right)$}
	\end{tabular}
	& $2\theta(\varepsilon)$
	& $\theta(\xi) + \theta(\xi^{-1})$
	& $0$
	\\ \hline 
	
\end{longtable}

\begin{proof}
	The calculations for this table are completely analogous to those in the proof of Theorem~\ref{thm:l|q-1} except that we take the trivial source $kG$-modules from Table \ref{tab:tsmodslmidq+1}, their  Green correspondents from Table~\ref{tab:tsmodsN'} and we use the character table of $N'$ from~\cite[Table 6.3]{BonBook}.	
\end{proof}

\normalsize{}


{ \normalsize
	
\vspace{6mm}	
	
	\section{$\PSL_2(q)$ with $q\equiv\pm3\pmod{8}$ at $\ell=2$}\label{sec:PSLl=2}
	
	Throughout this section, we assume that  $\ell=2$ and we consider the group $\overline{G}=\PSL_2(q)$ with $q\equiv\pm3\pmod{8}$.  By Lemma~\ref{lem:omnibusl=2} the Sylow $2$-subgroups of  $\overline{G}$ are Klein-four groups $C_{2}\times C_{2}$. The $2$-blocks of $\overline{G}$ have defect groups isomorphic to  $C_{2}\times C_{2}$, $C_{2}$ or $\{1\}$, and in fact $\bB_{0}(\overline{G})$ is the unique block with full defect. Trivial source modules in blocks of defect zero are just PIMs of $k\overline{G}$, whereas trivial source modules in blocks of defect~$1$ are easily dealt with via Remark~\ref{rem:cyclicDefectC2C4}. Therefore, in this case, it essentially remains to describe the trivial source modules lying in the principal block. Moreover,  as source-algebra equivalences preserve vertices and sources of modules, it is enough to consider the principal block up to source-algebra equivalence, which reduces our computations to $\PSL_{2}(3)$ and $\PSL_{2}(5)$.
	\par
	
	\begin{nota}\label{nota:psl2ql=2}%
		In this section, in order to describe $\Triv_{2}(\overline{G})$ according to 
		Convention~\ref{conv:tsctbl}, we adopt the following notation. 
		We let $Q_{3}$ be the Sylow $2$-subgroup of $\overline{G}$ defined in Lemma~\ref{lem:omnibusl=2}, that is,  
		$$Q_{3}=
		\begin{cases}
			\langle S_2', \sigma' \rangle/Z& \text{ if }q\equiv 3\pmod{8}, \\
			\langle S_2, \sigma \rangle/Z &\text{ if }q\equiv -3\pmod{8}. 
		\end{cases}
		$$
		It can be read from the character table of $\SL_{2}(q)$ (Table~\ref{tab:CTBLsl2}) that all involutions in~$\overline{G}$ are conjugate. Thus we may  fix the following  set of representatives for the conjugacy classes of $2$-subgroups of~$\overline{G}$:
		$$Q_{3}\cong C_{2}\times C_{2},\quad Q_1:=\{1\},\quad\text{and }\quad 
		Q_{2}:=
		\begin{cases}
			S_{2}'/Z\cong C_{2}& \text{ if }q\equiv 3\pmod{8}, \\
			S_{2}/Z\cong C_{2} &\text{ if }q\equiv -3\pmod{8}.
		\end{cases}$$
		We set 
		$\Gamma_{2'}:=[((\mu_{q-1})_{2'} \setminus \{1\})/\equiv]$,  $\Gamma'_{2'}:=[((\mu_{q+1})_{2'} \setminus \{1\})/\equiv]$. 
		Then, the structure of the quotients $\overline{N}_{i}=N_{\overline{G}}(Q_{i})/Q_{i}$ $(1\leq i\leq 3)$ follows from Lemma~\ref{lem:omnibusl=2}{\rm(b),(c)} and we choose the following sets of representatives of their $2'$-conjugacy classes:
		\begin{enumerate}[{\,\,\rm(i)}] \setlength{\itemsep}{2pt}
			\item 
			$\overline{N}_{1}=\overline{G}$ and we set 
			$$\phantom{XXXl}[\overline{N}_{1}]_{2'}:=\{I_{2}Z\}\cup\{{\bf d}(a)Z\mid a\in \Gamma_{2'}\}\cup\{{\bf d}'(\xi)Z\mid \xi \in \Gamma'_{2'}\}\cup\{u_{\tau}Z\mid\tau\in\{\pm1\}\}\,;$$
			\item 
			$$\overline{N}_{2}=
			\begin{cases}
				\langle T'/Z,\sigma 'Z\rangle/Q_{2}
				\cong \mathcal{D}_{\frac{q+1}{2}}  & \text{ if }q\equiv 3\pmod{8},  \\
				\langle T/Z,\sigma Z\rangle/Q_{2}
				\cong \mathcal{D}_{\frac{q-1}{2}}     & \text{ if }q\equiv -3\pmod{8},
			\end{cases}
			$$
			and  we set  
			$$[\overline{N}_{2}]_{2'}:=
			\begin{cases}
				\{1\}\cup\{\overline{{\bf d}'(\xi)Z}\in \overline{N}_{2} \mid \xi\in \Gamma'_{2'}\}  & \text{ if }q\equiv 3\pmod{8}, \\
				\{1\}\cup\{\overline{{\bf d}(a)Z} \in \overline{N}_{2}\mid a\in \Gamma_{2'}\}     & \text{ if }q\equiv -3\pmod{8},
			\end{cases}
			$$
			where $1$ simply denotes the trivial element of $\overline{N}_{2}$ and  the bar notation denotes left cosets in the \smallskip quotient $\overline{N}_{2}$;
			\item
			$\overline{N}_{3}\cong C_{3}=:\langle x\rangle$ and we set $[\overline{N}_{3}]_{2'}=:\{1,x,x^{2}\}$.
		\end{enumerate}
	\end{nota}
	\bigskip

	\vspace{4mm}
	\subsection{The $2$-blocks}\label{sec:blocksPSL2q}
	\enlargethispage{5mm}

	\begin{lem} {\label{lem:blockslmidq-1}}
		When $q \equiv \pm3\pmod{8}$ the $2$-blocks of $\overline{G}=\PSL_2(q)$, a defect group, and their decomposition matrices, resp. Brauer trees with type function, are as given in Table~\ref{tab:blocksPSL2q3mod8} and~Table~\ref{tab:blocksPSL2q5mod8}.
	\end{lem}
	
	\begin{longtable}{c||c|c|c}
		\caption{{\textup{The $2$-blocks of $\PSL_2(q)$ when $q \equiv 3\pmod{8}$.}}}\vspace{-2mm}
		\label{tab:blocksPSL2q3mod8}
		\\  
		Block
		&	\begin{tabular}{c}
			Number of Blocks \\
			(Type)
		\end{tabular}
		& 	\begin{tabular}{c}
			Defect \\
			Groups
		\end{tabular}
		& 
		\begin{tabular}{c}
			Decomposition Matrix/\\Brauer Tree with Type Function 	
		\end{tabular}  
		\\ \hline \hline 
		$\bB_{0}(\overline{G})$ 
		& 	\begin{tabular}{c}
			$1$ \\
			\\
			(Principal) 
		\end{tabular}
		& 	\begin{tabular}{c}
			$
			C_{2}\times C_{2}$ 
		\end{tabular}
		&  
		\begin{blockarray}{cccc}
			\vspace{-2.9ex} \\
			& {~} $k$ {~} & $\overline{S}_{+}$ & $\overline{S}_{-}$   \\ 
			\begin{block}{c(ccc)}
				$1_{\overline{G}}$					& $1$ & $0$ & $0$\\
				$\St$					& $1$ & $1$ & $1$\\
				$R_+'(\theta_0)$		& $0$ & $1$ & $0$\\
				$R_-'(\theta_0)$        & $0$ & $0$ & $1$\\
			\end{block}
		\end{blockarray}  	\vspace{-1.5ex} 
		\\ \hline 
		\vspace{-1.8ex} 
		\begin{tabular}{c}
			$A_\alpha$ \\
			\small{$(\alpha \in [T_{2'}^\wedge / \equiv] \setminus \{1\})$} 
			\\
		\end{tabular}	
		& \begin{tabular}{c}
			\vspace{-1.5ex}   \\
			$\frac{1}{2}\left( \frac{q-1}{2} - 1\right)$ \\
			\small{(Defect zero)} \\
			\\
		\end{tabular}
		& 		$
		\{1\}$	
		&      $\Irr(A_\alpha) = \{R(\alpha)\}$
		\\ \hline 
		
		\begin{tabular}{c}
			$A'_\theta$  \\ 
			
			\small{$(\theta \in [T_{2'}'^\wedge / \equiv] \setminus \{1\})$} 
		\end{tabular}
		& 	\begin{tabular}{c}
			$\frac{1}{2}\left(\frac{q+1}{4} - 1\right)$ \\ 
			
			\small{(Nilpotent)} \\
		\end{tabular}
		& 	$
		C_{2}$	
		& 	 \begin{tabular}{c}
			$$  \xymatrix@R=0.0000pt@C=50pt{	
				{_-} & {_+}\\
				{\Circle}  \ar@{-}[r]^{S_{\theta}} & {\Circle}    \\
				{^{R'(\theta)}}&{^{R'(\theta\eta^{2})}}
			} $$ \\
			
			\small{$(S_2'^\wedge = \langle \eta \rangle)$}\\
			\vspace{-.8ex}
		\end{tabular}
		\\ \hline 	
	\end{longtable}
	
	\begin{longtable}{c||c|c|c}
		\caption{{\textup{The $2$-blocks of $\PSL_2(q)$ when $q \equiv -3\pmod{8}$.}}}\vspace{-2mm}	
		\label{tab:blocksPSL2q5mod8}
		\\  
		Block
		&	\begin{tabular}{c}
			Number of Blocks \\
			(Type)
		\end{tabular}
		& 	\begin{tabular}{c}
			Defect \\
			Groups
		\end{tabular}
		& 
		\begin{tabular}{c}
			Decomposition Matrix/ \\ Brauer Tree with Type Function
			
		\end{tabular}
		\\ \hline \hline 
		
		$\bB_{0}(\overline{G})$ 
		& \begin{tabular}{c}
			$1$ \\
			\\
			(Principal)
		\end{tabular}
		& \begin{tabular}{c}
			$
			C_{2}\times C_{2}$ 
		\end{tabular}
		&  
		\begin{blockarray}{cccc}
			\vspace{-2.9ex} \\
			& {~} $k$ {~} & $\overline{S}_+$ & $\overline{S}_-$   \\ 
			\begin{block}{c(ccc)}
				$1_{\overline{G}}$					&$1$&$0$&$0$\\
				$\St$					&$1$&$1$&$1$\\
				$R_+(\alpha_0)$		&$1$&$1$&$0$\\
				$R_-(\alpha_0)$        &$1$&$0$&$1$\\
			\end{block} 	
		\end{blockarray}   \vspace{-1.5ex} 
		\\ \hline 
		
		\begin{tabular}{c}
			$A_\alpha$\\
			\small{$\alpha \in [T_{2'}^\wedge / \equiv] \setminus \{1\}$} 
		\end{tabular}
		& 	\begin{tabular}{c}
			$\frac{1}{2}\left( \frac{q-1}{4} - 1\right)$ \\
			\small{(Nilpotent)} \\
		\end{tabular}
		& 	$
		C_2$
		& 		\begin{tabular}{c}
			$$  \xymatrix@R=0.0000pt@C=30pt{	
				{_+} & {_-}\\
				{\Circle}  \ar@{-}[r]^{S_{\alpha}}  & {\Circle}    \\
				{^{R(\alpha)}}&{^{R(\alpha\eta^{2})}}
			}$$ \\
			\small{$(S_2^\wedge  = \langle \eta \rangle)$}\\
			\vspace{-1.5ex}	
		\end{tabular}
		\\ \hline 
		
		\begin{tabular}{c}
			$A'_\theta$  \\ 
			\small{$\theta \in [T_{2'}'^\wedge / \equiv] \setminus \{1\}$} 
			\\
		\end{tabular}
		& 	\begin{tabular}{c}
			\vspace{-1.2ex}     \\
			$\frac{1}{2}\left(\frac{q+1}{2} - 1\right)$ \\ 
			\small{(Defect zero)} 
			\vspace{1ex}  \\
		\end{tabular}
		& 	$
		\{1\}$		
		&  $\Irr(A'_\theta) = \{R'(\theta)\}$
		\\ \hline 
		
	\end{longtable}
	\bigskip
	
	\begin{proof}
		Since $\overline{G}=G/Z=\SL_{2}(q)/\{\pm I_{2}\}$, the block distribution, the number of blocks, the decomposition matrices and the Brauer trees follow directly from Table~\ref{tab:blocksqcong3} and~Table~\ref{tab:blocksqcong5} containing the corresponding information for $\SL_{2}(q)$. Therefore, it only remains to compute the type function on the Brauer trees.\par 
		First assume that $q\equiv 3\pmod{8}$ and let $\bB\in \{A'_{\theta}\mid \theta \in [T_{2'}'^\wedge / \equiv] \setminus \{1\}\}$. A defect group of $\bB$ is $D:=S_2'/Z\cong C_{2}$  with generator $\mathbf{d'}(\xi)Z$, where  $\mathbf{d'}(\xi)$ is an element of order $4$ of~$T'$.  Moreover, in all cases,  $\theta$ is a $2'$-character of $T'$, so $\theta(\xi)=1$  and we read from the character table of $G$ (see Table~\ref{tab:CTBLsl2}) that $R'(\theta)(\mathbf{d'}(\xi))=-\theta(\xi)-\theta(\xi)^{-1}=-1-1=-2$. Hence we have $R'(\theta)<0$, determining the type function on $\sigma(\bB)$. 
		\par
		If  $q\equiv -3\pmod{8}$ and  $\bB\in \{A_{\alpha}\mid \alpha \in [T_{2'}^\wedge / \equiv] \setminus \{1\}\}$, then a defect group of $\bB$ is ${D:=S_2/Z\cong C_{2}}$ with generator  $\mathbf{d}(a)Z$, where  $\mathbf{d}(a)$ is an element of~$T$ of order~$4$.  Again, in all cases, $\alpha$ is  a $2'$-character of $T$,  so  $ \alpha(a) = 1$ and it follows from the character table of $G$ that  $R(\alpha)(\mathbf{d}(a)) =\alpha(a)+\alpha(a)^{-1}=1+1=2$. Thus we have $R(\alpha)>0$,  determining the type function on $\sigma(\bB)$.
	\end{proof}
	
	\begin{cor}\label{cor:PuigEqC2xC2}%
		If $\overline{G}=\PSL_2(q)$ with $q \equiv \pm3\pmod{8}$, then:
		\begin{enumerate}[{\,\,\rm(a)}] \setlength{\itemsep}{2pt}
			\item $\bB_{0}(\overline{G})\sim_{SA} \bB_{0}(\PSL_{2}(3))$ if  $q \equiv 3\pmod{8}$;
			\item $\bB_{0}(\overline{G})\sim_{SA} \bB_{0}(\PSL_{2}(5))$ if  $q \equiv -3\pmod{8}$.
		\end{enumerate}
	\end{cor}
	
	\begin{proof}
		The main result of \cite{CEKL} establishes that up to source-algebra equivalence a $2$-block with a  Klein-four defect group  is $k[C_{2}\times C_{2}]$, $k\fA_{4}$ or $\bB_{0}(\fA_{5})$.  Thus, as  $\fA_{4}\cong \PSL_{2}(3)$ and $\fA_{5}\cong \PSL_{2}(5)$,  the claim  follows directly from the $2$-decomposition matrix of $\bB_{0}(k\overline{G})$ in Table~\ref{tab:blocksPSL2q3mod8} and Table~\ref{tab:blocksPSL2q5mod8} respectively.
	\end{proof}

	\vspace{2mm}
	\subsection{Trivial source modules in the principal block}
	
	We start with the description of the trivial source $k\overline{G}$-modules lying in the principal block, which is the unique block with defect groups which are not cyclic.
	
	\begin{lem}\label{lem:B0A4}%
		If $\overline{G}=\PSL_{2}(q)$ where  $q \equiv 3\pmod{8}$, then the trivial source $k\overline{G}$-modules and their ordinary characters are as given in Table~\ref{tab:psl23}.  
		{
			\renewcommand*\arraystretch{1.5}
			\begin{table}[!htpb]%
				\caption{Trivial source $\bB_{0}(\PSL_{2}(q))$-modules for  $q \equiv 3\pmod{8}$}\label{tab:psl23}\vspace{-5mm}
				\[\begin{array}{|c|c|c|}
					\hline
					\text{Vertices} 
					&	\text{Module }M
					&    \text{Character }\chi_{\widehat{M}} 
					\\ \hline \hline 
					\{1\}	
					& P_{k}
					& 1_{\overline{G}}+\St
					\\ 
					\{1\}	
					& P_{\overline{S}_{+}}
					& \St+R_+'(\theta_0)
					\\ 
					\{1\}
					& P_{\overline{S}_{-}}
					& \St+R_-'(\theta_0)  
					\\ \hline\hline
					C_{2}
					& \Sc(\overline{G},C_{2})
					& 1_{\overline{G}}+\St+R_+'(\theta_0)+R_-'(\theta_0)
					\\ \hline \hline
					C_{2}\times C_{2}	
					& k
					& 1_{\overline{G}}
					\\ 
					C_{2}\times C_{2}	
					& \overline{S}_{+}
					& R_+'(\theta_0)
					\\  
					C_{2}\times C_{2}	
					&  \overline{S}_{-} 
					& R_-'(\theta_0)  \\
					\hline	 
				\end{array}\]
			\end{table}
		}
	\end{lem}
	
	\begin{proof}
		To begin with, by Corollary~\ref{cor:PuigEqC2xC2}, $\bB_{0}(\overline{G})$  is source-algebra equivalent to $\bB_{0}(\PSL_{2}(3))$. Now, on the one hand, source-algebra equivalences preserve vertices and sources, hence trivial source modules together with their vertices (see Proposition~\ref{prop:omnts}(f)). On the other hand, they are Morita equivalences, hence preserve simple modules and composition factors, so that using the induced character bijections and  the $2$-decomposition matrix of $\bB_{0}(\overline{G})$ in Table~\ref{tab:blocksPSL2q3mod8}, we may assume that $\overline{G}=\PSL_{2}(3)\cong\fA_{4}$. 
		\par
		Next, we note that when $q=3$ the principal block is the unique $2$-block of $k\overline{G}$. See  Table~\ref{tab:blocksPSL2q3mod8}.  
		Thus, the trivial source $\bB_{0}(\overline{G})$-modules account for all the trivial source  $k\overline{G}$-modules in this case.
		\par 
		The PIMs  are the projective covers $P_{k}, P_{\overline{S}_{+}},  P_{\overline{S}_{-}}$ of the three simple modules $k,\overline{S}_{+},\overline{S}_{-}$ and the ordinary characters they afford are read off from the $2$-decomposition matrix in Table~\ref{tab:blocksPSL2q3mod8}. Next, the representatives of the conjugacy classes of non-trivial $2$-subgroups are $Q_{3}\cong C_{2}\times C_{2}$ and $Q_{2}\cong C_{2}$, thus 
		by Notation~\ref{nota:psl2ql=2} we have $\overline{N}_{3}\cong C_{3}$ and $\overline{N}_{2}\cong C_{2}$. It follows then from  Proposition~\ref{prop:omnts}(d) that 
		$${|}\TS(\overline{G};C_{2}\times C_{2})|=3\qquad\text{and}\qquad {|}\TS(\overline{G};C_{2})|=1\,.$$
		Because the simple modules $k$, $\overline{S}_{+}$ and $\overline{S}_{-}$ are all liftable and afford the ordinary characters $1_{\overline{G}}$,~$R_+'(\theta_0)$ and $R_-'(\theta_0)$ respectively, which are all linear when $q=3$, we obtain directly that $k$, $\overline{S}_{+}$ and $\overline{S}_{-}$ account for all the trivial source $k\overline{G}$-modules with vertex $C_{2}\times C_{2}$.
		By Proposition~\ref{prop:omnts}(e), the unique trivial source $k\overline{G}$-module with vertex $C_{2}$ is the Scott module $\Sc(\overline{G},C_{2})$, and, by definition, it is a direct summand of $k\!\indhg{Q_{2}}{\overline{G}}$. However, this induced module is indecomposable (see e.g. \cite[Remark after Theorem 4.3.3]{BensonBookI}, which classifies the indecomposable $k\fA_{4}$-modules) and it is easy to compute that 
		\[1^{}_{C_{2}}\indhg{C_{2}}{\overline{G}}=1_{\overline{G}}+\St+R_+'(\theta_0)+R_-'(\theta_0)\,.\]
		The claim follows. 
	\end{proof}
	
	\begin{lem}\label{lem:B0A5}%
		If $\overline{G}=\PSL_{2}(q)$ where $q\equiv -3\pmod{8}$, then the trivial source $\bB_{0}(\overline{G})$-modules and their ordinary characters are as given in Table~\ref{tab:psl25}. 
		{
			\renewcommand*\arraystretch{1.5}
			\begin{table}[!htpb]
				\caption{Trivial source $\bB_{0}(\PSL_{2}(q))$-modules for  $q \equiv -3\pmod{8}$}\label{tab:psl25}\vspace{-5mm}
				\[\begin{array}{|c|c|c|}
					\hline
					\text{vertices} 
					&	\text{module }M
					&    \text{character $\chi_{\widehat{M}}$} 
					\\ \hline \hline 
					\{1\}	
					& P_{k}
					& 1_{\overline{G}}+\St+R_+(\alpha_0)+R_-(\alpha_0)  
					\\ 
					\{1\}	
					& P_{\overline{S}_{+}}
					& \St+R_+(\alpha_0) 
					\\
					\{1\}	
					& P_{\overline{S}_{-}}
					& \St+R_-(\alpha_0)  
					\\ \hline\hline
					C_{2}
					&  \mbox{Sc}(\overline{G},C_{2}) & 1_{\overline{G}}+\St
					\\ \hline \hline
					C_{2}\times C_{2}	
					& k  		
					& 1_{\overline{G}}
					\\ 
					C_{2}\times C_{2}	
					& \boxed{
						\begin{smallmatrix}
							\overline{S}_{+}\\
							k\\
							\overline{S}_{-}
						\end{smallmatrix}
					}
					& \St
					\\ 
					&&	\vspace{-3ex} \\
					C_{2}\times C_{2}	
					& \boxed{
						\begin{smallmatrix}
							\overline{S}_{-}\\
							k\\
							\overline{S}_{+}
						\end{smallmatrix}
					}
					& \St	\\ 
					\hline	 
				\end{array}\]
			\end{table}
		}
	\end{lem}
	
	\begin{proof}
		Similarly to the proof of Lemma~\ref{lem:B0A4}, we may assume that $\overline{G}=\PSL_{2}(5)\cong \fA_{5}$ because by Corollary~\ref{cor:PuigEqC2xC2}, $\bB_{0}(\overline{G})$  is source-algebra equivalent to $\bB_{0}(\PSL_{2}(5))$. 
		\par
		First recall that the representatives of the conjugacy classes of non-trivial $2$-subgroups are $Q_{3}\cong C_{2}\times C_{2}$ and $Q_{2}\cong C_{2}$.
		Now, the group algebra $k\!\PSL_{2}(5)$ has exactly two $2$-blocks, namely the principal block (with a defect group equal to $Q_{3}$) and a block of defect zero. 
		See Table~\ref{tab:blocksPSL2q5mod8}. Thus, it follows that $\bB_{0}\left(\PSL_{2}(5)\right)$ contains three PIMs (see the $2$-decomposition matrix in Table~\ref{tab:blocksPSL2q5mod8}), and all other  trivial source $k\overline{G}$-modules with vertex $C_{2}\times C_{2}$ or $C_{2}$ lie in $\bB_{0}(\overline{G})$. \par
		By Notation~\ref{nota:psl2ql=2}  we have  $\overline{N}_{3}\cong C_{3}$ and $\overline{N}_{2}\cong C_{2}$, so it follows then from  Proposition~\ref{prop:omnts}(d) that 
		$${|}\TS(\overline{G};C_{2}\times C_{2})|=3\qquad\text{and}\qquad {|}\TS(\overline{G};C_{2})|=1\,.$$
		The PIMs in $\bB_{0}(\overline{G})$ are the projective covers of the three simple $\bB_{0}(\overline{G})$-modules $k,\overline{S}_{+},\overline{S}_{-}$ and the ordinary characters they afford are read off from the $2$-decomposition matrix in Table~\ref{tab:blocksPSL2q5mod8}. 
		\par
		As $N_{\overline{G}}(Q_{3})\cong \fA_{4}$ (see Lemma~\ref{lem:omnibusl=2}(b)(ii)), by Proposition~\ref{prop:omnts}(d), the three trivial source modules with vertex $Q_{3}$ are the Green correspondents in $\overline{G}$ of the three simple $1$-dimensional $k\fA_{4}$-modules obtained in Lemma~\ref{lem:B0A4}, which we denote in this proof by  $k,k_{\omega},k_{\omega}^{\ast}$ in order to avoid notation conflicts. The Green correspondent of the trivial module is the trivial module and clearly affords the trivial character~$1_{\overline{G}}$\,. 
		Next 
		\[\dim_{k}\left(k_{\omega}\indhg{\fA_{4}}{\overline{G}}\right)= \dim_{k}\left(k_{\omega}^{\ast}\indhg{\fA_{4}}{\overline{G}}\right)=5\]
		and if $1_{\omega}$ and $1_{\omega}^\ast$ denote the linear characters of $\fA_{4}$ afforded by $k_{\omega}$ and $k_{\omega}^{\ast}$, then it is easy to compute that 
		$$1_{\omega}\indhg{\fA_{4}}{\overline{G}}\,=1_{\omega}^{\ast}\indhg{\fA_{4}}{\overline{G}}=\St\,.$$
		It follows that $k_{\omega}\indhg{\fA_{4}}{\overline{G}}$ and its $k$-dual $k_{\omega}^{\ast}\indhg{\fA_{4}}{\overline{G}}$ are indecomposable and both afford the Steinberg character~$\St$.
		Then, we read from the $2$-decomposition matrix that  $k_{\omega}\indhg{\fA_{4}}{\overline{G}}$ has three composition factors, namely $k, \overline{S}_{+}, \overline{S}_{-}$ and because
		$$\Hom_{k\overline{G}}(k_{\omega}\indhg{\fA_{4}}{\overline{G}},k)=\langle \St ,1_{\overline{G}}\rangle_{\overline{G}}=0= \langle 1_{\overline{G}}, \St \rangle_{\overline{G}}=\Hom_{k\overline{G}}(k,k_{\omega}\indhg{\fA_{4}}{\overline{G}})$$
		the trivial module is neither in the head nor in the socle of $k_{\omega}\indhg{\fA_{4}}{\overline{G}}$. Thus, by duality, both $k_{\omega}\indhg{\fA_{4}}{\overline{G}}$ and  $k_{\omega}^{\ast}\indhg{\fA_{4}}{\overline{G}}$ are uniserial of length three with Loewy series
		$$
		\boxed{
			\begin{smallmatrix}
				\overline{S}_{+}\\
				k\\
				\overline{S}_{-}
			\end{smallmatrix}
		}
		\qquad\text{and}\qquad
		\boxed{
			\begin{smallmatrix}
				\overline{S}_{-}\\
				k\\
				\overline{S}_{+}
			\end{smallmatrix}
		}\,.
		$$
		\par
		It remains to compute the unique trivial source module with vertex $C_{2}$, which is the Scott module $\Sc(\overline{G},C_{2})$. Setting $\overline{B}:=B/Z$, we have $2\,||\,|\overline{B}|=|B|/2=10$ and so Proposition~\ref{prop:omnts}(e) yields
		$$\Sc(\overline{G},C_{2})\cong \Sc(\overline{G},\overline{B}),$$
		which is easier to compute. Indeed, by definition, $\Sc(\overline{G},\overline{B})\mid k\indhg{\overline{B}}{\overline{G}}$ and this induced module affords the character 
		$$1_{\overline{B}}\indhg{\overline{B}}{\overline{G}} \,=1_{\overline{G}}+\St$$
		since by~\cite[\S3.2.3]{BonBook} we have $1_{B}\indhg{B}{G}=R(1)=1_{G}+\St$. Now, on the one hand a vertex of $\Sc(\overline{G},C_{2})$ is~$C_{2}$,  so  $2\mid\dim_{k}(\Sc(\overline{G},C_{2}))$ and on the other hand both $1_{\overline{G}}$ and $\St$  have odd degree, therefore we obtain that 
		$$\Sc(\overline{G},C_{2})\cong \Sc(\overline{G},\overline{B})=k\indhg{\overline{B}}{\overline{G}}$$
		is indecomposable and the claim follows. 
	\end{proof}
	\bigskip

	\vspace{2mm}
	
	\subsection{The trivial source character table}

	\begin{thm}
		Assume  $\overline{G}=\PSL_{2}(q)$ where  $q\equiv \pm3\pmod{8}$. 
		Then the trivial source character table $\Triv_{2}(\overline{G})= [T_{i,v}]_{1\leq i,v\leq 3}$ is given as follows:
		\begin{enumerate}[{\,\,\rm(a)}] \setlength{\itemsep}{2pt}
			\item for $q\equiv \pm3\pmod{8}$, $T_{1,2}=T_{1,3}=T_{2,3}=\mathbf{0}$;
			\item if $q\equiv 3\pmod{8}$, then the matrices $T_{i,v}$ with $1\leq v\leq i\leq 3$ are as given in Table~\ref{tab:Tiapsl2q=3mod8}; and 
			{\footnotesize
				\renewcommand*\arraystretch{1.5}
				\begin{table}[!htpb]
					\caption{$T_{i,v}$  with $i\geq v$ when $q\equiv 3\pmod{8}$}\vspace{-5mm}
					\label{tab:Tiapsl2q=3mod8}
					\[
					\begin{array}{r|c||c|c|c|c|}
						\cline{2-6}	
						&
						&  I_{2}Z
						&	\begin{tabular}{c}
							${\bf d}(a)Z $\\
							\scriptsize{$a\in \Gamma_{2'}$}
						\end{tabular}
						&     \begin{tabular}{c}
							${\bf d}'(\xi)Z$\\
							\scriptsize{$\xi\in\Gamma'_{2'}$}
						\end{tabular}
						&       \begin{tabular}{c}
							$u_{\tau}Z$\\
							\scriptsize{$\tau\in\{\pm\}$}
						\end{tabular}
						\\ \hline \hline 	
						&     1_{\overline{G}}+\St 
						&	q+1
						&      2
						&      0
						&      1
						\\  \cline{2-6}		
						&      \St+R_{+}'(\theta_{0})
						&	\frac{3q-1}{2}
						&      1
						&      -2
						&        \frac{-1 +  \tau \sqrt {-q}}{2}
						\\  \cline{2-6}	    
						&  \St+R_{-}'(\theta_{0})
						&	\frac{3q-1}{2}
						&      1
						&      -2
						&        \frac{-1 -  \tau \sqrt {-q}}{2}
						\\  \cline{2-6}	
						T_{1,1}
						&    \begin{tabular}{c}
							$R'(\theta)+R'(\theta\eta^{2})$\\
							\scriptsize{$(\theta \in [T_{2'}'^\wedge / \equiv] \setminus \{1\}$,} \\
							\scriptsize{$S_{2}'^{\wedge}=\langle \eta \rangle)$}\\
						\end{tabular}
						&	2(q-1)
						&      0
						&      
						2(-\theta(\xi)-\theta(\xi)^{-1})
						&      -2
						\\  \cline{2-6}	
						&   \begin{tabular}{c}
							$R(\alpha)$\\
							\scriptsize{$(\alpha \in [T_{2'}^\wedge / \equiv] \setminus \{1\})$} 
						\end{tabular}
						&	q+1
						&    \alpha(a)+\alpha(a)^{-1}  
						&      0
						&      1
						\\ \hline \hline 	
						& 1_{\overline{G}}+\St+R_+'(\theta_0)+R_-'(\theta_0)
						&	2q
						&      2
						&      -2
						&       0
						\\ \cline{2-6}	
						T_{2,1}
						&     \begin{tabular}{c}
							$R'(\theta\eta^{2})$\\
							\scriptsize{$(\theta \in [T_{2'}'^\wedge / \equiv] \setminus \{1\}$, } \\
							\scriptsize{$S_{2}'^{\wedge}=\langle \eta \rangle)$}
						\end{tabular}
						&	q-1
						&      0
						&      -\theta(\xi)-\theta(\xi)^{-1}
						&      -1 
						\\ \hline \hline 	 
						&   1_{\overline{G}}
						&	1
						&      1
						&      1
						&      1
						\\ \cline{2-6}	
						T_{3,1}
						&  R_{+}'(\theta_{0})
						&  \frac{q-1}{2}
						&   0
						&   -1
						&  \frac{-1 +  \tau \sqrt {-q}}{2}
						\\  \cline{2-6}	
						& R_{-}'(\theta_{0})
						&   \frac{q-1}{2}
						&    0
						&  -1
						&   \frac{-1 - \tau \sqrt {-q}}{2}
						\\
						\hline\hline	
					\end{array}
					\]
					%
					
					\[
					\begin{array}{r|c||c|c|}
						\cline{2-4}	
						&    
						&  1
						&     \begin{tabular}{c}
							$\overline{{\bf d}'(\xi)Z}$\\
							\scriptsize{$\xi\in \Gamma'_{2'}$}
						\end{tabular}   
						\\ \hline \hline 	
						& 1_{\overline{G}}+\St+R_+'(\theta_0)+R_-'(\theta_0)
						&	2   
						&       2     
						\\ \cline{2-4}	
						\begin{tabular}{c}
							$T_{2,2}\!\!\!$\\
							\small{\phantom{X}} \\
						\end{tabular}
						&     \begin{tabular}{c}
							$R'(\theta\eta^{2})$\\
							\scriptsize{$(\theta \in [T_{2'}'^\wedge / \equiv] \setminus \{1\}$,}\\
							\scriptsize{$S_{2}'^{\wedge}=\langle \eta \rangle)$}
						\end{tabular}
						&	2
						&       \theta(\xi)+\theta(\xi)^{-1} 
						\\ \hline \hline 	 
						&   1_{\overline{G}}
						&	1
						&      1    
						\\ \cline{2-4}	
						T_{3,2}
						&  R_{+}'(\theta_{0})
						&  1
						&   1
						\\  \cline{2-4}	
						& R_{-}'(\theta_{0})
						&  1
						&  1 
						\\
						\hline\hline			
					\end{array}
					\quad\qquad
					\begin{array}{r|l||c|c|c|}
						\multicolumn{5}{c}{ } \\	
						\multicolumn{5}{c}{ } \\	
						\multicolumn{5}{c}{ } \\	
						\multicolumn{5}{c}{ } \\	
						\multicolumn{5}{c}{ } \\		
						\cline{2-5}	
						&
						&  1
						&	x
						&      x^{2}
						\\ \hline \hline 	
						&    1_{\overline{G}}
						&	1
						&      1
						&      1
						\\  \cline{2-5}	
						T_{3,3}	
						&  R_{+}'(\theta_{0})
						&  1
						&   \omega
						&   \omega^{2}
						\\  \cline{2-5}		
						&  R_{-}'(\theta_{0})
						& 1
						&   \omega^{2}
						&  \omega
						\\	 
						\hline\hline
						\multicolumn{5}{c}{\omega:=\mbox{\textup{primitive 3rd root of unity}}} 		  
					\end{array}
					\]
				\end{table}
			}	 
			
			\item if $q\equiv -3\pmod{8}$, then the matrices $T_{i,v}$ with $1\leq v\leq i\leq 3$ are as given in Table~\ref{tab:Tiapsl2q=-3mod8}.
			{\footnotesize
				\renewcommand*\arraystretch{1.5}
				\begin{table}[!htbp]
					\caption{$T_{i,v}$  with $i\geq v$ when $q\equiv -3\pmod{8}$}\vspace{-5mm}
					\label{tab:Tiapsl2q=-3mod8}
					\[
					\begin{array}{r|c||c|c|c|c|}
						\cline{2-6}	
						&
						&   I_{2}Z
						&	\begin{tabular}{c}
							${\bf d}(a)Z $\\
							\scriptsize{$a\in\Gamma_{2'}$}
						\end{tabular}
						&     \begin{tabular}{c}
							${\bf d}'(\xi)Z$\\
							\scriptsize{$\xi\in\Gamma'_{2'}$}
						\end{tabular}
						&       \begin{tabular}{c}
							$u_{\tau}Z$\\
							\scriptsize{$\tau\in\{\pm\}$}
						\end{tabular}
						\\ \hline \hline 	
						&     1_{\overline{G}}+\St + R_+(\alpha_0)+R_-(\alpha_0)
						&	2(q+1)
						&      4
						&      0
						&      2
						\\  \cline{2-6}		
						&      \St + R_+(\alpha_0)
						&	\frac{3q+1}{2}
						&     2
						&     -1
						&       \frac{1 +  \tau \sqrt {q}}{2}
						\\  \cline{2-6}	    
						&  \St+R_-(\alpha_0)
						&	\frac{3q+1}{2}
						&    2
						&    -1
						&      \frac{1 -  \tau \sqrt {q}}{2}
						\\  \cline{2-6}	
						T_{1,1}
						&    \begin{tabular}{c}
							$R(\alpha)+R(\alpha\eta^{2})$\\
							\scriptsize{$(\alpha \in [T_{2'}^\wedge / \equiv] \setminus \{1\}$, $S_{2}^{\wedge}=\langle \eta \rangle)$ } \\
						\end{tabular}
						&	2(q+1)
						&   	2(\alpha(a)+\alpha(a)^{-1})
						&    0
						&      2
						\\  \cline{2-6}	
						&   \begin{tabular}{c}
							$R'(\theta)$\\
							\scriptsize{$(\theta \in [T_{2'}'^\wedge / \equiv] \setminus \{1\})$} \\
						\end{tabular}
						&	q-1
						&     0
						&           -\theta(\xi)-\theta(\xi)^{-1}
						&           -1
						\\ \hline \hline 	
						& 1_{\overline{G}}+\St
						&	q+1
						&      2
						&      0
						&      1 
						\\ \cline{2-6}	
						\begin{tabular}{c}
							$T_{2,1}\!\!\!$\\
							\scriptsize{\phantom{X}} \\
						\end{tabular}
						&     \begin{tabular}{c}
							$R(\alpha)$\\
							\scriptsize{$(\alpha \in [T_{2'}^\wedge / \equiv] \setminus \{1\})$}\\
						\end{tabular}
						&	q+1
						&      \alpha(a)+\alpha(a)^{-1}
						&    0
						&     1
						\\ \hline \hline 	 
						&   1_{\overline{G}}
						&	1
						&      1
						&      1
						&      1
						\\ \cline{2-6}	
						T_{3,1}
						&  \St
						&  q
						&   1
						&   -1
						&   0
						\\  \cline{2-6}	
						& \St
						&  q
						&    1
						&   -1
						&  0
						\\
						\hline\hline			  
					\end{array}
					\]
					%
					%
					
					\[
					\begin{array}{r|c||c|c|}
						\cline{2-4}	
						&    
						& 1
						&     \begin{tabular}{c}
							$\overline{{\bf d}(a)Z}$ \\
							\scriptsize{$a\in \Gamma(T_{2'})$}
						\end{tabular}   
						\\ \hline \hline 	
						& 1_{\overline{G}}+\St
						&	2   
						&       2   
						\\ \cline{2-4}	
						\begin{tabular}{c}
							$T_{2,2}\!\!\!$\\
							\scriptsize{\phantom{X}} \\
						\end{tabular}
						&     \begin{tabular}{c}
							$R(\alpha)$\\
							\scriptsize{$(\alpha \in [T_{2'}^\wedge / \equiv] \setminus \{1\})$}\\
						\end{tabular}
						&	2
						&        \alpha(a)+\alpha(a)^{-1}
						\\ \hline \hline 	 
						&   1_{\overline{G}}
						&	1
						&      1
						\\ \cline{2-4}	
						T_{3,2}
						&  \St
						&  1
						&   1
						\\  \cline{2-4}	
						& \St
						&  1
						&  1
						\\
						\hline\hline			
					\end{array}
					\qquad\qquad
					\begin{array}{r|l||c|c|c|}
						\multicolumn{5}{c}{ } 	\\
						\multicolumn{5}{c}{ } 	\\
						\multicolumn{5}{c}{ } 	\\
						\multicolumn{5}{c}{ } 	\\
						\cline{2-5}	
						&
						&  1
						&	x
						&      x^{2}
						\\ \hline \hline 	
						&    1_{\overline{G}}
						&	1
						&      1
						&      1
						\\  \cline{2-5}	
						T_{3,3}	
						&  \St
						&  1
						&   \omega
						&   \omega^{2}
						\\  \cline{2-5}		
						&  \St
						& 1
						&   \omega^{2}
						&  \omega
						\\	 
						\hline\hline	
						\multicolumn{5}{c}{\omega:=\mbox{\textup{primitive 3rd root of unity}}} 		   
					\end{array}
					\]
				\end{table}
			}

		\end{enumerate}
	\end{thm}
	
	\clearpage

	\begin{proof}
		To begin with, the fact that $T_{1,2}=T_{1,3}=T_{2,3}=\mathbf{0}$ is immediate from \smallskip Remark~\ref{rem:tsctbl}(c). 
		So we may assume that $1\leq v\leq i\leq 3$, in which cases we will split our analysis of the trivial source modules with vertex $Q_{i}$ $2$-block by $2$-block according to the block distribution in Table~\ref{tab:blocksPSL2q3mod8} and Table~\ref{tab:blocksPSL2q5mod8}.

		\begin{enumerate}[$\cdot$]
			
			\item \textbf{The matrix $T_{1,1}$}\,.
			By Remark~\ref{rem:tsctbl}(a) the matrix  $T_{1,1}$ consists of  the values of the ordinary characters of the PIMs of $k\overline{G}$ evaluated at the $2'$-conjugagy classes of $\overline{G}$.
			Now, the character of the three PIMs of $\bB_{0}(\overline{G})$ are given in Table~\ref{tab:psl23} and Table~\ref{tab:psl25}, namely 
			$$
			\begin{cases}
				1_{\overline{G}}+\St,\,\St+R_{+}'(\theta_{0}),\,\St+R_{-}'(\theta_{0} ) & \text{ if } q\equiv 3\pmod{8}\,, \\
				1_{\overline{G}}+\St+R_+(\alpha_0)+R_-(\alpha_0),\,\St+R_+(\alpha_0),\,\St+R_-(\alpha_0)   & \text{ if } q\equiv -3\pmod{8}.
			\end{cases}
			$$
			Next, when $q\equiv 3\pmod{8}$, the PIMs  in the blocks $A'_\theta$  $(\theta \in [T_{2'}'^\wedge / \equiv] \setminus \{1\})$ with a  defect group isomorphic to $C_{2}$  afford the characters  $R'(\theta)+R'(\theta\eta^{2})$   by Remark~\ref{rem:cyclicDefectC2C4}, whereas the PIMs in the blocks $A_{\alpha}$ $(\alpha\in [T_{2'}^\wedge / \equiv] \setminus \{1\})$ of defect zero afford the characters~$R(\alpha)$.\\ 
			When $q\equiv -3\pmod{8}$, the PIMs  in the blocks $A_{\alpha}$ $(\alpha\in [T_{2'}^\wedge / \equiv] \setminus \{1\})$ with a  defect group isomorphic to $C_{2}$  afford the characters $R(\alpha)+R(\alpha\eta^{2})$ 
			whereas the PIMs in the blocks $A'_\theta$  $(\theta \in [T_{2'}'^\wedge / \equiv] \setminus \{1\})$ of defect zero afford the characters $R'(\theta)$.\\
			In all cases the values at the $2'$-classes are read directly from the character table of~$\SL_{2}(q)$  \smallskip(Table~\ref{tab:CTBLsl2}). 
			\item  \textbf{The matrix $T_{2,1}$}\,.
			By Remark~\ref{rem:tsctbl}(d), the matrix $T_{2,1}$ consists of  the values of the ordinary characters of the trivial source $k\overline{G}$-modules with vertex $Q_{2}\cong C_{2}$ evaluated at the $2'$-conjugagy classes of~$\overline{G}$. Clearly, the blocks of defect zero cannot contain any module with vertex $C_{2}$. By Lemma~\ref{lem:B0A4} and Lemma~\ref{lem:B0A5} the principal block contains precisely one such module, namely  $\Sc(\overline{G},C_{2})$ affording the character
			$$
			\begin{cases}
				1_{\overline{G}}+\St+R'_+(\theta_0)+R'_-(\theta_0)    & \text{ if } q\equiv 3\pmod{8}\,, \\
				1_{\overline{G}}+\St  & \text{ if } q\equiv -3\pmod{8}.
			\end{cases}
			$$
			By Remark~\ref{rem:cyclicDefectC2C4}, Table~\ref{tab:blocksPSL2q3mod8} and Table~\ref{tab:blocksPSL2q5mod8}, the blocks with a defect group isomorphic to $C_{2}$ all contain precisely one trivial source $k\overline{G}$-module affording the characters
			$$
			\begin{cases}
				R'(\theta\eta^{2})\quad(\theta \in [T_{2'}'^\wedge / \equiv] \setminus \{1\},\, S_{2}^{\wedge}=\langle \eta \rangle)  & \text{ if } q\equiv 3\pmod{8}\,, \\
				R(\alpha)\quad(\alpha\in [T_{2'}^\wedge / \equiv] \setminus \{1\}) & \text{ if } q\equiv -3\pmod{8}.
			\end{cases}
			$$
			Again, in all cases the values at the $2'$-classes are read directly from the character table \smallskip of~$\SL_{2}(q)$. 
			\item  \textbf{The matrix $T_{3,1}$}\,.
			By Remark~\ref{rem:tsctbl}(d), the matrix $T_{3,1}$ consists of  the values of the ordinary characters of the trivial source $k\overline{G}$-modules with vertex $Q_{3}\cong C_{2}\times C_{2}$ evaluated at the $2'$-conjugagy classes of~$\overline{G}$. As $\bB_{0}(\overline{G})$ is the unique block with full defect, the characters of these modules are given in Table~\ref{tab:psl23} and Table~\ref{tab:psl25}, namely
			$$
			\begin{cases}
				1_{\overline{G}},\,R_{+}'(\theta_{0}),\,R_{-}'(\theta_{0})  & \text{ if } q\equiv 3\pmod{8}\,, \\
				1_{\overline{G}},\,\St,\,\St  & \text{ if } q\equiv -3\pmod{8}.
			\end{cases}
			$$
			Again, the values of these characters at the $2'$-classes are read directly from the character table \smallskip of~$\SL_{2}(q)$. 
			\item  \textbf{The matrix $T_{2,2}$}\,.
			By Convention~\ref{conv:tsctbl}(a), the matrix $T_{2,2}$ consists of the values of the species~$\tau_{Q_{2},s}^{\overline{G}}$, with $s$ running through  $[\overline{N}_{2}]_{2'}$, evaluated at the  trivial source modules $[M]\in\TS(\overline{G};Q_{2})$.  
			By Remark~\ref{rem:tsctbl}(f), $s=1$ yields
			$$\tau_{Q_{2},1}^{\overline{G}}([M])=\tau_{\{1\},1}^{Q_{2}/Q_{2}}\circ \Br_{Q_{2}}^{Q_{2}}\circ \Res^{\overline{G}}_{Q_{2}}([M])\,.$$
			Now, by Remark~\ref{rem:tsctbl}(d), $\tau_{\{1\},1}^{Q_{2}/Q_{2}}$ returns the $k$-dimension of $\Br_{Q_{2}}^{Q_{2}}\circ \Res^{\overline{G}}_{Q_{2}}(M)$, which is easily computed as follows. Because  $Q_{2}\cong C_{2}$ the indecomposable direct summands of $\Res^{\overline{G}}_{Q_{2}}(M)$ are either trivial or projective and it follows that  $\Br_{Q_{2}}^{Q_{2}}$ returns only the trivial summands of the latter module. By Lemma~\ref{lem:tscharacters}(a) the multiplicity of the trivial module as a direct summand of $\Res^{\overline{G}}_{Q_{2}}(M)$ is given by $\chi^{}_{\widehat{M}}(z)\,$ where $z$ is the generator of $Q_{2}$.  Therefore, since the modules $[M]\in\TS(\overline{G};Q_{2})$ afford the characters
			$$
			\begin{cases}
				1_{\overline{G}}+\St+R_+'(\theta_0)+R_-'(\theta_0),\,\,  R'(\theta\eta^{2}) &     \text{ if } q\equiv 3\pmod{8}\,,\\
				\smallskip \scriptsize{(\theta \in [T_{2'}'^\wedge / \equiv] \setminus \{1\},  S_{2}'^{\wedge}=\langle \eta \rangle)} & \\
				1_{\overline{G}}+\St,\,\, R(\alpha)\,\, \scriptsize{(\alpha \in [T_{2'}^\wedge / \equiv] \setminus \{1\})}   & \text{ if } q\equiv -3\pmod{8} 
			\end{cases}
			$$
			and $z$ is an element of type ${\bf d}'(\xi)Z\in T'/Z$ with $\xi$ of order~4 when $q\equiv 3\pmod{8}$, respectively an element of type ${\bf d}(a)Z\in T/Z$ with $a$ of order~4 when $q\equiv -3\pmod{8}$, 
			we read from the character table of $\SL_{2}(q)$  that 
			$$\chi^{}_{\widehat{M}}(z)=2$$
			in all cases. \par
			Next, we prove that if $M=\Sc(\overline{G},C_{2})=\Sc(\overline{G},Q_{2})$,  then  
			$$\tau_{Q_{2},s}^{\overline{G}}([M])=2\text{ for each }1\neq s\in [\overline{N}_{2}]_{2'}\,.$$
			By definition $\tau_{Q_{2},s}^{\overline{G}}([M])$ is given by  the Brauer character $\varphi^{}_{M[Q_{2}]}$ of $M[Q_{2}]$ evaluated at $s$. Moreover, by Remark~\ref{prop:omnts}(d), $M[Q_{2}]$ seen as a $kN_2$-module  is the $kN_2$-Green correspondent of $M$, which is again the Scott module with vertex $Q_{2}$, that is,
			$$M[Q_{2}]=\Sc(N_2,Q_{2})$$
			(see \cite[\S2]{Bro85}). 
			Thus it suffices to prove that the ordinary character $\chi^{}_{\widehat{M[Q_{2}]}}$ takes value $2$ at all the $2'$-conjugacy classes of $N_2$.
			Now, by Lemma~\ref{lem:omnibusl=2}(b),(c) for both congruences $q\equiv\pm 3\pmod{8}$, the normaliser $N_2=:\mathcal{D}_{4w}$ is a dihedral group of order $4w$ with $w$ odd.
			Clearly, Scott modules belong to the principal block because they have a trivial composition factor by definition, and $\bB_{0}(\mathcal{D}_{4w}) \sim_{SA} k[C_{2}\times C_{2}]$ by the main result of \cite{CEKL} as $\mathcal{D}_{4w}$ is $2$-soluble. Over $k[C_{2}\times C_{2}]$, if $R_{2}\leq C_{2}\times C_{2}$ of order~$2$, then  it is straightforward to compute  that
			$$U:=\Sc(C_{2}\times C_{2},R_{2})=k\indhg{R_{2}}{C_{2}\times C_{2}}$$ 
			and affords the ordinary character $\chi^{}_{\widehat{U}}=1_{C_{2}\times C_{2}}+1_{b}$ where $1_{b}\in\Irr(C_{2}\times C_{2})\setminus\{1_{C_{2}\times C_{2}}\}$.
			It follows then directly from the character table of $\mathcal{D}_{4w}$ (see e.g. \cite[\S18.3]{JamesLiebeck}) and  the above source algebra equivalence that $\Irr(\bB_{0}(\mathcal{D}_{4w}))=\Lin(\bB_{0}(\mathcal{D}_{4w}))$ and 
			$\chi^{}_{\widehat{M[Q_{2}]}}$ is the sum of two linear characters. The claim now follows from the fact that all the  linear characters of $\mathcal{D}_{4w}$ take value $1$ at all $2'$-conjugacy classes (see  \cite[\S18.3]{JamesLiebeck}). \par
			Finally, assume that  $[M]\in\TS(\overline{G};Q_{2})\setminus\{\Sc(G,C_{2})\}$ and  $s\neq 1$. Then $M$ belongs to a block~$\bB$ of the form $A'_{\theta}$ ($\theta\in [T_{2'}'^\wedge / \equiv] \setminus \{1\}$) if $q\equiv 3\pmod{8}$, resp. of the form $A_{\alpha}$ ($\alpha \in [T_{2'}^\wedge / \equiv] \setminus \{1\}$) if $q\equiv -3\pmod{8}$, with cyclic defect group $Q_{2}\cong C_{2}$. By definition of the species $\tau_{Q_{2},s}^{\overline{G}}$ and Remark~\ref{prop:omnts}(d) we have
			$$\tau_{Q_{2},s}^{\overline{G}}(M)=\chi^{}_{\widehat{M[Q_{2}]}}(\tilde{s})$$
			where $\tilde{s}$ is a pre-image of $s$ in $N_{2}$ and $M[Q_{2}]$ is seen as the  $kN_{2}$-Green correspondent of $M$. Now, the latter module must lie in the Brauer correspondent of~$\bB$ at the level of $N_{2}$ and the theory of blocks with cyclic defect groups yields that $M[Q_{2}]$ is the unique simple $kN_{2}$-module. Then, by \cite[Theorem~7.1.2, \S6.2.2. and \S6.2.3]{BonBook}  we have
			$$
			\begin{cases}
				\chi^{}_{\widehat{M[Q_{2}]}}=\Ind_{T'}^{N'}(\theta)&     \text{ if } q\equiv 3\pmod{8}\,,\\
				\chi^{}_{\widehat{M[Q_{2}]}}=\Ind_{T}^{N}(\alpha) & \text{ if } q\equiv -3\pmod{8}, 
			\end{cases}
			$$
			seen as characters of $\overline{G}$. 
			Thus, it follows from the character tables of $N'$, resp. $N$, (see \cite[Table~6.2 and Table~6.3]{BonBook}) that for $s= {\bf d}'(\xi)$ with $\xi\in \Gamma'_{2'}$, resp. $s={\bf d}(a)$ with $a\in \Gamma_{2'}$,
			$$
			\begin{cases}
				\tau_{Q_{2},s}^{\overline{G}}(M) = \theta(\xi)+\theta(\xi)^{-1}&     \text{ if } q\equiv 3\pmod{8}\,,\\
				\tau_{Q_{2},s}^{\overline{G}}(M)=\alpha(a)+\alpha(a)^{-1} & \text{ if } q\equiv -3\pmod{8},
			\end{cases}
			$$
			as \smallskip required.

			\smallskip

			\item  \textbf{The matrix $T_{3,2}$}\,.
			By Convention~\ref{conv:tsctbl}(a), the matrix $T_{3,2}$ consists of the values of the species~$\tau_{Q_{2},s}^{\overline{G}}$, with $s$ running through  $[\overline{N}_{2}]_{2'}$, evaluated at the three trivial source modules $[M]\in\TS(\overline{G};Q_{3})$.   
			As in the previous case, if $s=1$, then
			$$\tau_{Q_{2},1}^{\overline{G}}([M])= \dim_{k}\left( \Br_{Q_{2}}^{Q_{2}}\circ \Res^{\overline{G}}_{Q_{2}}(M) \right)=\chi^{}_{\widehat{M}}(z)$$
			where $z$ is the generator of $Q_{2}$, which is  of type ${\bf d}'(\xi)Z\in T'/Z$ with $\xi$ of order~4 when $q\equiv 3\pmod{8}$, respectively an element of type ${\bf d}(a)Z\in T/Z$ with $a$ of order~4 when $q\equiv -3\pmod{8}$. Since the three modules $[M]\in\TS(\overline{G};Q_{3})$ afford the characters
			$$
			\begin{cases}
				1_{\overline{G}},\,R_{+}'(\theta_{0}),\,R_{-}'(\theta_{0})  & \text{ if } q\equiv 3\pmod{8}\,, \\
				1_{\overline{G}},\,\St,\,\St  & \text{ if } q\equiv -3\pmod{8},
			\end{cases}
			$$ we read from the character table of $\SL_{2}(q)$ that 
			$$\chi^{}_{\widehat{M}}(z)=1$$
			in all cases, as required. \\
			Next, we claim that 
			$$\tau_{Q_{2},s}^{\overline{G}}([M])=1 \quad \forall\, [M]\in\TS(\overline{G};Q_{3}),\forall\, s\in [\overline{N}_{2}]_{2'}\,.$$
			First, notice that by Remark~\ref{rem:tsctbl}(e), the above argument yields
			$$\dim_{k}M[Q_{2}]=\tau_{Q_{2},1}^{\overline{G}}([M])=1\quad\forall\, [M]\in\TS(\overline{G};Q_{3})\,.$$
			Now,  $\overline{N}_{2}$ has a unique trivial source module with vertex $C_{2}$, namely the trivial module. Indeed, this follows from Proposition~\ref{prop:omnts}(d) as $\overline{N}_{2}\cong \mathcal{D}_{\frac{q+1}{2}}$ when $q\equiv 3\pmod{8}$, respectively $\overline{N}_{2}\cong \mathcal{D}_{\frac{q-1}{2}}$ when $q\equiv -3\pmod{8}$, so that in both cases the subgroups of order~$2$ are conjugate and  self-normalising. Therefore, we conclude that $M[Q_{2}]=k$ for every $[M]\in\TS(\overline{G};Q_{3})$. In all cases, by definition $\tau_{Q_{2},s}^{\overline{G}}([M])$ is equal to the Brauer character of the trivial $k\overline{N}_{2}$-module evaluated at $s$, hence equal to $1$, 
			proving the \medskip claim.
			\item  \textbf{The matrix $T_{3,3}$}\,.
			Because $\overline{N}_{3}\cong C_{3}$, by Remark~\ref{rem:tsctbl}(a) the matrix $T_{3,3}$ of $\Triv_{2}(\overline{G})$ is just  the ordinary character table of the cyclic group $C_{3}$. 
		\end{enumerate}
	\end{proof}

	
	
	\section{$\SL_2(q)$ with $q\equiv\pm3\pmod{8}$ at $\ell=2$}\label{SL2l=2}

	Throughout this section, we continue to assume  that  $\ell=2$ and we now consider the group $G=\SL_2(q)$ with $q\equiv\pm3\pmod{8}$.
	By Lemma~\ref{lem:omnibusl=2}(b) the Sylow $2$-subgroups of $G$ are quaternion groups~$\mathcal{Q}_{8}$.  The $2$-blocks of $G$ have defect groups isomorphic to  $\mathcal{Q}_{8}$, $C_{4}$ or $C_{2}$, and  $\bB_{0}(G)$ is the unique block with full defect. 
	As all non-trivial $2$-subgroups contain the centre $Z$ of $G$ and inflation of a trivial source module is again a trivial source module, we may reduce a large amount of our computations to the  results obtained for $\overline{G}=\PSL_{2}(q)$ in the previous section and it essentially remains to list the PIMs of $G$.

	\begin{nota}\label{nota:sl2ql=2}%
		In this section we adopt the following notation.  We let $R_{4}$ be the Sylow $2$-subgroup of $G$ defined in Lemma~\ref{lem:omnibusl=2}, that is,  
		$$R_{4}=
		\begin{cases}
			\langle S_2', \sigma' \rangle& \text{ if }q\equiv 3\pmod{8}, \\
			\langle S_2, \sigma \rangle &\text{ if }q\equiv -3\pmod{8}.
		\end{cases}
		$$
		All subgroups of order $4$ in~$G$ are conjugate, thus we may  fix the following  set of representatives for the conjugacy classes of $2$-subgroups of $G$:
		$$
		R_{4}\cong \mathcal{Q}_{8},\quad
		R_{3}:=
		\begin{cases}
			S_{2}'\cong C_{4}& \text{ if }q\equiv 3\pmod{8}\\
			S_{2}\cong C_{4} &\text{ if }q\equiv -3\pmod{8}
		\end{cases}, 
		\quad R_2:=Z\cong C_{2} \quad\text{ and } \quad R_1:=\{1\}\,.
		$$
		Then, the structure of the quotients $\overline{N}_{i}=N_{\overline{G}}(R_{i})/R_{i}$ $(1\leq i\leq 4)$ follows from Lemma~\ref{lem:omnibusl=2} and we choose the following sets of representatives of their $2'$-conjugacy classes:
		\begin{enumerate}[{\,\,\rm(i)}] \setlength{\itemsep}{2pt}
			\item  $\overline{N}_{1}=G$ and we set 
			$$[\overline{N}_{1}]_{2'}:=\{I_{2}\}\cup\{{\bf d}(a)\mid a\in\Gamma_{2'}\}\cup\{{\bf d}(\xi)\mid \xi\in\Gamma'_{2'}\}\cup\{ u_{\tau}\mid\tau\in\{\pm1\}\}\,;$$
			\item  $\overline{N}_{2}=G/Z=\overline{G}$ and we set 
			$$[\overline{N}_{2}]_{2'}:=\{I_{2}Z\}\cup\{{\bf d}(a)Z\mid a\in \Gamma_{2'}\}\cup\{{\bf d}(\xi)Z\mid \xi\in \Gamma'_{2'}\}\cup\{u_{\tau}Z\mid\tau\in\{\pm1\}\}\,;$$
			\item 
			$$\overline{N}_{3}=
			\begin{cases}
				N'/R_{3} & \text{ if }q\equiv 3\pmod{8},  \\
				N/R_{3} & \text{ if }q\equiv -3\pmod{8},
			\end{cases}
			$$
			and  we set  
			$$[\overline{N}_{3}]_{2'}:=
			\begin{cases}
				\{1\}\cup\{\overline{{\bf d}(\xi)}\in \overline{N}_{3} \mid \xi\in \Gamma'_{2'}\}      & \text{ if }q\equiv 3\pmod{8}, \\
				\{1\}\cup\{\overline{{\bf d}(a)} \in \overline{N}_{3}\mid a\in \Gamma_{2'}\}     & \text{ if }q\equiv -3\pmod{8},
			\end{cases}
			$$
			where $1$ denotes the trivial element of $\overline{N}_{3}$ and  the bar notation left cosets in the quotient~$\overline{N}_{3}$;  \smallskip
			\item  $\overline{N}_{4}\cong C_{3}=:\langle x\rangle$ and we set \smallskip $[\overline{N}_{4}]_{2'}=:\{1,x,x^{2}\}$.
		\end{enumerate}
		With this notation, $\Triv_{2}(G)=[T_{i,v}]_{1\leq i,v\leq 4}$ with $T_{i,v}=\big[ \tau_{R_{v},s}^{G}([M])\big]_{M \in \TS(G;R_{i}), s\in [\overline{N}_{v}]_{2'} }$\,.
	\end{nota}
	\bigskip

	\vspace{2mm}
	\subsection{The $2$-blocks}
	
	\begin{lem} {\label{lem:blocksqcong5}}
		When $q \equiv \pm3\pmod{8}$ the $2$-blocks of $G$, their defect groups, and their decomposition matrices or Brauer trees are as given in Table~\ref{tab:blocksqcong3} and Table~\ref{tab:blocksqcong5}.
	\end{lem}	
	
		\begin{proof}
		All the information in Table~\ref{tab:blocksqcong3} and Table~\ref{tab:blocksqcong5} can be found in \cite[Chapters~8~and~9]{BonBook}. 
	\end{proof}
	
	\begin{longtable}{c||c|c|c}
		\caption{The $2$-blocks of $\SL_2(q)$ when  $q \equiv 3\pmod{8}$.}
		\label{tab:blocksqcong3}
		\\  
		Block
		&	\begin{tabular}{c}
			Number of Blocks \\
			(Type)
		\end{tabular}
		& 	\begin{tabular}{c}
			Defect \\
			Groups
		\end{tabular}
		& 
		\begin{tabular}{c}
			Decomposition Matrix/ \\ Brauer Tree
			
		\end{tabular}
		\\ \hline \hline 
		
		$\bB_{0}(G) $ 
		& 	\begin{tabular}{c}
			$1$ \\
			\\
			(Principal) 
		\end{tabular}
		& 	\begin{tabular}{c}
			$
			\mathcal{Q}_{8}$ 
		\end{tabular}
		&  
		\begin{blockarray}{cccc}
			\vspace{-2.5ex} \\
			& {~} $k$ {~} & $S_+$ & $S_-$   \\ 
			\begin{block}{c(ccc)}
				$1_G$					&$1$&$0$&$0$\\
				$\St$					&$1$&$1$&$1$\\
				$R_+'(\theta_0)$		&$0$&$1$&$0$\\
				$R_-'(\theta_0)$        &$0$&$0$&$1$\\
				$R_+(\alpha_0)$		&$1$&$1$&$0$\\
				$R_-(\alpha_0)$        &$1$&$0$&$1$\\
				$R'(\theta_1)$			&$0$&$1$&$1$\\
			\end{block} \vspace{-1.4ex}\\
			\BAmulticolumn{4}{c}{\small{$(\theta_1 \in [S_2'^\wedge / \equiv]$ is of order $4)$}}  \\
		\end{blockarray}   \vspace{-1.4ex}
		\\ \hline 
		
		\begin{tabular}{c}
			\vspace{-1.2ex}\\
			$A_\alpha$ \\
			\\
			\small{$(\alpha \in [T_{2'}^\wedge / \equiv] \setminus \{1\})$} \\
			\\
		\end{tabular}	
		& \begin{tabular}{c}
			$\frac{1}{2}\left( \frac{q-1}{2} - 1\right)$ \\
			\\
			\small{(Nilpotent)} \\
		\end{tabular}
		& 	$
		C_2$	
		& 	\begin{tabular}{c}
			$$  \xymatrix@R=0.0000pt@C=30pt{	
				\\
				{\Circle}  \ar@{-}[r]^{S_\alpha} & {\Circle}    \\
				{^{R(\alpha)}}&{^{R(\alpha\eta)}}
			}$$\\
			\small{$(S_2^{\wedge} = \langle \eta \rangle)$}\\
			\vspace{-.8ex}
		\end{tabular}
		\\ \hline 
		\begin{tabular}{c}
			$A'_\theta$  \\ 
			\\
			\small{$(\theta \in [T_{2'}'^\wedge / \equiv] \setminus \{1\})$} \\
		\end{tabular}
		& 	\begin{tabular}{c}
			$\frac{1}{2}\left(\frac{q+1}{4} - 1\right)$ \\ 
			\\
			\small{(Nilpotent)} \\
		\end{tabular}
		& 	$
		C_4$	
		& 	\begin{tabular}{c}
			$$  \xymatrix@R=0.0000pt@C=30pt{	
				\\
				{\Circle}  \ar@{-}[r]^{S_\theta} & {\CIRCLE}    \\
				{^{R'(\theta)}}&{^{\chi_{\Lambda}}}
			} $$ \\
			
			\small{$(S_2'^{\wedge} = \langle \eta \rangle$, $\chi_{\Lambda} =  \sum\limits_{j = 1}^{3}  R'(\theta\eta^j))$}\\
			\vspace{-1.2ex}
		\end{tabular}
		\\ \hline 
		
	\end{longtable}

	\newpage
	\begin{longtable}{c||c|c|c}
		\caption{The $2$-blocks of $\SL_2(q)$ when $q \equiv -3\pmod{8}$.}	\vspace{-2mm}
		\label{tab:blocksqcong5}
		\\  
		Block
		&	\begin{tabular}{c}
			Number of Blocks \\
			(Type)
		\end{tabular}
		& 	\begin{tabular}{c}
			Defect \\
			Groups
		\end{tabular}
		& 
		\begin{tabular}{c}
			Decomposition Matrix/ \\Brauer Tree 
			
		\end{tabular}
		\\ \hline \hline 
		
		$\bB_{0}(G)$ 
		& \begin{tabular}{c}
			$1$ \\
			\\
			(Principal)
		\end{tabular}
		& \begin{tabular}{c}
			$
			\mathcal{Q}_{8}$ 
		\end{tabular}
		&  
		\begin{blockarray}{cccc}
			\vspace{-2.5ex} \\
			& {~} $k$ {~} & $S_+$ & $S_-$   \\ 
			\begin{block}{c(ccc)}
				$1_G$					&$1$&$0$&$0$\\
				$\St$					&$1$&$1$&$1$\\
				$R_+'(\theta_0)$		&$0$&$1$&$0$\\
				$R_-'(\theta_0)$        &$0$&$0$&$1$\\
				$R_+(\alpha_0)$		&$1$&$1$&$0$\\
				$R_-(\alpha_0)$        &$1$&$0$&$1$\\
				$R(\alpha_1)$			&$2$&$1$&$1$\\
			\end{block} \vspace{-1.4ex}
			\\
			\BAmulticolumn{4}{c}{\small{$(\alpha_1 \in [S_2^\wedge / \equiv]$ of order $4)$}} \\
		\end{blockarray}  \vspace{-1.4ex}
		\\ \hline 
		
		\begin{tabular}{c}
			\vspace{-2ex}\\
			$A_\alpha$\\
			\\
			\small{$(\alpha \in [T_{2'}^\wedge / \equiv] \setminus \{1\})$} \\
			\\
		\end{tabular}
		& 	\begin{tabular}{c}
			\\
			$\frac{1}{2}\left( \frac{q-1}{4} - 1\right)$ \\
			\\
			\small{(Nilpotent)} \\
			\\
		\end{tabular}
		& 	$
		C_4$
		&   	\begin{tabular}{c}
			$$  \xymatrix@R=0.0000pt@C=30pt{	
				\\
				{\Circle}  \ar@{-}[r]^{S_\alpha} & {\CIRCLE}    \\
				{^{R(\alpha)}}&{^{\chi_{\Lambda}}}
			}$$ \\
			\small{$(S_2^{\wedge} = \langle \eta \rangle$, $\chi_{\Lambda} =  \sum\limits_{j = 1}^{3}  R(\alpha\eta^j))$}\\
			\vspace{-1.8ex}
		\end{tabular}
		\vspace{-1.2ex}\\ \hline 
		
		\begin{tabular}{c}
			$A'_\theta$  \\ 
			\\
			\small{$(\theta \in [T_{2'}'^\wedge / \equiv] \setminus \{1\})$} \\
		\end{tabular}
		& 	\begin{tabular}{c}
			$\frac{1}{2}\left(\frac{q+1}{2} - 1\right)$ \\ 
			\\
			\small{(Nilpotent)} \\
		\end{tabular}
		& 	$
		C_2$		
		& 	\begin{tabular}{c}
			$$  \xymatrix@R=0.0000pt@C=30pt{	
				\\ 
				{\Circle}  \ar@{-}[r]^{S_\theta} & {\Circle}    \\
				{^{R'(\theta)}}&{^{R'(\theta\eta)}}
			} $$ \\
			\small{$(S_2'^{\wedge} = \langle \eta \rangle)$}\\
			\vspace{-1.8ex}
		\end{tabular}
		\\ \hline 	
	\end{longtable}

	\vspace{2mm}
	\subsection{The trivial source character table}{\ }
	\label{secSL2l=2}
	
	\noindent The trivial source character table $\Triv_{2}(G)= [T_{i,v}]_{1\leq i,v\leq 4}$ of~$G$ is now up to a large extent obtained via inflation from $\overline{G}$. For this reason, we write $T_{i,v}(G)$ for the matrix $T_{i,v}$ of  $\Triv_{2}(G)$ and $T_{i,v}(\overline{G})$ for the matrix $T_{i,v}$ of  $\Triv_{2}(\overline{G})$. 
	
	\enlargethispage{15mm}
	\begin{thm}\label{thm:sl2ql=2}
		Assume  $G=\SL_{2}(q)$ where  $q\equiv \pm3\pmod{8}$. Then the trivial source character table $\Triv_{2}(G)= [T_{i,v}]_{1\leq i,v\leq 4}$ is given as follows.
		\begin{enumerate}[{\,\,\rm(a)}] \setlength{\itemsep}{2pt}
			\item For $q\equiv \pm3\pmod{8}$ the following holds:
			\begin{itemize}
				\item[\rm(i)] $T_{i,v}=\mathbf{0}$ for every $1\leq i < v\leq 4$\,;
				\item[\rm(ii)] $T_{i,1}(G)=T_{i,2}(G)=T_{i-1,1}(\overline{G})$ for every  $2\leq i \leq 4$\,;
				\item[\rm(iii)]  $T_{3,3}(G)=T_{2,2}(\overline{G})$, $T_{4,3}(G)=T_{3,2}(\overline{G})$, $T_{4,4}(G)=T_{3,3}(\overline{G})$.
			\end{itemize}
			\item If $q\equiv 3\pmod{8}$, then the matrix $T_{1,1}$ is as given in Table~\ref{tab:Tiasl2q=3mod8}.
			\item If $q\equiv -3\pmod{8}$, then the matrix $T_{1,1}$ is as given in Table~\ref{tab:Tiasl2q=5mod8}.
		\end{enumerate}   
	\end{thm}
	
	{\footnotesize
		\renewcommand*\arraystretch{1.5}
		\begin{table}[!htpb]
			\caption{$T_{1,1}$  when $q\equiv 3\pmod{8}$}\vspace{-3mm}
			\label{tab:Tiasl2q=3mod8}
			\[
			\begin{array}{r|c||c|c|c|c|}
				\cline{2-6}	
				&
				&   I_{2}
				&	\begin{tabular}{c}
					${\bf d}(a) $\\
					\scriptsize{$a\in \Gamma_{2'}$}
				\end{tabular}
				&     \begin{tabular}{c}
					${\bf d}'(\xi)$\\
					\scriptsize{$\xi\in \Gamma'_{2'}$}
				\end{tabular}
				&       \begin{tabular}{c}
					$ u_{\tau}$\\
					\scriptsize{$\tau\in\{\pm\}$}
				\end{tabular}
				\\ \hline \hline 	
				&     1_{G}+\St+R_{+}(\alpha_{0})+R_{-}(\alpha_{0})
				&	2(q+1)
				&      4
				&      0
				&      2
				\\  \cline{2-6}		
				&       \St+R_{+}(\alpha_{0})+R'_{+}(\theta_{0})+R'(\theta_{1})
				&	3q-1
				&      2
				&      -4
				&        \tau \sqrt {-q}-1
				\\  \cline{2-6}	    
				&   \St+R_{-}(\alpha_{0})+R'_{-}(\theta_{0})+R'(\theta_{1})
				&	3q-1
				&      2
				&     -4
				&      -\tau \sqrt {-q}-1
				\\  \cline{2-6}	
				T_{1,1}
				&    \begin{tabular}{c}
					$\sum\limits_{j = 1}^{4}  R'(\theta\eta^j)$\\
					\scriptsize{$(\theta \in [T_{2'}'^\wedge / \equiv] \setminus \{1\}$, $S_{2}'^{\wedge}=\langle \eta \rangle)$} 
				\end{tabular}
				&	4(q-1)
				&      0
				&      4(-\theta(\xi)-\theta(\xi)^{-1})	
				&      -4
				\\  \cline{2-6}	
				&   \begin{tabular}{c}
					$R(\alpha)+R(\alpha\eta)$\\
					\scriptsize{$(\alpha \in [T_{2'}^\wedge / \equiv] \setminus \{1\}$, $S_{2}^{\wedge}=\langle \eta \rangle)$} 
				\end{tabular}
				&	2(q+1)
				&    2(\alpha(a)+\alpha(a)^{-1})
				&      0
				&      2
				\\ \hline\hline	
			\end{array}
			\]

		\end{table}
	}	 
	
	{\footnotesize
		\renewcommand*\arraystretch{1.5}
		\begin{table}[!htbp]
			\caption{$T_{1,1}$  when $q\equiv -3\pmod{8}$}\vspace{-3mm}
			\label{tab:Tiasl2q=5mod8}
			\[
			\begin{array}{r|c||c|c|c|c|}
				\cline{2-6}	
				&
				&    I_{2}
				&	\begin{tabular}{c}
					${\bf d}(a) $\\
					\scriptsize{$a\in \Gamma_{2'}$}
				\end{tabular}
				&     \begin{tabular}{c}
					${\bf d}'(\xi)$\\
					\scriptsize{$\xi\in \Gamma'_{2'}$}
				\end{tabular}
				&       \begin{tabular}{c}
					$u_{\tau}$\\
					\scriptsize{$\tau\in\{\pm\}$}
				\end{tabular}
				\\ \hline \hline 	
				&     1_{G}+\St+R_{+}(\alpha_{0})+R_{-}(\alpha_{0})+2R(\alpha_{1})
				&	4(q+1)
				&      8
				&      0
				&      4
				\\  \cline{2-6}		
				&      \St+R_{+}(\alpha_{0})+R'_{+}(\theta_{0})+R(\alpha_{1})
				&     3q+1
				&     4
				&     -2
				&    1+ \tau \sqrt {q}
				\\  \cline{2-6}	    
				&   \St+R_{-}(\alpha_{0})+R'_{-}(\theta_{0})+R(\alpha_{1})
				&     3q+1
				&     4
				&    -2
				&     1- \tau \sqrt {q}
				\\  \cline{2-6}	
				T_{1,1}
				&    \begin{tabular}{c}
					$\sum\limits_{j = 1}^{4}  R(\alpha\eta^j)$\\
					\scriptsize{$(\alpha \in [T_{2'}^\wedge / \equiv] \setminus \{1\}$, $S_{2}^{\wedge}=\langle \eta \rangle)$} \\
				\end{tabular}
				&	4(q+1)
				&   	4(\alpha(a)+\alpha(a)^{-1})
				&      0
				&      4
				\\  \cline{2-6}	
				&   \begin{tabular}{c}
					$R'(\theta)+R'(\theta\eta)$\\
					\scriptsize{$(\theta \in [T_{2'}'^\wedge / \equiv] \setminus \{1\}$, $S_{2}'^{\wedge}=\langle \eta \rangle)$} \\
				\end{tabular}
				&	2(q-1)
				&     0
				&           2(-\theta(\xi)-\theta(\xi)^{-1})
				&           -2
				\\ \hline \hline 	
			\end{array}
			\]
			
		\end{table}
	}

	\begin{proof}{\ }
		\begin{enumerate}[{\rm(a)}]
			\item
			Again, the fact that $T_{i,v}=\mathbf{0}$ for every $1\leq i < v\leq 4$ is immediate from \smallskip Remark~\ref{rem:tsctbl}(c). This \smallskip proves~{\rm(i)}.\\ 
			Next, we notice that the subgroups $R_{2},R_{3}, R_{4}$ all contain the centre $Z=R_{2}$. The following assertions follow immediately.
			\begin{enumerate}[\,\,1.]  
				\item  For every $2\leq i\leq 4$ and every $M\in\TS(G,R_{i})$ we have $M[R_{2}]=M$ (as $kN_{2}$-modules). Therefore, as our choices of $[\overline{N}_{1}]_{2'}$ and $[\overline{N}_{2}]_{2'}$ agree modulo $Z$, by definition of the species we have  $T_{i,1}(G)=T_{i,2}(G)$ for every  \smallskip $2\leq i \leq 4$\,.
				\item For every $2\leq i\leq 4$, we have $R_{i}/Z=Q_{i-1}$ (where $Q_{i-1}$ is as defined in Notation~\ref{nota:psl2ql=2}), and so any trivial source module in $\TS(G,R_{i})$ is the inflation from $\overline{G}=G/Z$ to~$G$ of a trivial source $k\overline{G}$-module with vertex $R_{i}/Z$,  i.e.
				$$\TS(G,R_{i})=\{\Inf_{\overline{G}}^{G}(M)\mid M\in \TS(\overline{G},Q_{i-1})\}$$
				and the corresponding characters are 
				$$\chi^{}_{\widehat{\Inf_{\overline{G}}^{G}(M)}}= \Inf_{\overline{G}}^{G}(\chi^{}_{\widehat{M}})\,.$$
				It follows that $T_{i,2}(G)=T_{i-1,1}(\overline{G})$ for every  $2\leq i \leq 4$ and  $T_{3,3}(G)=T_{2,2}(\overline{G})$, $T_{4,3}(G)=T_{3,2}(\overline{G})$, $T_{4,4}(G)=T_{3,3}(\overline{G})$ because our choices of the representatives of the $2'$-conjugacy classes agree modulo $Z$, proving {\rm(ii)} and {\rm(iii)}.
			\end{enumerate}

			\item / {\rm(c)} The PIMs of $\bB_{0}(G)$ are the projective covers of the three simple $\bB_{0}(G)$-modules $k,S_{+},S_{-}$ and their characters can be read from the $2$-decomposition matrix in Table~\ref{tab:blocksqcong3} and Table~\ref{tab:blocksqcong5}. The PIMs of the blocks of type $A_{\alpha}$, $A'_{\theta}$ are read from their Brauer trees in Table~\ref{tab:blocksqcong3} and Table~\ref{tab:blocksqcong5} using Remark~\ref{rem:tsmodulecyclicdef}(1). We summarize this information in Table~\ref{tab:pimsSLl=2q3mod8} and Table~\ref{tab:pimsSLl=2q5mod8} below.
			\renewcommand*\arraystretch{1.5}
			\begin{table}[!htpb]
				\caption{PIMs of $kG$ when $q\equiv3\pmod{8}$}\label{tab:sl23}\vspace{-5mm}
				\label{tab:pimsSLl=2q3mod8}
				\[\begin{array}{|c|c|c|}
					\hline
					\text{Block} 
					&	\text{Module $M$}
					&    \text{Character $\chi_{\widehat{M}}$ }
					\\ \hline \hline 
					\bB_{0}(G)
					& P_{k}
					& 1_{G}+\St+R_{+}(\alpha_{0})+R_{-}(\alpha_{0})
					\\ 
					\bB_{0}(G)
					& P_{S_{+}}
					& \St+R_{+}(\alpha_{0})+R'_{+}(\theta_{0})+R'(\theta_{1})
					\\ 
					\bB_{0}(G)
					& P_{S_{-}}
					& \St+R_{-}(\alpha_{0})+R'_{-}(\theta_{0})+R'(\theta_{1})
					\\ \hline 
					\begin{tabular}{c}
						$A'_{\theta}$\\
						\scriptsize{$(\theta \in [T_{2'}'^\wedge / \equiv] \setminus \{1\})$} 
					\end{tabular}
					& 	 \boxed{
						\begin{smallmatrix}
							S_{\theta}\\
							S_{\theta}\\
							S_{\theta} \\
							S_{\theta}
						\end{smallmatrix}
					}
					&   \sum\limits_{j = 1}^{4}  R'(\theta\eta^j)
					\\ \hline
					\begin{tabular}{c}
						$A_{\alpha}$\\
						\scriptsize{$(\alpha \in [T_{2'}^\wedge / \equiv] \setminus \{1\})$} 
					\end{tabular}
					& 	 \boxed{
						\begin{smallmatrix}
							S_{\alpha}\\
							S_{\alpha}
						\end{smallmatrix}
					}
					&  R(\alpha)+R(\alpha\eta)
					\\ \hline
				\end{array}\]
			\end{table}
			\renewcommand*\arraystretch{1.5}
			\begin{table}[!htpb]
				\caption{PIMs of $kG$ when $q\equiv-3\pmod{8}$}\label{tab:sl25}\vspace{-6mm}
				\label{tab:pimsSLl=2q5mod8}
				\[\begin{array}{|c|c|c|}
					\hline
					\text{Block} 
					&	\text{Module }M
					&    \text{Character } \chi_{\widehat{M}} 
					\\ \hline \hline 
					\bB_{0}(G)
					& P_{k}
					& 1_{G}+\St+R_{+}(\alpha_{0})+R_{-}(\alpha_{0})+2R(\alpha_{1})
					\\ 
					\bB_{0}(G)
					& P_{S_{+}}
					& \St+R_{+}(\alpha_{0})+R'_{+}(\theta_{0})+R(\alpha_{1})
					\\ 
					\bB_{0}(G)
					& P_{S_{-}}
					& \St+R_{-}(\alpha_{0})+R'_{-}(\theta_{0})+R(\alpha_{1})
					\\ \hline
					\begin{tabular}{c}
						$A_{\alpha}$\\
						\scriptsize{$(\alpha \in [T_{2'}^\wedge / \equiv] \setminus \{1\})$} 
					\end{tabular}
					& 	 \boxed{
						\begin{smallmatrix}
							S_{\alpha}\\
							S_{\alpha}\\
							S_{\alpha} \\
							S_{\alpha}
						\end{smallmatrix}
					}
					&  \sum\limits_{j = 1}^{4}  R(\alpha\eta^j)
					\\ \hline
					
					\begin{tabular}{c}
						$A'_{\theta}$\\
						\scriptsize{$(\theta \in [T_{2'}'^\wedge / \equiv] \setminus \{1\})$} 
					\end{tabular}
					& 	 \boxed{
						\begin{smallmatrix}
							S_{\theta}\\
							S_{\theta}
						\end{smallmatrix}
					}
					& R'(\theta)+R'(\theta\eta)
					\\ \hline
				\end{array}\]
			\end{table}
			
			\noindent Evaluating these characters at the $2'$-conjugacy classes of $G$ yields $T_{1,1}$, as required.
		\end{enumerate}
	\end{proof}
	
}

	\clearpage

	\noindent\textbf{Acknowledgments.} 
	Much of the work for this project was done at the Mathematisches Forschungsinstitut Oberwolfach in March 2021 supported through the program ``Research in Pairs'' and Oberwolfach Research Fellowships. The authors thank the MFO for their generous support and hospitality. Part of the results in the case $\ell=2$ rely on methods and earlier computations  of trivial source modules and trivial source character tables  for  blocks with a Klein-four defect group and dihedral groups of order $4w$ for $w$ an odd number realised by the first author in his doctoral thesis. 
	The authors also wish to thank Gunter Malle and Robert Boltje for useful comments on a preliminary version of this manuscript. 
	\bigskip\bigskip
	\bigskip\bigskip


	\nocite{}
	\bibliographystyle{amsalpha}
	\bibliography{biblio.bib}
	\bigskip
	


\end{document}